\documentclass[11pt]{amsart}
\usepackage{amsmath,amssymb}
\usepackage{graphicx}
\usepackage{psfrag} 
\usepackage{hyperref}
\usepackage{tikz,epsfig,color,graphicx,epstopdf,subfig}
\usetikzlibrary{positioning,chains,fit,shapes,calc, fit}

\title[Proof of the $1$-factorization \& Hamilton decomposition conjectures II]{Proof of the $1$-factorization and Hamilton decomposition conjectures II: the bipartite case}
\author{B\'ela Csaba, Daniela K\"uhn, Allan Lo, Deryk Osthus and Andrew~Treglown}
\thanks {The research leading to these results was partially supported by the  European Research Council
under the European Union's Seventh Framework Programme (FP/2007--2013) / ERC Grant
Agreement no. 258345 (B.~Csaba, D.~K\"uhn and A.~Lo), 306349 (D.~Osthus) and 259385 (A.~Treglown).
The research was also partially supported by the EPSRC, grant no. EP/J008087/1 (D.~K\"uhn and D.~Osthus).
}
\date{\today} 
\def\COMMENT#1{}
\def\TASK#1{}

\begin{document}
\numberwithin{equation}{section}\textbf{•}

\def\noproof{{\unskip\nobreak\hfill\penalty50\hskip2em\hbox{}\nobreak\hfill%
        $\square$\parfillskip=0pt\finalhyphendemerits=0\par}\goodbreak}
\def\endproof{\noproof\bigskip}
\newdimen\margin   % needed for macros \textdisplay & \ltextdisplay
\def\textno#1&#2\par{%
    \margin=\hsize
    \advance\margin by -4\parindent
           \setbox1=\hbox{\sl#1}%
    \ifdim\wd1 < \margin
       $$\box1\eqno#2$$%
    \else
       \bigbreak
       \hbox to \hsize{\indent$\vcenter{\advance\hsize by -3\parindent
       \sl\noindent#1}\hfil#2$}%
       \bigbreak
    \fi}
\def\proof{\removelastskip\penalty55\medskip\noindent{\bf Proof. }}

\def\cC{\mathcal{C}}
\def\cV{\mathcal{V}}
\def\C{\mathcal{C}}
\def\cM{\mathcal{M}}
\def\cP{\mathcal{P}}
\def\eps{\varepsilon}
\def\ex{\mathbb{E}}
\def\prob{\mathbb{P}}
\def\pr{\mathbb{P}}
\def\eul{{\rm e}}
\def\i{(i_1,i_2,i_3,i_4)}
\def\I{i_1,i_2,i_3,i_4}
\def\cJ{\mathcal{J}}
\def\epszero{\eps_0}
\newcommand{\cB}{\mathcal{B}}

\newtheorem{firstthm}{Proposition}[section]
\newtheorem{thm}[firstthm]{Theorem}
\newtheorem{prop}[firstthm]{Proposition}
\newtheorem{fact}[firstthm]{Fact}
\newtheorem{lemma}[firstthm]{Lemma}
\newtheorem{cor}[firstthm]{Corollary}
\newtheorem{problem}[firstthm]{Problem}
\newtheorem{defin}[firstthm]{Definition}
\newtheorem{conj}[firstthm]{Conjecture}
\newtheorem{claim}[firstthm]{Claim}
\newtheorem{remark}[firstthm]{Remark}

\maketitle

\begin{abstract}
In a sequence of four papers, we prove the following results (via a unified approach) for all sufficiently large $n$:
\begin{itemize}
\item[(i)] [\emph{$1$-factorization conjecture}]
Suppose  that $n$ is even and $D\geq 2\lceil n/4\rceil -1$. 
Then every $D$-regular graph $G$ on $n$ vertices has a decomposition into perfect matchings.
Equivalently, $\chi'(G)=D$.

\item[(ii)] [\emph{Hamilton decomposition conjecture}]
Suppose that $D \ge   \lfloor n/2 \rfloor $.
Then every $D$-regular graph $G$ on $n$ vertices has a decomposition
into Hamilton cycles and at most one perfect matching.

\item[(iii)] [\emph{Optimal packings of Hamilton cycles}] Suppose that $G$ is a graph on $n$ vertices with
minimum degree $\delta\ge n/2$.
Then $G$ contains at least ${\rm reg}_{\rm even}(n,\delta)/2 \ge (n-2)/8$ edge-disjoint Hamilton cycles.
Here $\textnormal{reg}_{\textnormal{even}}(n,\delta)$ denotes the degree of the 
largest even-regular spanning subgraph one can guarantee in a graph on $n$ vertices
with minimum degree~$\delta$.
\end{itemize}
According to Dirac, (i) was first raised in the 1950s.
(ii) and the special case $\delta= \lceil n/2 \rceil$ of (iii) answer questions of Nash-Williams from 1970.
All of the above bounds are best possible. In the current paper, we prove the above results for the case when $G$ is close to a
complete balanced bipartite graph.
%A key tool is the recent result of K\"uhn and Osthus that every dense even-regular robustly expanding graph
%has a Hamilton decomposition.
\end{abstract}
\maketitle

\tableofcontents

\section{Introduction}
The topic of decomposing a graph into a given collection of edge-disjoint subgraphs has a
long history. Indeed, in 1892, Walecki~\cite{lucas} proved that every complete graph of odd order has a  decomposition into edge-disjoint Hamilton cycles.
In a sequence of four papers, we provide a unified approach towards proving three  long-standing graph decomposition conjectures for all sufficiently large graphs.
%Central to our approach are the methods established by K\"uhn and Osthus~\cite{Kelly} to prove a generalization of
%Kelly's conjecture that every regular tournament has a Hamilton decomposition
%(for large tournaments).

\subsection{The $1$-factorization conjecture}

Vizing's theorem states that for any graph~$G$ of maximum degree $\Delta$, its edge-chromatic number
$\chi'(G)$ is either $\Delta$ or $\Delta+1$. However, the problem of determining the precise value of $\chi '(G)$
for an arbitrary graph $G$ is NP-complete~\cite{NP}. 
Thus, it is of interest to determine classes of graphs $G$ that attain the (trivial) lower bound~$\Delta$
-- much of the recent book~\cite{stiebitz} is devoted to the subject.
If $G$ is a regular graph then $\chi'(G)=\Delta(G)$ precisely when $G$ has a $1$-factorization:
a \emph{$1$-factorization} of a graph~$G$ consists of a set of edge-disjoint perfect matchings covering all edges of~$G$.
The  $1$-factorization conjecture states that every regular graph of sufficiently high degree has a 
$1$-factorization. It was first stated explicitly by Chetwynd and Hilton~\cite{1factorization,CH} (who also proved partial results).
However, they state that
according to Dirac, it was already discussed in the 1950s. We prove the $1$-factorization conjecture for sufficiently large graphs.

\begin{thm}\label{1factthm}
There exists an $n_0 \in \mathbb N$ such that the following holds.
Let $ n,D \in \mathbb N$ be such that $n\geq n_0$ is even and $D\geq 2\lceil n/4\rceil -1$. 
Then every $D$-regular graph $G$ on $n$ vertices has a $1$-factorization.%
    \COMMENT{So this means that $D\ge n/2$ if $n = 2 \pmod 4$ and $D\ge n/2-1$ if $n = 0 \pmod 4$.}
    Equivalently, $\chi'(G)=D$.
\end{thm}
The bound on the minimum degree in Theorem~\ref{1factthm} is best possible.
In fact, a smaller degree bound does not even ensure a single perfect matching.%
\COMMENT{Deryk added new sentence and deleted overfull conjecture below}
 To see this, suppose first that $n = 2 \pmod{4}$.
Consider the graph which is the disjoint union of two cliques of order $n/2$ (which is odd). If $n = 0 \pmod 4$,
consider the graph obtained from the disjoint union of cliques of orders $n/2-1$ and $n/2+1$ (both odd)
by deleting a Hamilton cycle in the larger clique.

 Perkovic and Reed~\cite{reed} proved an approximate version of Theorem~\ref{1factthm}
(they assumed that $D \ge n/2+\eps n$). Recently, 
this was generalized by Vaughan~\cite{vaughan} to multigraphs of bounded multiplicity, thereby
proving an approximate version of a `multigraph  $1$-factorization conjecture' which was raised by
Plantholt and Tipnis~\cite{PT}.
Further related results and problems are discussed in the recent monograph~\cite{stiebitz}.

\subsection{The Hamilton decomposition conjecture}

A \emph{Hamilton decomposition} of a graph~$G$ consists of a set of edge-disjoint Hamilton cycles covering all the edges of~$G$.
A natural extension of this to regular graphs $G$ of odd degree is to ask for a decomposition into 
Hamilton cycles and one perfect matching (i.e.~one perfect matching $M$ in $G$ together with a 
Hamilton decomposition of $G-M$). 
Nash-Williams~\cite{initconj,decompconj} raised the problem of finding a Hamilton decomposition 
in an even-regular graph of sufficiently large degree.
The following result completely solves this problem for large graphs.
\begin{thm} \label{HCDthm} 
There exists an $n_0 \in \mathbb N$ such that the following holds.
Let $ n,D \in \mathbb N$ be such that $n \geq n_0$ and
$D \ge   \lfloor n/2 \rfloor $.
Then every $D$-regular graph $G$ on $n$ vertices has a decomposition into Hamilton cycles and 
at most one perfect matching.
\end{thm}
The bound on the degree in Theorem~\ref{HCDthm} is best possible (see Proposition~3.1 in~\cite{2clique} for a proof of this). Note that Theorem~\ref{HCDthm} does not quite imply Theorem~\ref{1factthm},
as the degree threshold in the former result is slightly higher.

Previous results include the following:
Nash-Williams~\cite{NWreg} showed that the degree bound in Theorem~\ref{HCDthm} ensures a single Hamilton cycle.
Jackson~\cite{Jackson79} showed that one can ensure close to $D/2-n/6$ edge-disjoint Hamilton cycles. More recently,
Christofides, K\"uhn and Osthus~\cite{CKO} obtained an approximate decomposition under the assumption that $D \ge n/2 +\eps n$. Finally,
under the same assumption, K\"uhn and Osthus~\cite{KellyII} obtained an exact decomposition
(as a consequence of the main result in~\cite{Kelly} on Hamilton decompositions of robustly expanding graphs). 

\subsection{Packing Hamilton cycles in graphs of large minimum degree}
Dirac's theorem is best possible in the sense that one cannot lower the minimum degree condition.
Remarkably though,
the conclusion can be strengthened considerably: Nash-Williams~\cite{Diracext} proved that every graph $G$ on $n$ vertices with minimum degree $\delta(G) \ge n/2$
contains $\lfloor 5n/224 \rfloor$ edge-disjoint Hamilton cycles.
Nash-Williams~\cite{Diracext,initconj,decompconj} raised the question of finding the best possible bound on the number of edge-disjoint Hamilton cycles in a Dirac graph. This question is answered by Corollary~\ref{NWmindegcor} below. 

In fact, we answer a more general form of this question: what is the number of edge-disjoint Hamilton cycles one can guarantee in a graph $G$ of minimum degree~$\delta$?
Let $\textnormal{reg}_{\textnormal{even}}(G)$
be the largest degree of an even-regular spanning subgraph of $G$. Then let
\[
\textnormal{reg}_{\textnormal{even}}(n,\delta):=\min\{\textnormal{reg}_{\textnormal{even}}(G):|G|=n,\ \delta(G)=\delta\}.
\]
Clearly, in general we cannot guarantee more than $\textnormal{reg}_{\textnormal{even}}(n,\delta)/2$
edge-disjoint Hamilton cycles in a graph of order $n$ and minimum
degree $\delta$. The next result shows that this bound is best possible (if $\delta < n/2$, then $\textnormal{reg}_{\textnormal{even}}(n,\delta)=0$).%
\COMMENT{Deryk changed this}
\begin{thm}\label{NWmindeg}
There exists an $n_0 \in \mathbb N$ such that the following holds. Suppose that $G$ is a graph on $n\ge n_0$ vertices with
minimum degree $\delta\ge n/2$. Then $G$ contains at least ${\rm reg}_{\rm even}(n,\delta)/2$ edge-disjoint Hamilton cycles.
\end{thm}

K\"uhn, Lapinskas and Osthus~\cite{KLOmindeg} proved Theorem~\ref{NWmindeg} in the case when $G$ is not close to one of the extremal graphs for Dirac's theorem.
An approximate version of Theorem~\ref{NWmindeg} for $\delta \ge n/2+\eps n$
was obtained earlier by Christofides, K\"uhn and Osthus~\cite{CKO}.
Hartke and Seacrest~\cite{HartkeHCs} gave a simpler argument with improved error bounds.

The following consequence of Theorem~\ref{NWmindeg} answers the original question of Nash-Williams. 
\begin{cor}\label{NWmindegcor}
There exists an $n_0 \in \mathbb N$ such that the following holds. Suppose that $G$ is a graph on $n\ge n_0$ vertices with
minimum degree $\delta\ge n/2$. Then $G$ contains at least $(n-2)/8$ edge-disjoint Hamilton cycles.
\end{cor}
See~\cite{2clique} for an explanation as to why Corollary~\ref{NWmindegcor} follows from Theorem~\ref{NWmindeg}
and for a construction showing the bound on the number of edge-disjoint Hamilton cycles in
Corollary~\ref{NWmindegcor} is best possible (the construction is also described in Section~\ref{sec:sketch1}).

\subsection{Overall structure of the argument}

For all three of our main results, we split the argument according to the structure of the graph $G$ under consideration:
\begin{enumerate}
\item[{\rm (i)}] $G$ is  close to the complete balanced bipartite graph $K_{n/2,n/2}$;
\item[{\rm (ii)}] $G$ is close to the union of two disjoint copies of a clique $K_{n/2}$;
\item[{\rm (iii)}] $G$ is a `robust expander'.
\end{enumerate}
Roughly speaking, $G$ is a robust expander if for every set $S$ of vertices, its neighbourhood is at least a little larger than $|S|$,
even if we delete a small proportion of the edges of $G$.
The main result of~\cite{Kelly} states that every dense regular robust expander
has a Hamilton decomposition.
This immediately implies Theorems~\ref{1factthm} and~\ref{HCDthm} in Case~(iii).
For Theorem~\ref{NWmindeg}, Case (iii) is proved in~\cite{KLOmindeg} using a more involved argument, but also based on
the main result of~\cite{Kelly}.

Case~(ii) is proved in~\cite{2clique, paper4}. The current paper is devoted to the proof of Case~(i).
In~\cite{2clique} we derive Theorems~\ref{1factthm},~\ref{HCDthm} and~\ref{NWmindeg} from the structural results covering Cases (i)--(iii). 

The arguments in the current paper for Case~(i) as well as those in~\cite{2clique} for Case~(ii)
make use of an `approximate' decomposition result  proved in~\cite{paper3}.
In both Case~(i) and Case~(ii) we use the main lemma from~\cite{Kelly} (the `robust decomposition lemma') when transforming this approximate
decomposition into an exact one.

%%%%%%%%%%%%%%%%%%%%%%%%%%%%%
\subsection{Statement of the main results of this paper}
As mentioned above, the focus of this paper is to prove Theorems~\ref{1factthm},~\ref{HCDthm} and~\ref{NWmindeg}
when our graph is \emph{close} to the complete balanced bipartite graph $K_{n/2,n/2}$. More precisely,
we say that a graph $G$ on $n$ vertices is \emph{$\eps$-bipartite} if there is
a partition $S_1, S_2$ of $V(G)$ which satisfies the following:
\begin{itemize}
\item $n/2-1<|S_1|, |S_2|<n/2+1$;
\item $e(S_1), e(S_2)\le \varepsilon n^2.$
\end{itemize}   
The following result  implies Theorems~\ref{1factthm} and \ref{HCDthm} in the case when
our given graph is close to $K_{n/2,n/2}$.
\begin{thm}\label{1factbip}
There are $\eps_{\rm ex} >0$ and $n_0\in\mathbb{N}$ such that the following holds. Suppose that $D \ge (1/2-\eps_{\rm ex})n$ and $D$ is even
and suppose that $G$ is a $D$-regular graph on $n \ge n_0$ vertices which is $\eps_{\rm ex}$-bipartite.
Then $G$ has a Hamilton decomposition.%
\COMMENT{We could probably prove the following stronger result: 
`For every $\alpha>1/4$, there is an $\eps_{\rm ex} >0$ and $n_0$ such that if $G$ is an $\alpha n$-regular graph $G$ on $n \ge n_0$ vertices which is $\eps_{\rm ex}$-close to being bipartite, 
then $G$ has a Hamilton decomposition.'
The main changes needed for this are probably in the approx cover section and the fact that we would need to construct the universal walk via shifted walks. 
So it probably does not seem worth it. Also, note that we get the $1$-factorization from this by removing a perfect matching if the $D$ degree is odd.
(Note that $D$ odd implies $n$ even). To find this perfect matching we need to apply either Dirac (if $D\ge n/2$) or Tutte (if $D$ is only
almost $n/2$).}
\end{thm}
The next result implies Theorem~\ref{NWmindeg} in the case when
our graph is close to $K_{n/2,n/2}$.
\begin{thm}\label{NWmindegbip}
For each $\alpha>0$ there are $\eps_{\rm ex} >0$ and $n_0\in\mathbb{N}$ such that the following holds. Suppose that
$F$ is an $\eps_{\rm ex}$-bipartite graph on $n \ge n_0$ vertices with $\delta(F)\ge (1/2-\eps_{\rm ex})n$.
Suppose that $F$ has a $D$-regular spanning subgraph $G$%
\COMMENT{later on we say e.g. `let $G$ be as in Theorem~\ref{NWmindegbip} so this is necessary here}  
 such that $n/100\le D\le (1/2-\alpha)n$ and $D$ is even. Then $F$ contains $D/2$ edge-disjoint Hamilton cycles.
\end{thm}
Note that Theorem~\ref{1factbip} implies that the degree bound in Theorems~\ref{1factthm} and~\ref{HCDthm} is not tight in the almost bipartite case
(indeed, the extremal graph is close to being the union of two cliques). On the other hand, the extremal construction for Corollary~\ref{NWmindegcor}
is close to bipartite (see Section~\ref{sec:sketch1} for a description). 
So it turns out that the bound on the number of edge-disjoint Hamilton cycles in Corollary~\ref{NWmindegcor}
is best possible in the almost bipartite case but not when the graph is close to the union of two cliques.%
\COMMENT{Deryk added this and we might want to add this and the comment about $delta<n/2$ into the 2clique paper. }

In Section~\ref{sec:sketch} we give an outline of the proofs of Theorems~\ref{1factbip} and~\ref{NWmindegbip}.
The results from Sections~\ref{eliminating} and~\ref{findBES} are used in both the proofs of Theorems~\ref{1factbip} and~\ref{NWmindegbip}.
In Sections~\ref{sec:spec} and~\ref{sec:robust} we build up machinery for the proof of
Theorem~\ref{1factbip}. We then prove Theorem~\ref{NWmindegbip} in Section~\ref{sec:proof1} and Theorem~\ref{1factbip} in Section~\ref{sec:proof2}.

\section{Notation and Tools}

\subsection{Notation}
Unless stated otherwise, all the graphs and digraphs considered in this paper are simple and do not contain loops. So in a digraph $G$, we allow up to
two edges between any two vertices; at most one edge in each direction.
Given a graph or digraph $G$, we write $V(G)$ for its vertex set, $E(G)$ for its edge set, $e(G):=|E(G)|$ for
the number of its edges and $|G|:=|V(G)|$ for the number of its vertices. 

Suppose that $G$ is an undirected graph. We write $\delta(G)$ for the minimum degree of $G$ and $\Delta(G)$ for its maximum degree.
Given a vertex $v$ of $G$ and a set $A\subseteq V(G)$,
we write $d_G(v,A)$ for the number of neighbours of $v$ in $G$ which lie in~$A$. Given $A,B\subseteq V(G)$,
we write $E_G(A)$ for the set of all those edges of $G$ which have both endvertices in $A$ and $E_G(A,B)$ for the set of
all those edges of $G$ which have one endvertex in $A$ and its other endvertex in $B$. We also call the edges in $E_G(A,B)$
\emph{$AB$-edges} of $G$. We let $e_G(A):=|E_G(A)|$ and $e_G(A,B):=|E_G(A,B)|$.
We denote by $G[A]$ the subgraph of $G$ with vertex set $A$ and edge set $E_G(A)$.
If $A\cap B=\emptyset$, we denote by $G[A,B]$ the bipartite subgraph of $G$
with vertex classes $A$ and $B$ and edge set $E_G(A,B)$. 
If $A=B$ we%
 \COMMENT{We do use this notation, e.g. when
considering $G[A_i,A_j]$ for all $i,j\leq K$.}
 define $G[A,B]:=G[A]$.%
	\COMMENT{Deryk deleted`Note that $2 E(G[A]) = E_G(G[A,A])$' as it seems incorrect}
We often omit the index $G$ if the graph $G$ is clear from the context.
A spanning subgraph $H$ of $G$ is an \emph{$r$-factor} of $G$ if every vertex has degree $r$
in $H$.

Given a vertex set $V$ and two multigraphs%
   \COMMENT{Have to define this for multigraphs instead of graph since if we take $G+H+F$ then $G+H$ might already be a multigraph.}
$G$ and $H$ with $V(G),V(H)\subseteq V$, we write $G+H$ for the multigraph whose vertex
set is $V(G)\cup V(H)$ and in which the multiplicity of $xy$ in $G+H$ is the sum of the multiplicities of $xy$ in $G$ and in~$H$
(for all $x,y\in V(G)\cup V(H)$).
We say that a graph $G$ has a \emph{decomposition} into $H_1,\dots,H_r$ if $G=H_1+\dots +H_r$ and the $H_i$ are pairwise
edge-disjoint.%
\COMMENT{Andy: added definition of decomposition}

If $G$ and $H$ are simple graphs, we write $G\cup H$ for the (simple) graph whose vertex set is
$V(G)\cup V(H)$ and whose edge set is $E(G)\cup E(H)$.%
   \COMMENT{I think this is a good convention. It has the advantage that if we talk about
$G+J^*$ then it is clear that we treat the edges in $J^*$ as being distinct from those in $G$. However, this means that we cannot
talk about $G+BEPS^*$ (we would need to write $G+{\rm Fict}(BEPS^*)$ instead). Have to make sure that this fits with the almost decomposition section/paper.}
Similarly, $G\cap H$ denotes the graph whose vertex set is $V(G)\cap V(H)$ and whose edge set is $E(G)\cap E(H)$.
We write $G-H$ for the subgraph of $G$ which is obtained from $G$
by deleting all the edges in $E(G)\cap E(H)$.%
    \COMMENT{So we don't require that $H\subseteq G$ when using this notation.} 
Given $A\subseteq V(G)$, we write $G-A$ for the graph obtained from $G$ by deleting all vertices in~$A$.

A \emph{path system} is a graph $Q$ which is the union of vertex-disjoint paths (some of them might be trivial).
We say that $P$ is a \emph{path in Q} if $P$ is a component of $Q$ and, abusing the notation, sometimes write $P\in Q$ for this.

If $G$ is a digraph, we write $xy$ for an edge directed from $x$ to $y$.
 A digraph $G$ is an \emph{oriented graph} if there are no $x,y\in V(G)$ such that $xy, yx\in E(G)$.
Unless stated otherwise, when we
refer to paths and cycles in digraphs, we mean directed paths and cycles, i.e.~the edges on these paths/cycles are oriented consistently.
If $x$ is a vertex of a digraph $G$, then $N^+_G(x)$ denotes the \emph{outneighbourhood} of $x$, i.e.~the
set of all those vertices $y$ for which $xy\in E(G)$. Similarly, $N^-_G(x)$ denotes the \emph{inneighbourhood} of $x$, i.e.~the
set of all those vertices $y$ for which $yx\in E(G)$. The \emph{outdegree} of $x$ is $d^+_G(x):=|N^+_G(x)|$ and the
\emph{indegree} of $x$ is $d^-_G(x):=|N^-_G(x)|$. 
%We define $\Delta (G)$ to be the maximum of the maximum indegree of $G$
%and the maximum outdegree of $G$.\COMMENT{Used in proof of Lemma~\ref{lma:EF-bipartite}.}
We write $\delta(G)$ and $\Delta(G)$ for the minimum and maximum degrees of the underlying simple%
   \COMMENT{Daniela: added simple}
undirected graph of $G$ respectively.%
\COMMENT{AL: added definitions of $\delta(G)$ and $\Delta(G)$. use this convention in e.g. the proof of Lemma~\ref{lma:EF-bipartite}}

For a digraph $G$,%
\COMMENT{AL:added `For a digraph $G$'}
 whenever $A,B\subseteq V(G)$ with $A\cap B=\emptyset$, we denote by $G[A,B]$ the bipartite subdigraph of $G$ with vertex classes $A$ and $B$ whose edges are all the edges of $G$ directed from $A$ to~$B$,
and let $e_G (A,B)$ denote the number of edges in $G[A,B]$. We define $\delta (G[A,B])$ to be the minimum
degree of the underlying undirected graph of $G[A,B]$ and define
$\Delta (G[A,B])$ to be the maximum
degree of the underlying undirected graph of $G[A,B]$.\COMMENT{we use this convention in e.g. the proof of Prop~\ref{lem:dirscheme}}
A spanning subdigraph $H$ of $G$ is an \emph{$r$-factor} of $G$ if the outdegree and the indegree of every vertex of $H$ is~$r$.

If $P$ is a path and $x,y\in V(P)$, we write $xPy$ for the subpath of $P$ whose endvertices are $x$ and $y$.
We define $xPy$ similarly if $P$ is a directed path and $x$ precedes $y$ on~$P$.%
   \COMMENT{Daniela added these 2 sentences}

In order to simplify the presentation, we omit floors and ceilings and treat large numbers as integers whenever this does
not affect the argument. The constants in the hierarchies used to state our results have to be chosen from right to left.
More precisely, if we claim that a result holds whenever $0<1/n\ll a\ll b\ll c\le 1$ (where $n$ is the order of the graph or digraph),
then this means that
there are non-decreasing functions $f:(0,1]\to (0,1]$, $g:(0,1]\to (0,1]$ and $h:(0,1]\to (0,1]$ such that the result holds
for all $0<a,b,c\le 1$ and all $n\in \mathbb{N}$ with $b\le f(c)$, $a\le g(b)$ and $1/n\le h(a)$. 
We will not calculate these functions explicitly. Hierarchies with more constants are defined in a similar way.
We will write $a = b \pm c$ as shorthand for $ b - c \le a \le b+c$.

\subsection{$\eps$-regularity} 
If $G=(A,B)$ is an undirected bipartite graph with vertex classes $A$ and $B$, then the
\emph{density} of $G$ is defined as
$$ d(A, B) := \frac{e_G(A,B)}{|A||B|}.$$
For any $\eps >0$, we say that $G$ is \emph{$\eps$-regular} if for any $A'\subseteq A$
and $B' \subseteq B$ with $|A'| \geq \eps |A|$ and $|B'| \geq \eps |B|$
we have $|d(A',B') - d(A, B)| < \eps$. We say that $G$ is \emph{$(\eps, \ge d)$-regular} if
it is $\eps$-regular and has density $d'$ for some $d' \ge d-\eps$.

We say that $G$ is \emph{$[\eps,d]$-superregular} if it is $\eps$-regular and $d_G(a)=(d\pm \eps)|B|$
for every $a\in A$ and $d_G(b)=(d\pm \eps)|A|$ for every $b\in B$. $G$ is \emph{$[\eps, \ge d]$-superregular} if
it is  \emph{$[\eps, d']$-superregular} for some $d' \ge d$.

Given disjoint vertex sets $X$ and $Y$ in a digraph $G$, recall that $G[X,Y]$ denotes
the bipartite subdigraph of $G$ whose vertex classes are $X$ and $Y$ and whose
edges are all the edges of $G$ directed from $X$ to $Y$. We often view $G[X,Y]$ as
an undirected bipartite graph. In particular, we say $G[X,Y]$ is \emph{$\eps$-regular},
\emph{$(\eps,\ge d)$-regular}, $[\eps,d]$-superregular or $[\eps,\ge d]$-superregular
if this holds when $G[X,Y]$ is viewed as an undirected graph.%

We often use the following simple proposition which follows easily from the definition of (super-)regularity.
We omit the proof, a similar argument can be found e.g.~in~\cite{Kelly}.%
   \COMMENT{Daniela: previously had $0<1/m \ll \eps \le d' \le d \leq 1$. A proof for the new hierarchy in the prop can be found in a comment in paper 1.
I also deleted the first part of the lemma on $\eps$-regularity since we never use this.}
\begin{prop} \label{superslice} 
Suppose that $0<1/m \ll \eps \le d' \ll d \leq 1$. Let $G$ be a bipartite graph with vertex classes
$A$ and $B$ of size $m$. Suppose that $G'$ is obtained from $G$ by removing at most $d'm$
vertices from each vertex class and at most $d'm$ edges incident to each vertex from $G$.
If $G$ is $[\eps,d]$-superregular then $G'$ is $[2\sqrt{d'},d]$-superregular.
\end{prop}
We will also use the following simple fact.
\begin{fact}\label{simplefact} Let $\eps >0$.
Suppose that $G$ is a bipartite graph with vertex classes  of size $n$ such that
$\delta (G) \geq (1-\eps)n$. Then $G$ is $[\sqrt{\eps},1]$-superregular.
\end{fact}

\subsection{A Chernoff-Hoeffding bound}
We will often use the following Chernoff-Hoeffding bound for binomial and hypergeometric
distributions (see e.g.~\cite[Corollary 2.3 and Theorem 2.10]{Janson&Luczak&Rucinski00}).
Recall that the binomial random variable with parameters $(n,p)$ is the sum
of $n$ independent Bernoulli variables, each taking value $1$ with probability $p$
or $0$ with probability $1-p$.
The hypergeometric random variable $X$ with parameters $(n,m,k)$ is
defined as follows. We let $N$ be a set of size $n$, fix $S \subseteq N$ of size
$|S|=m$, pick a uniformly random $T \subseteq N$ of size $|T|=k$,
then define $X:=|T \cap S|$. Note that $\mathbb{E}X = km/n$.

\begin{prop}\label{chernoff}
Suppose $X$ has binomial or hypergeometric distribution and $0<a<3/2$. Then
$\mathbb{P}(|X - \mathbb{E}X| \ge a\mathbb{E}X) \le 2 e^{-\frac{a^2}{3}\mathbb{E}X}$.
\end{prop}

%%%%%%%%%%%%%%%%%%%%%%%%%
\section{Overview of the proofs of Theorems~\ref{1factbip} and~\ref{NWmindegbip}}\label{sec:sketch}
Note%
    \COMMENT{Deryk changed this section - so read again}
that, unlike in Theorem~\ref{1factbip}, in Theorem~\ref{NWmindegbip} we do not require a complete
decomposition of our graph $F$ into edge-disjoint Hamilton cycles. Therefore, the proof of Theorem~\ref{1factbip} is considerably more involved than the proof of Theorem~\ref{NWmindegbip}.
Moreover, the ideas in the proof of Theorem~\ref{NWmindegbip} are all used in the proof of Theorem~\ref{1factbip} too.

\subsection{Proof overview for Theorem~\ref{NWmindegbip}} \label{sec:sketch1}
Let $F$ be a graph on $n$ vertices with $\delta (F) \geq (1/2-o(1))n$ which is close to the balanced
bipartite graph $K_{n/2, n/2}$. Further, suppose that $G$ is a $D$-regular spanning subgraph of $F$
as in Theorem~\ref{NWmindegbip}.
Then there is a partition $A$, $B$ of $V(F)$ such that $A$
and $B$ are of roughly equal size and most edges in $F$ go between $A$ and $B$.
 Our ultimate aim is to construct $D/2$ edge-disjoint Hamilton
cycles in $F$.

Suppose first that, in the graph $F$, both $A$ and $B$ are independent sets of equal size. So
$F$ is an almost complete balanced bipartite graph.
In this case, the densest spanning even-regular subgraph $G$ of $F$ is also almost complete bipartite.
This means that one can extend existing techniques (developed e.g. in \cite{CKO, FKS,fk, HartkeHCs,OS})
to find an approximate Hamilton decomposition.
This is achieved in \cite{paper3} and is more than enough to prove Theorem~\ref{NWmindegbip} in this case.
(We state the main result from \cite{paper3} as Lemma~\ref{almostthm} in the current paper.)
The real difficulties arise when
\begin{itemize}
	\item[{\rm (i)}] $F$ is unbalanced;
	\item[{\rm (ii)}] $F$ has vertices having high degree in both $A$ and $B$ (these are called exceptional vertices).
\end{itemize}

To illustrate (i), consider the following example due to Babai%
\COMMENT{osthus added Babai} 
(which is the extremal construction for Corollary~\ref{NWmindegcor}).
Consider the graph $F$ on $n = 8k+2$ vertices consisting of one vertex class $A$ of size $4k+2$ containing
a perfect matching and no other edges, one empty
vertex class $B$ of size $4k$,  and all possible edges between $A$ and $B$.
Thus the minimum degree of $F$ is $4k+1=n/2$. 
Then one can use Tutte's factor theorem to show that the largest even-regular spanning subgraph $G$ of $F$ has degree $D = 2k=(n-2)/4$.
Note that to prove Theorem~\ref{NWmindegbip} in this case, each of the $D/2=k$ Hamilton cycles we find must contain exactly two of the $2k+1$ edges in $A$.
In this way, we can `balance out' the difference in the vertex class sizes.

More generally we will construct our Hamilton cycles in two steps.
In the first step, we find a path system $J$ which balances out the vertex class sizes (so in the above example, $J$ would contain two edges in $A$).
Then we extend $J$ into a Hamilton cycle using only $AB$-edges in $F$.
It turns out that the first step is the difficult one.
It is easy to see that a path system $J$%
\COMMENT{Andy: added $J$}
will balance out the sizes of $A$ and $B$ (in the sense that the number of uncovered vertices in $A$ and $B$ is the same) if 
and only if
\begin{align}
e_J(A)- e_J(B) = |A| - |B|. \label{eJA}
\end{align}
Note that any Hamilton cycle also satisfies this identity.
So we need to find a set of $D/2$ path systems $J$ satisfying \eqref{eJA} (where $D$ is the degree of $G$).
This is achieved (amongst other things) in Sections~\ref{sec:slice} and~\ref{sec:global}.

As indicated above, our aim is to use Lemma~\ref{almostthm} in order to extend each such $J$ into a Hamilton cycle.
To apply Lemma~\ref{almostthm} we also need to extend the balancing path systems $J$ into `balanced exceptional (path)
systems' which contain all the exceptional vertices from (ii).
This is achieved in Section~\ref{besconstruct}.
Lemma~\ref{almostthm} also assumes that the path systems are `localized' with respect to a given subpartition of $A,B$ (i.e.~they are induced by a small number of partition classes).
Section~\ref{partition} prepares the ground for this.

Finding the balanced exceptional systems is extremely difficult if $G$ contains edges between the set $A_0$ of exceptional vertices in $A$ and the
set $B_0$ of exceptional vertices in $B$.
So in a preliminary step, we find and remove a small number of edge-disjoint Hamilton cycles covering all $A_0B_0$-edges in Section~\ref{eliminating}.
We put all these steps together in Section~\ref{sec:proof1}.
(Sections~\ref{sec:spec}, \ref{sec:robust} and~\ref{sec:proof2}%
\COMMENT{osthus added sec 9} 
are only relevant for the proof of Theorem~\ref{1factbip}.)

\subsection{Proof overview for Theorem~\ref{1factbip}}
The main result of this paper is Theorem~\ref{1factbip}.
Suppose that $G$ is a $D$-regular graph satisfying the conditions of that theorem.
Using the approach of the previous subsection, one can obtain an approximate decomposition of $G$,
i.e.~a set of edge-disjoint Hamilton cycles covering almost all edges of~$G$.
However, one does not have any control over the `leftover' graph~$H$, which makes a complete decomposition seem infeasible.
This problem was overcome in~\cite{Kelly} by introducing the concept of a `robustly decomposable graph'~$G^{\rm rob}$.
Roughly speaking, this is a sparse regular graph with the following property:
given \emph{any} very sparse regular graph~$H$ with $V(H)=V(G^{\rm rob})$ which is edge-disjoint from $G^{\rm rob}$,
one can guarantee that $G^{\rm rob} \cup H$ has a Hamilton decomposition.
This leads to the following strategy to obtain a decomposition of $G$:
\begin{itemize}
\item[(1)] find a (sparse) robustly decomposable graph~$G^{\rm rob}$ in $G$ and let $G'$ denote the leftover;
\item[(2)] find an approximate Hamilton decomposition of $G'$ and let $H$ denote the (very sparse) leftover;
\item[(3)] find a Hamilton decomposition of~$G^{\rm rob} \cup H$.
\end{itemize}
It is of course far from obvious that such a graph $G^{\rm rob}$ exists.
By assumption our graph $G$ can be partitioned into two classes $A$ and $B$ of almost equal size such that almost all the edges in $G$ go between $A$ and $B$.
If both $A$ and $B$ are independent sets of equal size then the `robust decomposition lemma' of~\cite{Kelly} guarantees our desired subgraph $G^{\rm rob}$ of $G$.
Of course, in general our graph $G$  will contain edges in $A$ and $B$. 
Our aim is therefore to replace such edges with `fictive  edges' between $A$ and $B$, so that we can apply the 
robust decomposition lemma (which is introduced in Section~\ref{sec:robust}).

More precisely, similarly as in the proof of Theorem~\ref{NWmindegbip}, we construct a collection of localized balanced exceptional systems.
Together these path systems contain all the edges in $G[A]$ and $G[B]$. 
Again,%
\COMMENT{osthus added `again'} 
each balanced exceptional system balances out the sizes of $A$ and $B$ and covers the exceptional
vertices in $G$ (i.e.~those vertices having high degree into both $A$ and $B$).

By replacing  edges of the balanced exceptional systems with fictive edges,
we obtain from $G$ an auxiliary (multi)graph $G^*$ which only contains edges between $A$ and $B$ and 
which does not contain the exceptional vertices of $G$. 
This will allow us to apply
 the robust decomposition lemma. In particular this ensures that each Hamilton cycle obtained in $G^*$ contains
a collection of fictive edges corresponding to a single balanced exceptional system (the set-up of
the robust decomposition lemma does allow for this). Each such Hamilton cycle in $G^*$ then corresponds
to a Hamilton cycle in $G$.

We now give an example of how we introduce fictive edges.
Let $m$ be an integer so that $(m-1)/2$ is even. Set $m':= (m-1)/2$ and $m'':= (m+1)/2$. 
Define the graph $G$ as follows:
Let $A$ and $B$ be disjoint vertex sets of size $m$. Let $A_1, A_2$ be a partition of $A$ and
$B_1,B_2$ be a partition of $B$ such that $|A_1|=|B_1|=m''$. Add all edges between $A$ and $B$. Add
a matching $M_1=\{e_1, \dots, e_{m'/2}\}$ covering precisely the
vertices of $A_2$ and add a matching $M_2=\{e'_1, \dots, e'_{m'/2}\}$ covering precisely the vertices of $B_2$. Finally add a vertex
$v$ which sends an edge to every vertex in $A_1 \cup B_1$. So $G$ is $(m+1)$-regular (and $v$ would be regarded as a exceptional vertex).

Now pair up each edge $e_i$ with the edge $e'_i$. Write $e_i=x_{2i-1} x_{2i}$ and
$e'_i=y_{2i-1} y_{2i}$ for each $1\leq i \leq m'/2$.
Let $A_1 =\{a_1, \dots , a_{m''}\}$ and $B_1=\{b_1, \dots , b_{m''}\}$ and write
$f_i: =a_ib_i$ for all $1 \leq i \leq m''$. Obtain $G^*$ from $G$ by deleting  $v$ together with the
edges in $M_1 \cup M_2$ and by adding the following fictive edges: 
add $f_i$ for each $1\leq i \leq m''$ and add $x_jy_j$ for each $1\leq j \leq m'$.
Then $G^*$ is a balanced bipartite $(m+1)$-regular multigraph containing only edges between $A$ and $B$.

First, note that any Hamilton cycle $C^*$ in $G^*$ that contains precisely one fictive edge $f_i$ for some
$1\le i \leq m''$ corresponds to a Hamilton cycle $C$ in $G$, where we replace the fictive edge $f_i$ with
$a_iv$ and $b_iv$.
Next, consider any Hamilton cycle $C^*$ in $G^*$ that contains precisely three fictive edges; $f_i$ for some
$1\le i \leq m''$ together with $x_{2j-1}y_{2j-1}$ and $x_{2j}y_{2j}$ for some $1\leq j \leq m'/2$.
Further suppose $C^*$ traverses the vertices $a_i,b_i,x_{2j-1},y_{2j-1}, x_{2j},y_{2j}$ in this order. Then $C^*$ corresponds
to a Hamilton cycle $C$ in $G$, where we replace the fictive edges with $a_iv, b_i v, e_j$ and $e'_j$
(see Figure~\ref{fig:bridges}). Here the path system $J$ formed by the edges $a_iv, b_iv, e_j$ and $e'_j$ is an example
of a balanced exceptional system.
The above ideas are formalized in Section~\ref{sec:spec}.

\begin{figure}[htb!]
\begin{center}%\footnotesize
\includegraphics[width=0.34\columnwidth]{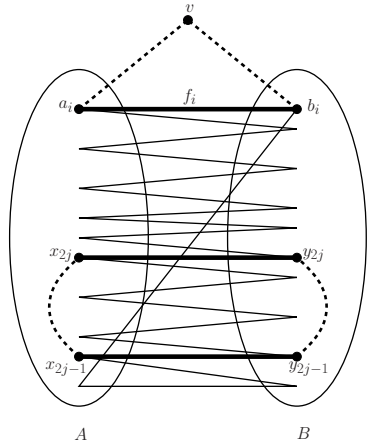}  % Filename.eps
\caption{Transforming the problem of finding a Hamilton cycle in $G$ into finding a Hamilton cycle in the balanced bipartite
graph $G^*$}
\label{fig:bridges}
\end{center}
\end{figure}

We can now summarize the steps leading
to proof of Theorem~\ref{1factbip}.
In Section~\ref{eliminating}, we find and remove a set of edge-disjoint Hamilton cycles covering all edges in $G[A_0,B_0]$.
We can then find the localized balanced exceptional systems in Section~\ref{findBES}. After this, we need to extend and combine them into 
certain path systems and factors (which contain fictive edges) in Section~\ref{sec:spec}, before we can use them as an `input' for the robust decomposition lemma in~Section~\ref{sec:robust}.
Finally, all these steps are combined in Section~\ref{sec:proof2} to prove Theorem~\ref{1factbip}.

%%%%%%%%%%%%%%%%%%%%%%%%%%%%%%%%%%%%%%%%%%%%%%%%%%%%%%%%%%%%%%%%%%%%%%%%%%%%%%%%%%%%
\section{Eliminating edges between the exceptional sets}\label{eliminating}

Suppose that $G$ is a $D$-regular graph as in Theorem~\ref{1factbip}. 
The purpose of this section is to prove Corollary~\ref{coverA0B02c}.
Roughly speaking, given $K\in \mathbb{N}$, this corollary states that one can delete a small number of edge-disjoint Hamilton cycles from $G$
to obtain a spanning subgraph $G'$ of $G$ and a partition $A,A_0,B,B_0$ of $V(G)$ such that (amongst others) the following properties hold:
\begin{itemize}
\item almost all edges of $G'$ join $A\cup A_0$ to $B\cup B_0$;
\item $|A|=|B|$ is divisible by $K$;
\item every vertex in $A$ has almost all its neighbours in $B\cup B_0$ and every vertex in $B$ has almost all its neighbours in $A\cup A_0$;
\item $A_0\cup B_0$ is small and there are no edges between $A_0$ and
$B_0$ in $G'$.
\end{itemize}
We will call $(G',A,A_0,B,B_0)$ a framework.
(The formal definition of a framework is stated before Lemma~\ref{coverA0B02}.)%
\COMMENT{AL: added this sentense}
Both $A$ and $B$ will then be split into $K$ clusters of equal size.
Our assumption that $G$ is $\eps_{\rm ex}$-bipartite easily implies that there is such a partition $A,A_0,B,B_0$
which satisfies all these properties apart from the property that there are no edges
between $A_0$ and $B_0$.
So the main part of this section shows that we can cover the collection of all  edges 
between $A_0$ and $B_0$ by a small number of edge-disjoint
Hamilton cycles.

Since Corollary~\ref{coverA0B02c} will also be used in the proof of Theorem~\ref{NWmindegbip}, instead of working with regular
graphs we need to consider so-called balanced graphs. We also need to find the above Hamilton cycles in the graph $F\supseteq G$ rather than
in $G$ itself (in the proof of Theorem~\ref{1factbip} we will take $F$ to be equal to~$G$).
 
More precisely, suppose that $G$ is a graph and that $A'$, $B'$ is a partition of $V(G)$, 
where $A'=A_0\cup A$, $B'=B_0\cup B$ and $A,A_0,B,B_0$ are disjoint.
Then we say that $G$ is \emph{$D$-balanced (with respect to $(A,A_0,B,B_0)$)} if 
\begin{itemize}
\item[(B1)] $e_G(A')-e_G(B')=(|A'| -|B'|)D/2$;
\item[(B2)] all vertices in $A_0 \cup B_0$ have degree exactly $D$.%
    \COMMENT{Previously (B2) also included that all vertices in $A \cup B$ have degree at least $D$. But then
$A_0B_0$-path systems are not $2$-balanced since some vertices in $A\cup B$ have degree zero. Check whether we made the necessary changes
and worked with the new def throughout the paper...}
\end{itemize}
Proposition~\ref{edge_number} below implies that whenever $A,A_0,B,B_0$ is a partition of the vertex set of a $D$-regular graph $H$,
then $H$ is $D$-balanced with respect to $(A,A_0,B,B_0)$.
Moreover, note that if $G$ is $D_G$-balanced with respect to $(A,A_0,B,B_0)$ and
$H$ is a spanning subgraph of $G$ which is $D_H$-balanced with respect to $(A,A_0,B,B_0)$, then $G-H$ is
$(D_G-D_H)$-balanced with respect to $(A,A_0,B,B_0)$. Furthermore, a graph $G$ is $D$-balanced with respect to
$(A,A_0,B,B_0)$ if and only if $G$ is $D$-balanced with respect to $(B,B_0,A,A_0)$.
\begin{prop}\label{edge_number}
Let $H$ be a graph and let $A'$, $B'$ be a partition of $V(H)$. Suppose that $A_0$, $A$ is a partition of $A'$
and that $B_0$, $B$ is a partition of $B'$ such that $|A|=|B|$. Suppose that $d_H(v)=D$ for every $v\in A_0\cup B_0$
and $d_H(v)=D'$ for every $v\in A\cup B$. Then $e_H(A')-e_H(B')=(|A'| -|B'|)D/2.$ 
\end{prop}
\proof
Note that $$\sum_{x\in A'}d_H(x, B')=e_H(A',B')=\sum_{y\in B'} d_H(y, A').$$ Moreover,
$$2e_H(A')=\sum_{x\in A_0}(D-d_H(x, B'))+\sum_{x\in A}(D'-d_H(x, B'))=D|A_0|+D'|A|-\sum_{x\in A'}d_H(x, B')$$ and
$$2e_H(B')=\sum_{y\in B_0}(D-d_H(y, A'))+\sum_{y\in B}(D'-d_H(y, A'))=D|B_0|+D'|B|-\sum_{y\in B'}d_H(y, A').$$
Therefore $$2e_H(A')-2e_H(B')=D(|A_0|-|B_0|)+D'(|A|-|B|)=D(|A_0|-|B_0|)=D(|A'|-|B'|),$$
as desired.
\endproof

The following observation states that balancedness is preserved under
suitable modifications of the partition.
\begin{prop} \label{keepbalance}
Let $H$ be $D$-balanced with respect to $(A,A_0,B,B_0)$. 
 Suppose that $A'_0,B'_0$ is a partition of $A_0 \cup B_0$.
Then $H$ is $D$-balanced with respect to $(A,A'_0,B,B'_0)$.
\end{prop}
\proof
Observe that the general result follows if we can show that $H$ is $D$-balanced with respect to $(A,A'_0,B,B'_0)$, where
$A'_0=A_0 \cup \{ v\}$, $B'_0=B_0 \setminus \{v\}$ and $v \in B_0$.
(B2) is trivially satisfied in this case, so we only need to check (B1) for the new partition.
For this, let $A':=A_0 \cup A$ and $B':=B_0 \cup B$. Now note that (B1) for the original partition implies that
\begin{align*}
e_H(A'_0 \cup A)-e_H(B'_0 \cup B) & = e_H(A') +d_H(v,A') -(e_H(B')-d_H(v,B'))  \\
& = (|A'|-|B'|)D/2 +D =(|A'_0 \cup A|-|B'_0 \cup B|)D/2.
\end{align*}
Thus (B1) holds for the new partition.
\endproof

Suppose that $G$ is a graph and  $A',B'$ is a partition of $V(G)$.
For every vertex $v\in A'$ we call $d_G(v,A')$ the \emph{internal degree
of~$v$} in~$G$. Similarly, for every vertex $v\in B'$ we call $d_G(v,B')$ the \emph{internal degree
of~$v$} in~$G$.

Given a graph $F$ and a spanning subgraph $G$ of $F$ , 
we say that $(F,G,A,A_0,B,B_0)$ is an \emph{$(\eps,\eps',K,D)$-weak framework} if the following holds,
where $A':=A_0\cup A$, $B':=B_0\cup B$ and $n:=|G|=|F|$: 
\begin{itemize}
\item[{\rm (WF1)}] $A,A_0,B,B_0$ forms a partition of $V(G)=V(F)$;
\item[{\rm (WF2)}] $G$ is $D$-balanced with respect to $(A,A_0,B,B_0)$;
\item[{\rm (WF3)}] $e_G(A'), e_G(B')\le \eps n^2$;
\item[{\rm (WF4)}] $|A|=|B|$ is divisible by $K$. Moreover,
 $a+b\le \eps n$, where $a:=|A_0|$ and $b:=|B_0|$;
\item[{\rm (WF5)}] all vertices in $A \cup B$  have internal degree at most $\eps' n$ in $F$;
\item[{\rm (WF6)}] any vertex $v$ has internal degree at most $d_G(v)/2$ in $G$.%
\COMMENT{NEW: Got rid of error term here.}
\end{itemize}
 Throughout the paper, when referring to internal degrees without mentioning the partition, we always mean with
respect to the partition $A'$, $B'$, where $A'=A_0\cup A$ and $B'=B_0\cup B$. Moreover, $a$ and $b$ will always denote $|A_0|$ and $|B_0|$.

We say that $(F,G,A,A_0,B,B_0)$ is an \emph{$(\eps,\eps',K,D)$-pre-framework} if it satisfies (WF1)--(WF5).
The following observation states that pre-frameworks are preserved if we remove suitable balanced subgraphs.

\begin{prop} \label{WFpreserve}
Let $\eps, \eps ' >0$ and $K,D_G,D_{H} \in \mathbb N$.%
	\COMMENT{AL:changed $D_{G'}$ to $D_{H}$}
Let $(F,G,A,A_0,B,B_0)$ be an $(\eps,\eps',K,D_G)$-pre framework.
Suppose that $H$ is a $D_H$-regular spanning subgraph of $F$ such that
 $G\cap H$ is $D_H$-balanced with respect to $(A,A_0,B,B_0)$. 
Let $F':=F-H$ and $G':=G-H$. Then $(F',G',A,A_0,B,B_0)$ is an $(\eps,\eps',K,D_G-D_H)$-pre framework.
\end{prop}
\proof
Note that all required properties except possibly (WF2) are not affected by removing edges.
But $G'$ satisfies (WF2) since $G\cap H$ is $D_H$-balanced with respect to $(A,A_0,B,B_0)$.
\endproof

\begin{lemma}\label{bip_decomp}
Let $0<1/n\ll \varepsilon \ll  \varepsilon', 1/K\ll 1$ and let $D\ge n/200$. 
Suppose that $F$ is a  graph on $n$ vertices which is $\eps$-bipartite and that $G$ is a $D$-regular spanning subgraph of $F$.
Then there is a partition $A,A_0,B,B_0$ of $V(G)=V(F)$ so that 
$(F,G,A,A_0,B,B_0)$ is an $(\eps^{1/3},\eps',K,D)$-weak framework.
\end{lemma}
\proof
Let $S_1,S_2$ be a partition of $V(F)$ which is guaranteed by the assumption that $F$ is $\eps$-bipartite.
Let $S$ be the set of all those vertices $x \in S_1$ with $d_F(x,S_1) \ge \sqrt{\eps} n$
together with all those vertices $x \in S_2$ with $d_F(x,S_2) \ge \sqrt{\eps} n$.
Since $F$ is $\eps$-bipartite, it follows that $|S| \le 4\sqrt{\eps} n$.%
    \COMMENT{Have $4\sqrt{\eps} n$ instead of $2\sqrt{\eps} n$ since we get $|S\cap S_i|\sqrt{\eps} n\le 2e(S_i)\le 2\eps n^2$.}

Given a partition $X,Y$ of $V(F)$, we say that $v \in X$ is \emph{bad for $X,Y$} if $d_G(v,X)>d_G(v,Y)$
and similarly that  $v \in Y$ is \emph{bad for $X,Y$} if $d_G(v,Y)>d_G(v,X)$.
Suppose that there is a vertex $v \in S$ which is bad for $S_1$, $S_2$. 
Then we move $v$ into the class which does not currently contain $v$ to obtain a new partition $S'_1$, $S'_2$.
We do not change the set $S$.
If there is a vertex $v' \in S$ which is bad for $S'_1$, $S'_2$, then again we move it into the other class.

We repeat this process. After each step, the number of edges in $G$ between the two classes increases, so this process has to terminate with some partition $A'$, $B'$
such that $A' \bigtriangleup S_1 \subseteq  S$ and $B' \bigtriangleup S_2 \subseteq  S$.
Clearly, no vertex in $S$ is now bad for $A'$, $B'$. Also, for any $v \in A' \setminus S$ we have 
\begin{align}\label{eq:dvA'}
d_G(v,A') & \le d_F (v,A') \le d_F(v,S_1)+|S|\le \sqrt{\eps} n + 4\sqrt{\eps} n< \eps'n \\
& < D/2= d_G(v)/2. \nonumber
\end{align}
 Similarly, $d_G(v,B') < \eps' n< d_G(v)/2$ for all $v \in B' \setminus S$.
Altogether this implies that no vertex is bad for $A'$, $B'$ and thus (WF6) holds.
 Also note that
$e_G(A',B')\ge e_G(S_1,S_2)\ge e(G)-2\eps n^2$.%
	\COMMENT{AL: replaced $2\eps$ with $4\eps$ same with equation below, DO changed these back again...}
So 
\begin{equation}\label{A'B'edge}
e_G(A'),e_G(B') \le 2\eps n^2.
\end{equation}
This implies~(WF3).

Without loss of generality we may assume that $|A'|\ge |B'|$. 
Let $A'_0$ denote the set of all those vertices $v\in A'$ for which $d_F(v, A')\ge \varepsilon' n$.
Define $B'_0\subseteq B'$ similarly.
We will choose sets $A \subseteq A' \setminus A'_0$ and $A_0\supseteq A'_0$
and sets $B\subseteq B'\setminus B'_0$ and $B_0\supseteq B'_0$ such that $|A|=|B|$ is 
divisible by $K$ and so that $A,A_0$ and $B,B_0$ are partitions of $A'$ and $B'$ respectively.
We obtain such sets by moving at most 
$\left| |A'\setminus A'_0|-|B'\setminus B'_0|\right| +K$ vertices from 
$A'\setminus A'_0$ to $A'_0$ and at most $\left| |A'\setminus A'_0|-|B'\setminus B'_0|\right| +K$ vertices from $B'\setminus B'_0$ to $B'_0$.
The choice of $A,A_0,B,B_0$ is such that (WF1) and  (WF5) hold.  Further, since $|A|=|B|$, Proposition~\ref{edge_number} implies~(WF2).

In order to verify (WF4), it remains to show that $a+b=|A_0\cup B_0|\le  \eps^{1/3} n$.
But (\ref{eq:dvA'}) together with its analogue for the vertices in $B'\setminus S$ implies that $A'_0\cup B'_0\subseteq S$.
Thus $|A'_0|+ |B'_0|\le |S|\le 4\sqrt{\eps} n$. 
Moreover, (WF2), (\ref{A'B'edge}) and our assumption that $D\ge n/200$ together imply that
$$|A'|-|B'| =(e_G(A')-e_G(B'))/(D/2)\le 2\eps n^2/(D/2) \le 800 \eps n.$$
So altogether, we have 
\begin{align*}
a+b & \le |A'_0\cup B'_0|+2\left| |A'\setminus A'_0|-|B'\setminus B'_0|\right| +2K \\ &
\le
4\sqrt{\eps} n+2\left| |A'|-|B'|-(|A'_0|-|B'_0|)\right|+2K \\ &
 \le 4\sqrt{\eps} n+1600 \eps n+ 8\sqrt{\eps} n+2K
\le \eps^{1/3}n.
\end{align*}
Thus (WF4) holds.
\endproof

Throughout this and the next section, we will often use the following result, which is a simple consequence of 
Vizing's theorem and was first  observed by McDiarmid and independently by de Werra (see e.g.~\cite{west}).%
\COMMENT{osthus removed proof}
\begin{prop}\label{basic_matching_dec}\
 Let $H$ be a graph with maximum degree at most $\Delta.$
Then $E(H)$ can be decomposed into $\Delta+1$ edge-disjoint matchings $M_1, \ldots, M_{\Delta+1}$
such that $||M_i|-|M_j||\le 1$ for all $i, j\le \Delta+1$.
\end{prop}
%\proof  By Vizing's theorem~\cite{Viz} one can 
%properly colour the edges of $H$ with $\Delta+1$ colours. Clearly,
%each colour class is a matching of $H.$ Assume that $M_1, M_2$ are two colour classes so that $|M_1|>|M_2|+1.$
%$M_1 \cup M_2$ is the disjoint union of alternating paths and even cycles. Since $M_1$ is larger, there is an alternating path $P$
%that begins and ends with an edge from $M_1.$ Switching colours along $P$ will decrease the size of $M_1$ by 1 and increase the size of $M_2$ by 1.
%If $M_1$ still contains at least 2 more edges than $M_2,$ we can repeat the above process. We can proceed in this way to achieve
%that $|M_2|\le |M_1|\le |M_2|+1.$ Iterating this for every pair of matchings with sizes that differ by at least 2,
%we get our desired matchings.
%\endproof

Our next goal is to cover the edges of $G[A_0 , B_0]$ by edge-disjoint Hamilton cycles.
To do this, we will first decompose $G[A_0 , B_0]$ into a collection of matchings. We will then extend each such matching into a system of
vertex-disjoint paths such that altogether these paths cover every vertex in $G[A_0 , B_0]$, each path has its endvertices in $A\cup B$ and
the path system is $2$-balanced. Since our path system will only contain a
small number of nontrivial paths, we can then extend the path system into a Hamilton cycle (see Lemma~\ref{extendpaths}).  

We will call the path systems we are working with $A_0B_0$-path systems. More precisely, an \emph{$A_0B_0$-path system (with respect to
$(A,A_0,B,B_0)$)} is a path system $Q$ satisfying the following properties:
\begin{itemize}
\item Every vertex in $A_0\cup B_0$ is an internal vertex of a path in $Q$.
\item $A\cup B$ contains the endpoints of each path in $Q$ but no internal vertex of a path in $Q$.
\end{itemize}
The following observation (which motivates the use of the word `balanced') will often be helpful.%
\COMMENT{Deryk added bracket}

\begin{prop} \label{balpathcheck}
Let $A_0,A,B_0,B$ be a partition of a vertex set $V$.
Then an $A_0B_0$-path system $Q$ with $V(Q)\subseteq V$ is $2$-balanced with respect to $(A,A_0,B,B_0)$ if and only if
the number of vertices in $A$ which are endpoints of nontrivial paths
in $Q$ equals the number of vertices in $B$ which are endpoints of nontrivial paths in~$Q$.
\end{prop}
\proof
Note that by definition any $A_0B_0$-path system satisfies (B2), so we only need to consider (B1).
Let $n_A$ be the number of  vertices in $A$ which are endpoints of nontrivial paths
in $Q$ and define $n_B$ similarly.
Let $a:=|A_0|$, $b:=|B_0|$, $A':=A\cup A_0$ and $B':=B\cup B_0$.
Since $d_Q(v) =2 $ for all $v \in A_0$ and since every vertex in $A$ is either an endpoint of a nontrivial path
in $Q$ or has degree zero in~$Q$, we have
\begin{align*}
	2e_Q(A') + e_Q(A',B') = \sum_{v \in A'} d_Q (v) = 2a +n_A.
\end{align*}
So $n_A = 2(e_Q(A')-a) + e_Q(A',B')$, and similarly $n_B = 2(e_Q(B')-b) + e_Q(A',B') $.
Therefore, $n_A = n_B$ if and only if $2( e_Q(A') - e_Q(B')-a +b) = 0 $ if and only if $Q$ satisfies~(B1), as desired.
\endproof

The next observation shows that if we have a suitable path system satisfying (B1), we can extend it into a path system which 
also satisfies (B2).%
\COMMENT{Deryk added new sentence}

\begin{lemma}\label{balpathextend}
Let $0<1/n\ll \alpha \ll 1$.
Let $G$ be a graph on $n$ vertices such that there is a partition $A',B'$ of $V(G)$ which satisfies the following properties:
\begin{itemize}
\item[{\rm (i)}] $A'=A_0\cup A$, $B'=B_0\cup B$ and $A_0,A,B_0,B$ are disjoint;
\item[{\rm (ii)}] $|A|=|B|$ and $a+b \le \alpha n$, where $a:=|A_0|$ and $b:=|B_0|$;
\item[{\rm (iii)}] if $v \in A_0$ then $d_G(v,B) \ge 4 \alpha n$ and if $v \in B_0$ then $d_G(v,A) \ge 4 \alpha  n$.
\end{itemize}
Let $Q'\subseteq G$ be a path system consisting of at most $\alpha n$ nontrivial
paths such that $A \cup B$ contains no internal vertex of a path
in $Q'$ and $e_{Q'}(A')   - e_{Q'}(B')= a -b $.
Then $G$ contains a $2$-balanced $A_0B_0$-path system $Q$ (with respect to $(A,A_0,B,B_0)$) which extends~$Q'$ and
consists of at most $2 \alpha n$ nontrivial paths.
Furthermore, $E(Q)\setminus E(Q') $ consists of $A_0B$- and $AB_0$-edges only.
\end{lemma}
\proof
Since $A \cup B$ contains no internal vertex of a path in $Q'$ and since $Q'$ contains at most $\alpha n$ nontrivial
paths, it follows that at most $2 \alpha n$ vertices in $A \cup B$ lie on nontrivial paths in $Q'$.
We will now extend $Q'$ into an $A_0B_0$-path system $Q$ consisting of at most $a+b +  \alpha n \le 2 \alpha n$ nontrivial paths as follows:
\begin{itemize}
\item for every vertex $v \in A_0$, we join $v$ to $2-d_{Q'}(v)$ vertices in $B$;
\item for every vertex $v \in B_0$, we join $v$ to $2-d_{Q'}(v)$ vertices in $A$.
\end{itemize}
Condition~(iii) and the fact that at most $2\alpha n$ vertices in $A\cup B$ lie on nontrivial paths in $Q'$ together ensure that we can extend
$Q'$ in such a way that the endvertices in $A\cup B$ are distinct for different paths in $Q$.
Note that $e_Q(A')   - e_Q(B') = e_{Q'}(A')   - e_{Q'}(B')= a -b $.
Therefore, $Q$ is $2$-balanced with respect to $(A,A_0,B,B_0)$.
\endproof

The next lemma constructs a small number of $2$-balanced $A_0B_0$-path systems covering the edges of $G[A_0, B_0]$.
Each of these path systems will later be extended%
\COMMENT{osthus corrected new sentence}
into a Hamilton cycle.%
\COMMENT{In Lemma~\ref{coverA0B01}: We really do use tightness of (WF6) but crucially when applying lemma 
in Chapter 8 $G$ is not regular so can't use this property in this lemma.}

\begin{lemma}\label{coverA0B01}
Let $0<1/n\ll \varepsilon \ll  \varepsilon',1/K\ll \alpha\ll 1$.
Let $F$ be a graph on $n$ vertices and let $G$ be a spanning subgraph of $F$.%
	\COMMENT{AL:add condition for $F$ and $G$}
Suppose that $(F,G,A,A_0,B,B_0)$
is an $(\eps,\eps',K,D)$-weak framework with $\delta(F)\ge (1/4+\alpha )n$
and  $D\geq n/200$.
Then for some $r^* \le \eps n$ the graph $G$ contains $r^*$ edge-disjoint
$2$-balanced $A_0B_0$-path systems $Q_1, \dots , Q_{r^*}$ which satisfy the following
properties:
\begin{itemize}
\item[{\rm (i)}] Together $Q_1,\dots,Q_{r^*}$  cover all edges in $G[A_0, B_0]$;
\item[{\rm (ii)}] For each $i \leq r^*$, $Q_i$ contains at most $2 \eps n$ nontrivial paths;
\item[{\rm (iii)}] For each $i \leq r^*$, $Q_i$ does not contain any edge from $G[A,B]$.
\end{itemize}
\end{lemma}
\proof
(WF4) implies that $|A_0|+|B_0|\leq \eps n$. Thus, by Proposition~\ref{basic_matching_dec}, there exists a collection $M'_1,
\dots , M'_{r^*}$ of $r^*$ edge-disjoint matchings in $G[A_0,B_0]$ that together cover all the
edges in $G[A_0,B_0]$, where $r^* \leq \eps n$. 

We may assume that $a \geq b$ (the case when $b>a$ follows analogously).
We will use  edges in $G[A']$ to extend each $M'_i$ into a $2$-balanced $A_0B_0$-path system. 
(WF2) implies that $e_G (A') \geq (a-b)D/2$. Since $d_G (v)=D$ for all $v \in A_0 \cup B_0$ by (WF2),
(WF5) and (WF6) imply that $\Delta (G[A']) \leq D/2$.
Thus Proposition~\ref{basic_matching_dec} implies that $E(G[A'])$ can be decomposed
into $\lfloor D/2 \rfloor+1$ edge-disjoint matchings $M_{A,1}, \dots , M_{A,\lfloor D/2 \rfloor+1}$ such that $||M_{A,i}|-|M_{A,j}|| \leq 1$ for all $i,j \leq \lfloor D/2 \rfloor+1$.

Notice that at least $\eps n$ of the matchings $M_{A,i}$ are such that $|M_{A,i}|\geq a-b$.
Indeed, otherwise we have that 
\begin{align*}
(a-b)D/2 \leq e_G (A') & \leq \eps n (a-b) +(a-b-1)(D/2+1-\eps n) 
\\ & = (a-b)D/2 + a-b-D/2-1+\eps n \\ &< (a-b)D/2 +2\eps n -D/2 < (a-b)D/2 ,
\end{align*}
a contradiction. (The last inequality follows since $D \geq n/200$.)
In particular, this implies that $G[A']$ contains $r^*$ edge-disjoint matchings $M''_1, \dots ,M''_{r^*}$ 
that each consist of precisely $a-b$ edges. 

For each $i \leq r^*$, set $M_i:=M'_i \cup M'' _i$. So for each $ i \leq r^*$, $M_i$ is a path
system consisting of at most $b+(a-b)=a \leq \eps n$ nontrivial paths such that $A \cup B$ contains
no internal vertex of a path in $M_i$ and $e_{M_i} (A')-e_{M_i} (B') =e_{M''_i} (A') =a-b$.

Suppose for some $0 \leq r < r^*$ we have already found a collection $Q_1, \dots , Q_r$ of $r$
edge-disjoint $2$-balanced $A_0B_0$-path systems which satisfy the following properties for
each $i \leq r$:
\begin{itemize}
\item[($\alpha$)$_i$] $Q_i$ contains at most $2\eps n$ nontrivial paths;
\item[($\beta$)$_i$]  $M_i \subseteq Q_i$;
\item[($\gamma$)$_i$] $Q_i$ and $M_j$ are edge-disjoint for each $j\leq r^*$ such that $i \not =j$;
\item[($\delta$)$_i$] $Q_i$ contains no edge from $G[A,B]$.
\end{itemize}
(Note that ($\alpha$)$_0$--($\delta$)$_0$ are vacuously true.)
Let $G'$ denote the spanning subgraph of $G$ obtained from $G$ by  deleting  the edges
lying in $Q_1 \cup \dots \cup Q_r$. (WF2), (WF4) and (WF6) imply that, if $v \in A_0$, $d_{G'} (v,B) \geq D/2-\eps n-2r \geq
4 \eps n$ and if $v \in B_0$ then $d_{G'} (v,A) \geq 4 \eps n$.
Thus Lemma~\ref{balpathextend} implies that $G'$ contains a $2$-balanced $A_0B_0$-path system 
$Q_{r+1}$ that satisfies ($\alpha$)$_{r+1}$--($\delta$)$_{r+1}$.

So we can proceed in this way in order to obtain edge-disjoint $2$-balanced $A_0B_0$-path systems
$Q_1, \dots , Q_{r^*}$ in $G$ such that ($\alpha$)$_{i}$--($\delta$)$_{i}$ hold for each $i \leq r^*$.
Note that (i)--(iii) follow immediately from these conditions, as desired.
\endproof

%%%%%%%%%%%%%

The next lemma (Corollary~5.4 in~\cite{3con}) allows us to extend a $2$-balanced path system into a Hamilton cycle.
Corollary~5.4 concerns so-called `$(A,B)$-balanced'-path systems rather than $2$-balanced $A_0B_0$-path systems.
But the latter satisfies the requirements of the former by Proposition~\ref{balpathcheck}.%
\COMMENT{Deryk changed sentence and deleted `balanced' in from of $H$ in the lemma below, Daniela deleted $n_0$ in the lemma below}
\begin{lemma}\label{3conlem}
Let $0<1/n\ll   \varepsilon' \ll \alpha \ll 1.$ 
Let $F$ be a graph and suppose that $A_0,A,B_0,B$ is a partition of $V(F)$ such
that $|A|=|B|=n$. Let $H$ be a  bipartite subgraph of $F$
with vertex classes $A$ and $B$ such that $\delta (H) \geq (1/2+\alpha)n$.
Suppose that $Q$ is a $2$-balanced $A_0B_0$-path system with respect to $(A,A_0,B,B_0)$ in $F$ which consists of at most $\eps' n$
nontrivial paths. Then $F$ contains a Hamilton cycle $C$ which satisfies the following properties:
\begin{itemize}
\item $Q \subseteq C$;
\item $E(C)\setminus E(Q) $ consists of edges from $H$.
\end{itemize}

\end{lemma}

Now we can apply Lemma~\ref{3conlem} to  extend a $2$-balanced $A_0B_0$-path system in
a pre-framework into a Hamilton cycle.%
  \COMMENT{AT: I have dropped $\delta (G)\geq D$ condition in this lemma.}

\begin{lemma}\label{extendpaths}
Let $0<1/n\ll \varepsilon \ll  \varepsilon',1/K \ll \alpha \ll 1$.
Let $F$ be a graph on $n$ vertices and let $G$ be a spanning subgraph of $F$.%
	\COMMENT{AL:add condition for $F$ and $G$}
Suppose that $(F,G,A,A_0,B,B_0)$ is an  $(\eps,\eps',K,D)$-pre-framework, i.e.~it
satisfies (WF1)--(WF5).%
   \COMMENT{this exception is mainly needed for applications in later sections}
Suppose also that  $\delta(F) \ge (1/4+\alpha )n$. 
Let $Q$ be a $2$-balanced $A_0B_0$-path system with respect to $(A,A_0,B,B_0)$ in $G$ which consists of at most $\eps' n$
nontrivial paths. Then $F$ contains a Hamilton cycle $C$ which satisfies the following properties:
\begin{itemize}
\item[{\rm (i)}] $Q\subseteq C$;
\item[{\rm (ii)}] $E(C)\setminus E(Q)$ consists of $AB$-edges;
\item[{\rm (iii)}] $C\cap G$ is $2$-balanced with respect to $(A,A_0,B,B_0)$.
\end{itemize}
\end{lemma}
\proof
Note that (WF4), (WF5) and our assumption that $\delta (F) \ge (1/4+\alpha )n$ together imply that
every vertex $x \in A$ satisfies
$$d_{F} (x,B) \geq d_F(x,B')-|B_0|\geq d_F (x) -\eps ' n -|B_0| \geq (1/4+\alpha /2)n
\geq (1/2+ \alpha /2)|B|.$$
Similarly, $d_F (x,A)\geq (1/2+\alpha /2)|A|$ for all $x \in B$. Thus,
$\delta (F[A,B]) \geq (1/2+\alpha /2)|A|$. Applying Lemma~\ref{3conlem} with
$F[A,B]$ playing the role of $H$, we obtain a Hamilton cycle $C$ in $F$ that
satisfies (i) and (ii). To verify (iii), note that (ii) and the 2-balancedness of $Q$ together imply that
$$e_{C \cap G}(A')-e_{C \cap G}(B')=e_Q(A')-e_Q(B')=a-b.$$
Since every vertex $v\in A_0\cup B_0$ satisfies $d_{C\cap G}(v)=d_Q(v)=2$, (iii) holds.
\endproof
%%%%%%%%%%%%

We now combine Lemmas~\ref{coverA0B01} and~\ref{extendpaths} to find a collection of edge-disjoint Hamilton cycles covering all the edges
in $G[A_0, B_0]$.%
\COMMENT{AT: Again I have dropped that $G$ is $D$-regular. Need this relaxation for Chapter 8} 
\begin{lemma}\label{coverA0B0}
Let $0<1/n\ll \varepsilon \ll  \varepsilon',1/K\ll \alpha\ll 1$ and let $D\geq n/100$.
Let $F$ be a graph on $n$ vertices and let $G$ be a spanning subgraph of $F$.%
	\COMMENT{AL:add condition for $F$ and $G$}
Suppose that $(F,G,A,A_0,B,B_0)$ is an $(\eps,\eps',K,D)$-weak framework with $\delta(F)\ge (1/4+\alpha )n$. 
Then for some $r^* \le \eps n$ the graph $F$ contains edge-disjoint Hamilton cycles $C_1,\dots,C_{r^*}$ which
satisfy the following properties:
\begin{itemize}
\item[{\rm (i)}] Together $C_1,\dots,C_{r^*}$  cover all edges in $G[A_0, B_0]$;
\item[{\rm (ii)}] $(C_1\cup \dots \cup C_{r^*})\cap G$ is $2r^*$-balanced with respect to $(A,A_0,B,B_0)$.
\end{itemize}
\end{lemma}
\proof
Apply Lemma~\ref{coverA0B01} to obtain a collection of $r^* \leq \eps n$ edge-disjoint $2$-balanced
$A_0B_0$-path systems $Q_1, \dots , Q_{r^*}$ in $G$ which satisfy Lemma~\ref{coverA0B01}(i)--(iii).
We will extend each $Q_i$ to a Hamilton cycle $C_i$.

Suppose that for some $0 \le r < r^*$ we have found a collection $C_1,\dots,C_r$ of $r$ edge-disjoint Hamilton
cycles in $F$ such that the following holds for each $0 \leq i \leq r$:
\begin{itemize}
\item[($\alpha$)$_i$] $Q_i \subseteq C_i$;
\item[($\beta$)$_i$] $E(C_i) \setminus E(Q_i)$ consists of $AB$-edges;
\item[($\gamma$)$_i$] $G \cap C_i $ is $2$-balanced with respect to $(A,A_0,B,B_0)$.
\end{itemize}
(Note that ($\alpha$)$_0$--($\gamma$)$_0$ are vacuously true.)
Let $H_r:=C_1\cup \dots \cup C_r$ (where $H_0:=(V(G),\emptyset)$). So $H_r$ is $2r$-regular. 
Further, since  $G \cap C_i$ is $2$-balanced for each $i \leq r$, $G \cap H_r$ is $2r$-balanced. Let
$G_r:=G-H_r$ and $F_r:=F-H_r$. Since $(F,G,A,A_0,B,B_0)$ is an $(\eps,\eps',K,D)$-pre-framework,
Proposition~\ref{WFpreserve} implies that $(F_r,G_r,A,A_0,B,B_0)$ is an $(\eps,\eps',K,D-2r)$-pre-framework. 
Moreover, $\delta (F_r) \geq \delta (F)-2r \geq (1/4+\alpha/2)n$.
Lemma~\ref{coverA0B01}(iii) and ($\beta$)$_1$--($\beta$)$_r$ together%
    \COMMENT{Daniela added more detail}
imply that $Q_{r+1}$ lies in $G_r$.
Therefore, Lemma~\ref{extendpaths} implies that $F_r$ contains a Hamilton cycle
$C_{r+1}$ which satisfies ($\alpha$)$_{r+1}$--($\gamma$)$_{r+1}$.

So we can proceed in this way in order to obtain $r^*$ edge-disjoint Hamilton cycles 
$C_1, \dots , C_{r^*}$ in $F$ such that for each $i\leq r^*$, ($\alpha$)$_{i}$--($\gamma$)$_{i}$
hold. Note that this implies that (ii) is satisfied. Further, the choice of $Q_1, \dots, Q_{r^*}$
ensures that (i) holds.
\endproof

Given a graph $G$, we say that $(G,A,A_0,B,B_0)$ is an \emph{$(\eps,\eps',K,D)$-framework} if the following holds,
where $A':=A_0\cup A$, $B':=B_0\cup B$ and $n:=|G|$:
\begin{itemize}
\item[{\rm (FR1)}]  $A,A_0,B,B_0$ forms a partition of $V(G)$;
\item[(FR2)] $G$ is $D$-balanced with respect to $(A,A_0,B,B_0)$;
\item[{\rm (FR3)}] $e_G(A'), e_G(B')\le \varepsilon n^2$;
\item[{\rm (FR4)}] $|A|=|B|$ is divisible by $K$. Moreover, $b\le a$ and $a+b\le \eps n$, where $a:=|A_0|$ and $b:=|B_0|$;
\item[{\rm (FR5)}] all vertices in $A \cup B$ have internal degree at most $\varepsilon' n$ in $G$;
\item[{\rm (FR6)}] $e(G[A_0, B_0])=0$;
\item[{\rm (FR7)}] all vertices $v\in V(G)$ have internal degree at most $d_G(v)/2+\eps n$ in $G$.
\end{itemize}
%Note that (FR1)--(FR5) imply (WF1)--(WF5). 
Note that the main differences to a weak framework are (FR6) and the fact that a weak framework involves an additional graph $F$.
In particular (FR1)--(FR4) imply (WF1)--(WF4).%
\COMMENT{Deryk: additional 2 sentences}
Suppose that $\eps_1\ge \eps$, $\eps'_1\ge \eps'$ and that $K_1$ divides $K$.
Then note that every $(\eps,\eps',K,D)$-framework is also an $(\eps_1,\eps'_1,K_1,D)$-framework.%
\COMMENT{AL:removed `$(F,G,A,A_0,B,B_0)$ is an $(\eps,\eps',K,D)$-weak framework, thus $F$ satisfies (WF5)
with respect to the partition $A,A_0,B,B_0$.' in the next lemma}

\begin{lemma}\label{coverA0B02}
Let $0<1/n\ll \varepsilon \ll \varepsilon',1/K\ll \alpha \ll 1$ and let
$ D \ge n/100$.
Let $F$ be a graph on $n$ vertices and let $G$ be a spanning subgraph of $F$.%
	\COMMENT{AL:add condition for $F$ and $G$}
Suppose that $(F,G,A,A_0,B,B_0)$ is an $(\eps,\eps',K,D)$-weak framework. Suppose also that $\delta(F) \ge (1/4+\alpha)n$ and $|A_0|\geq |B_0|$.
 Then the following properties hold:
\begin{itemize}
\item[(i)] there is an $(\eps,\eps',K,D_{G'})$-framework
$(G',A,A_0,B,B_0)$ such that $G'$ is a spanning subgraph of $G$ with
$D_{G'} \ge D -2\eps n$;%
\COMMENT{AL: removed. `and such that $F$ satisfies (WF5)
(with respect to the partition $A,A_0,B,B_0$)'.}
\item[(ii)] there is a set of $(D-D_{G'})/2 \leq  \eps  n$
edge-disjoint Hamilton cycles in $F-G'$ containing all edges of $G-G'$.
In particular, if $D$ is even then $D_{G'}$ is even.%
\COMMENT{we do need the explicit statement of the number of ham cycles here.}
\end{itemize}
\end{lemma}
\proof
Lemma~\ref{coverA0B0} implies that there exists some $r^* \leq \eps  n$ such that
$F$ contains a spanning subgraph $H$ satisfying the following properties:
\begin{itemize}
\item[(a)] $H$ is $2r^*$-regular;
\item[(b)] $H$ contains all the edges in $G[A_0,B_0]$;
\item[(c)] $G \cap H$ is $2r^*$-balanced with respect to $( A,A_0,B,B_0)$;
\item[(d)] $H$ has a decomposition into $r^*$ edge-disjoint Hamilton cycles.
\end{itemize}

Set $G':=G-H$. Then $(G',A,A_0,B,B_0)$ is an $(\eps,\eps',K,D_{G'})$-framework
where $D_{G'} :=D-2r^* \geq D- 2\eps n$. Indeed, since $(F,G,A,A_0,B,B_0)$ is an $(\eps ,\eps',K,D)$-weak framework,
(FR1) and (FR3)--(FR5) follow from (WF1) and (WF3)--(WF5). Further, (FR2) follows
from (WF2) and (c) while (FR6) follows from (b). (WF6) implies that all
vertices $v \in V(G)$ have internal degree at most $d_G (v)/2$ in $G$.
Thus  all vertices $v \in V(G')$ have internal degree at most $d_G (v)/2\leq (d_{G'} (v)+2r^*)/2
\leq d_{G'} (v)/2 +\eps n$ in $G'$. So (FR7) is satisfied. Hence, (i) is satisfied.

Note that by definition of $G'$, $H$ contains all edges of $G-G'$. So since 
$r^* =(D-D_{G'})/2 \leq \eps n$, (d) implies (ii).
\endproof

The following result follows immediately from Lemmas~\ref{bip_decomp} and~\ref{coverA0B02}.%
\COMMENT{AT: It is convenient to state both Lemma~\ref{coverA0B02} and Corollary~\ref{coverA0B02c}.
We apply Lemma~\ref{coverA0B02} in a situation where $G$ is not regular. But elsewhere it is handy
to apply Corollary~\ref{coverA0B02c} instead of first applying Lemma~\ref{bip_decomp} and then
Lemma~\ref{coverA0B02}.}
\begin{cor}\label{coverA0B02c}
Let $0<1/n\ll \varepsilon \ll  \eps^*\ll \varepsilon',1/K\ll \alpha \ll 1$ and let
$ D \ge n/100$. Suppose that $F$ is an $\eps$-bipartite graph on $n$ vertices with $\delta(F) \ge (1/4+\alpha)n$.%
	\COMMENT{AL: added $n$ vertices.}
Suppose that $G$ is a $D$-regular spanning subgraph of $F$. Then the following properties hold:
\begin{itemize}
\item[(i)] there is an $(\eps^*,\eps',K,D_{G'})$-framework%
\COMMENT{Deryk: it seems odd to have $\eps^*$ here and $\eps ^{1/3}$ in the other two bounds? Daniela: yes, but in this way
we don't have to say that a $(\eps^{1/3},\eps',K,D_{G'})$-framework is a $(\eps^*,\eps',K,D_{G'})$-framework later one.}
$(G',A,A_0,B,B_0)$ such that $G'$ is a spanning subgraph of $G$,
$D_{G'} \ge D -2\eps^{1/3}n$ and such that $F$ satisfies (WF5)
(with respect to the partition $A,A_0,B,B_0$);
\item[(ii)] there is a set of $(D-D_{G'})/2 \leq  \eps ^{1/3} n$
edge-disjoint Hamilton cycles in $F-G'$ containing all edges of $G-G'$.
In particular, if $D$ is even then $D_{G'}$ is even.%
\COMMENT{we do need
the explicit statement of the number of ham cycles here. Further, note that this implies that $\delta (G') \geq D_{G'}$}
\end{itemize}
\end{cor}
%%%%%%%%%%%%%%%%%%%%%%%%%%%%%%%%%%%%%%%%%%%%%%%%%%%%%%%%%%%%%%%%%%%%%%%%%%%%%%%
%%%%%%%%%%%%%%%%%%%%%%%%%%%%%%%%%%%%%%%%%%%%%%%%%%%%%%%%%%%%%%%%%%%%%%%%%%%%

\section{Finding path systems which cover all the edges within the classes}\label{findBES}

The purpose of this section is to prove Corollary~\ref{BEScor} which, given a framework $(G,A,A_0,B,B_0)$,
guarantees a set $\cC$ of edge-disjoint Hamilton cycles and a set $\cJ$ of suitable edge-disjoint
$2$-balanced $A_0B_0$-path systems such that the graph $G^*$ obtained from $G$ by deleting the edges in
all these Hamilton cycles and path systems is bipartite with vertex classes $A'$ and $B'$ and
$A_0\cup B_0$ is isolated in $G^*$. Each of the path systems in $\cJ$ will later be extended into
a Hamilton cycle by adding suitable edges between $A$ and~$B$.
The path systems in $\cJ$ will need to be `localized' with respect to a given partition.
We prepare the ground for this in the next subsection.%
\COMMENT{osthus added sentence and repositioned convention on E(S) etc}

Throughout this section, given sets $S,S'\subseteq V(G)$ we often write $E(S)$, $E(S,S')$, $e(S)$ and $e(S,S')$ for
$E_G(S)$, $E_G(S,S')$, $e_G(S)$ and $e_G(S,S')$ respectively.

\subsection{Choosing the partition and the localized slices}\label{partition}  

Let $K,m\in\mathbb{N}$ and $\eps>0$.
A \emph{$(K,m,\eps)$-partition} of a set $V$ of vertices is a partition of $V$ into sets $A_0,A_1,\dots,A_K$
and $B_0,B_1,\dots,B_K$ such that $|A_i|=|B_i|=m$ for all $1 \le i \le K $%
\COMMENT{AL: replaced $i \ge 1$ with $1 \le i \le K $}
 and $|A_0\cup B_0|\le \eps |V|$.
We often write $V_0$ for $A_0\cup B_0$ and think of the vertices in $V_0$ as `exceptional vertices'.
The sets $A_1,\dots,A_K$ and $B_1,\dots,B_K$ are called \emph{clusters} of the $(K,m,\eps_0)$-partition
and $A_0$, $B_0$ are called \emph{exceptional sets}. Unless stated otherwise,
when considering a $(K,m,\eps)$-partition $\mathcal P$ we denote the elements
of $\mathcal P$ by $A_0,A_1,\dots,A_K$ and $B_0,B_1,\dots,B_K$ as above.
Further, we will often write $A$ for $A_1 \cup \dots \cup A_K$ and 
$B$ for $B_1 \cup \dots \cup B_K$.

Suppose that $(G,A,A_0,B,B_0)$ is an $(\eps,\eps',K,D)$-framework with $|G|=n$ and that
$\eps _1, \eps _2 >0$.
We say that $\cP$ is a \emph{$(K,m,\eps,\eps_1,\eps_2)$-partition for $G$} if 
$\cP$ satisfies the following properties:
\begin{itemize}
\item[(P1)] $\cP$ is a $(K,m,\eps)$-partition of $V(G)$
such that the exceptional sets $A_0$ and $B_0$ in the partition $\cP$ are the same as the
sets $A_0$, $B_0$ which are part of the framework $(G,A,A_0,B,B_0)$. In particular, $m=|A|/K=|B|/K$;
\item[(P2)] $d(v, A_i)=(d(v, A)\pm \eps_1 n)/K$  for all $1 \le i\le K$ and $v\in V(G)$;
\item[(P3)] $e(A_i, A_j)=2(e(A) \pm  \eps_2 \max\{n, e(A)\} )/K^2$ for all $1\le i < j\le K$;
\item[(P4)] $e(A_i)=(e(A) \pm \eps_2\max\{n, e(A)\})/K^2$ for all $1 \le i\le K$;
\item[(P5)] $e(A_0, A_i)=(e(A_0, A)\pm \eps_2\max\{n, e(A_0, A)\})/K$ for all $1 \le i\le K$;
\item[(P6)] $e(A_i, B_j)=(e(A, B) \pm 3\eps_2 e(A, B))/K^2$ for all $1 \le i, j \le K$;%
   \COMMENT{Previously had $\max\{n, e(A, B)\}$ instead of just $e(A, B)$. But since we are always assuming that $D$ is
linear, we don't need the max here.}
\end{itemize}
and the analogous assertions hold if we replace $A$ by $B$ (as well as $A_i$ by $B_i$ etc.) in (P2)--(P5).\COMMENT{Check we never need e.g. a mixed condition
like $e(A_0^s, B_i)=...$ in (P5)}

Our first aim is to show that for every framework we can find such a partition with suitable parameters
(see Lemma~\ref{part}). To do this, we need the following lemma.

\begin{lemma}\label{lma:partition2}
Suppose that $0 <  1/n \ll \eps , \eps_1 \ll \eps_2  \ll 1/K \ll 1$, that $r \le 2K$,
that $Km \ge n/4$ and that $r,K,n,m \in \mathbb N$.%
\COMMENT{Andy: added "and that $r,K,n,m \in \mathbb N$."}
Let $G$ and $F$ be graphs on $n$ vertices with $V(G)=V(F)$.
Suppose that there is a vertex partition of $V(G)$ into $U,R_1, \dots, R_r$ with the following properties:
\begin{itemize}
                \item $|U| = K m $.
                \item $\delta(G[U]) \ge \eps n$ or $\Delta(G[U]) \le  \eps n $.
                \item For each $j \le r$ we either have $d_G(u,R_j)\le \eps n$ for all $u\in U$ or $d_G(x,U)\ge \eps n$ for all $x\in R_j$.
\end{itemize}
Then there exists a partition of $U$ into $K$ parts $U_1, \dots, U_K$
satisfying the following properties:
\begin{itemize}
                \item[\rm (i)] $|U_{i}|  = m $ for all $i \le K$.
                \item[\rm (ii)] $d_G(v,U_i)  = (d_G(v,U) \pm \eps_1 n) /K $ for all $v \in V(G)$ and all $i\le K$.
                \item[\rm (iii)] $e_G( U_i, U_{i'} )  = 2 (e_G(U) \pm \eps_2 \max\{  n, e_G(U) \} ) / K^2$ for all $1\le i \neq i' \le K$.
                \item[\rm (iv)] $e_G( U_i)  = (e_G(U) \pm \eps_2 \max\{  n, e_G(U) \} ) / K^2$ for all $ i \le K$.
                \item[\rm (v)] $e_G( U_i, R_{j} ) = ( e_G(U,R_j) \pm   \eps_2 \max \{ n , e_G(U,R_j) \} ) /K  $ for all $i \le K$ and $j \le r$.
                \item[\rm (vi)] $d_F(v,U_i)  = (d_F(v,U) \pm \eps_1 n) /K $ for all $v \in V(F)$ and all $i\le K$.
\end{itemize}
\end{lemma}
 
\begin{proof}
Consider an equipartition $U_1, \dots , U_K$ of $U$ which is chosen uniformly at random.%
     \COMMENT{the previous binomial approach had the problem that at some stage the argument involved a sum over a random variable.
The present one also seems simpler.}
So (i) holds by definition.
Note that for a given vertex $v \in V(G)$, $d_G(v,U_i)$ has the hypergeometric distribution
with mean $d_G(v, U)/K$.
So if $d_G(v, U) \ge \eps_1 n/K$,  Proposition~\ref{chernoff} implies that
$$
                \pr \left( \left| d_G(v,U_{i}) -  \frac{d_G(v, U)}{K} \right|  \ge \frac{\eps_1 d_G(v,U) }{K}\right) 
  \le 2 \exp \left(- \frac{\eps_1^2 d_G(v,U)}{3K} \right)
\le \frac{1}{n^2}.
$$
Thus we deduce that for all $v \in V(G)$ and all $i \le K$,
\begin{align*} %\label{eqn:prdeg2}
                \pr \left( \left| d_G(v,U_{i}) -  d_G(v, U)/K \right|  \ge \eps_1 n /K \right) 
\le 1/n^2.
\end{align*}
Similarly, 
\begin{equation*} %\label{Fdegbound}
\pr \left( \left| d_F(v,U_{i}) -  d_F(v, U)/K \right|  \ge \eps_1 n /K \right) \le 1/n^2.
\end{equation*}
So with probability at least 3/4, both (ii) and (vi) are satisfied.
 
We now consider (iii) and (iv).
Fix $i,i' \le K$. If $i \neq i'$, let $X:=e_G(U_{i}, U_{i'})$.
If $i = i'$, let $X:=2e_G(U_{i})$.
For an edge $f \in E(G[U])$, let $E_f$ denote the event that $f \in E(U_i, U_{i'})$.
So if $f=xy$ and $i \neq i'$, then%
    \COMMENT{to derive these and other expressions, it's best to view the sample space as the set of all
injective mappings from $U$ to $[|U|]$, where
we assign $x$ to $U_i$ if it is mapped to $[(i-1)m+1,im]$. Then each equipartition has the same number of mappings which produce this equipartition.
So as probability space one can consider the uniform distribution over the above mappings.}
\begin{equation} \label{prf}
\pr (E_f) =2 \pr (x \in U_i) \pr (y \in U_{i'} \mid x \in U_i) = 2 \frac{m}{|U|} \cdot  \frac{m}{|U|-1}.
\end{equation}
Similarly, if $f$ and $f'$ are disjoint (that is, $f$ and $f'$ have no common 
endpoint)\COMMENT{B.Cs: the explanation for disjointness is given here} and $i \neq i'$, then%
   \COMMENT{note $\frac{m-1}{m} \frac{|U|}{|U|-2} \le (1-1/m)(1+3/|U|) \le 1$ and similarly for the second fraction.}
\begin{equation} \label{ff'}
\pr ( E_{f'} \mid E_f )=2\frac{m-1}{|U|-2} \cdot \frac{m-1}{|U|-3} \le 
2\frac{m}{|U|} \cdot \frac{m}{|U|-1} =
\pr ( E_{f'}).
\end{equation}
By (\ref{prf}), if $i \not =i'$, we also have
\begin{equation}  \label{exX}
\ex ( X ) = 2 \frac{e_G(U)}{K^2} \cdot \frac{|U|}{|U|-1}= \left( 1 \pm \frac{2}{|U|} \right) \frac{2e_G(U)}{K^2}= \left( 1 \pm \eps_2/4 \right)\frac{2e_G(U)}{K^2}.
\end{equation}
If%
\COMMENT{Deryk divided by 4 in the final equality as we need this in eq. devex}
$f=xy$ and $i = i'$, then  
\begin{equation} \label{prf2}
\pr (E_f) = \pr (x \in U_i) \pr (y \in U_{i} \mid x \in U_i) =  \frac{m}{|U|} \cdot  \frac{m-1}{|U|-1}.
\end{equation}
So if $i = i'$, similarly to (\ref{ff'}) we also obtain $\pr ( E_{f'} \mid E_f ) \le \pr (E_f)$ for disjoint $f$ and $f'$%
\COMMENT{$\pr ( E_{f'} \mid E_f )=\frac{m-2}{|U|-2} \cdot \frac{m-3}{|U|-3} \le
\frac{m}{|U|} \frac{m-1}{|U|-1} =\pr ( E_{f'}).$}
and we obtain the same bound as in (\ref{exX}) on $\ex ( X )$ (recall that $X=2e_G(U_i)$ in this case).%
     \COMMENT{Daniela added brackets}
 
Note that if $i\neq i'$ then%
   \COMMENT{To see the first inequality note that by (\ref{ff'}) we have $\pr(E_{f'} \mid E_f ) -   \pr ( E_{f'} )\le 0$ if $f$ and
$f'$ are disjoint. But there are at most $2 \Delta (G[U])$ edges $f$ which share an endvertex with $f'$ (this includes the
case $f=f'$) and each of these edges $f$ contributes at most 1.}
\begin{eqnarray*}
\textrm{Var} ( X ) 
    & = & \sum_{f \in E(U)}  \sum_{f' \in E(U)}\left(\pr ( E_f \cap E_{f'} )  -   \pr ( E_f ) \pr ( E_{f'} )\right) \\
                & = & \sum_{f \in E(U)} \pr ( E_f ) \sum_{f' \in E(U)}  \left( \pr ( E_{f'} \mid E_f )  -    \pr ( E_{f'} ) \right) \\  
                & \stackrel{(\ref{ff'})}{\le} & \sum_{f \in E(U)} \pr ( E_f )\cdot 2 \Delta (G[U])  
\stackrel{(\ref{exX})}{\le} \frac{3e_G(U)}{K^2}  \cdot 2\Delta (G[U]) 
    \le e_G(U)  \Delta (G[U]).
\end{eqnarray*}
Similarly, if $i=i'$ then
\begin{eqnarray*}
\textrm{Var} ( X ) 
     =  4\sum_{f \in E(U)}  \sum_{f' \in E(U)}\left(\pr ( E_f \cap E_{f'} )  -   \pr ( E_f ) \pr ( E_{f'} ) \right)
                \le e_G(U)  \Delta (G[U]).
\end{eqnarray*}
Let $a:= e_G(U)  \Delta (G[U])$. In both cases, from Chebyshev's inequality,%
   \COMMENT{$\pr \left( | X- \ex(X) |  \ge b  \right) \le Var(X)/b^2$}
it follows that
\begin{align}
\pr \left( | X- \ex(X) |  \ge \sqrt{ a/\eps^{1/2} } \right) \le \eps^{1/2}.  \nonumber
\end{align}
Suppose that $\Delta (G[U]) \le \eps n$.
If we also have have $e_G(U) \le  n$, then 
$\sqrt{ a/\eps^{1/2} } \le   \eps^{1/4}  n \le \eps_2n/2K^2$.
If $e_G(U) \ge n$, then $\sqrt{ a/ \eps^{1/2} } \le \eps^{1/4 }e_G(U) \le \eps_2 e_G(U)/2K^2$.

If we do not have $\Delta (G[U]) \le \eps n$, then our assumptions imply that $\delta (G[U]) \ge \eps n$.
So $\Delta(G[U]) \le n \le \eps e_G(G[U])$ with room to spare.
This in turn means that $\sqrt{ a/ \eps^{1/2} } \le \eps^{1/4} e_G(U) \le \eps_2 e_G(U)/2K^2$.
So in all cases, we have 
\begin{align}
                \pr \left(  \left| X - \ex(X) \right| \ge  \frac{\eps  _2 \max \{ n, e_G(U)\} }{2K^2}\right) \le \eps^{1/2}. \label{eqn:preU}
\end{align}
Now note that by (\ref{exX}) we have
\begin{equation} \label{devex}
\left| \ex(X)- \frac{2e_G(U)}{K^2} \right| \le \frac{\eps  _2 e_G(U) }{2K^2}.
\end{equation}
So (\ref{eqn:preU}) and~(\ref{devex}) together imply that for fixed $i,i'$ the bound in (iii) fails with probability
at most $\eps^{1/2}$. The analogue holds for the bound in~(iv).
By summing over all possible values of $i,i' \le K$, we have that~(iii) and~(iv) hold with probability at least $3/4$.

A similar argument shows that for all $i \le K$ and $j \le r$, we have%
    \COMMENT{
Fix $i \le K$, $j \le r$ and let $X:= e_G(U_i, R_j)$.
For an edge $f$, let $E_f$ denote the event that $f \in E(U_i, R_j)$.
Note $\pr (E_f)=m/|U|=1/K$ and that if $f,f'$ have different endpoints in $U$, we have
$\pr ( E_{f'} \mid E_f )= \frac{m-1}{|U|-1} \le \frac{m}{|U|} =\pr(E_f).
$
Clearly, $\ex ( X ) = e_G(U, R_j) /K$.
Also let $\Delta'(U,R_j))=\max_{u \in U} d (u,R_j)$.
Note that
\begin{eqnarray*}
\textrm{Var} ( X ) 
                & = & \sum_{f \in E(U,R_j)} \pr ( E_f ) \sum_{f' \in E(U,R_j)}  \left( \pr ( E_{f'}) | E_f )  -    \pr ( E_{f'} ) \right) \\       
                & \le & \sum_{f \in E(U,R_j)} \pr ( E_f )\Delta' (U,R_j) \le  \frac{e_G(U,R_j)}{K}  \cdot \Delta' (U,R_j)
\le  e_G(U,R_j)  \Delta' (U,R_j).
\end{eqnarray*}
Let $a:=e_G(U,R_j)\Delta' (U,R_j)$. From Chebyshev's inequality, it follows that
\begin{align}
\pr \left( | X- \ex(X) |  \ge \sqrt{ a/\eps^{1/2} } \right) \le \eps^{1/2}.  \nonumber
\end{align}
Suppose that $\Delta' (U,R_j) \le \eps n$.
If we also have have $e_G(U,R_j) \le  n$, then 
$\sqrt{ a/\eps^{1/2} } \le  \eps^{1/4} n \le \eps_2n/K$.
If $e_G(U,R_j) \ge n$, then $\sqrt{ a/ \eps^{1/2} } \le  \eps^{1/4} e_G(U,R_j) \le \eps_2 e_G(U,R_j)/K$.
If we do not have $\Delta' (U,R_j) \le \eps n$, then our assumptions imply that $d_G(x,U) \ge \eps n$
for every $x \in R_j$.
So $\Delta'(U,R_j) \le |R_j| \le \eps e_G(U,R_j)$ with room to spare.
This in turn means that $\sqrt{ a/ \eps^{1/2} } \le \eps^{1/4} e_G(U,R_j) \le \eps_2 e_G(U,R_j)/K$.
}
%%%%%%comment ends
\begin{equation} \label{eqn:preUR}
                \pr \left(  \left| e_G(U_i, R_j) - \frac{e_G(U,R_j)}{K} \right| \ge  \frac{\eps_2 \max \{ n, e_G(U,R_j)\} }{K} \right)
\le \eps^{1/2}.  
\end{equation}
Indeed, fix $i \le K$, $j \le r$ and let $X:= e_G(U_i, R_j)$.
For an edge $f \in G[U,R_j]$, let $E_f$ denote the event that $f \in E(U_i, R_j)$.
Then $\pr (E_f)=m/|U|=1/K$ and so $\ex ( X ) = e_G(U, R_j) /K$.
The remainder of the argument proceeds as in the previous case
(with slightly simpler calculations).
 
So (v) holds with probability at least 3/4, by summing over all possible values of $i \le K$ and $j \le r$ again.
So with positive probability, the partition satisfies all requirements.
\end{proof}

\begin{lemma}\label{part} Let
$0<1/n\ll \eps\ll \eps'\ll \eps_1\ll \eps_2\ll 1/K\ll 1$.
Suppose that $(G,A,A_0,B,B_0)$ is an $(\eps,\eps',K,D)$-framework with $|G|=n$ and 
$\delta (G) \geq D\ge n/200$. 
Suppose that $F$ is a graph with $V(F)=V(G)$. Then there exists a partition 
$\cP=\{A_0,A_1, \dots , A_K, B_0, B_1, \dots , B_K \}$ of $V(G)$ so that 
\begin{itemize}
\item[(i)] $\cP$ is a $(K,m,\eps,\eps_1,\eps_2)$-partition for $G$.
\item[(ii)] $d_F(v, A_i)=(d_F(v, A)\pm \eps_1 n)/K$ and $d_F(v, B_i)=(d_F(v, B)\pm \eps_1 n)/K$ for all $1 \le i\le K$ and $v\in V(G)$.%
\COMMENT{Deryk rephrased this}
\end{itemize}
\end{lemma}
\proof In order to find the required partitions $A_1,\dots,A_K$ of $A$ and $B_1,\dots,B_K$ of $B$
we will apply Lemma~\ref{lma:partition2} twice, as follows. 
In the first application we let $U:=A,$ $R_1:=A_0$,   $R_2:=B_0$ and
$R_3:= B$. 
 Note that
$\Delta(G[U]) \le  \eps' n $ by (FR5) and $d_G(u,R_j)\le |R_j|\le \eps n\le \eps' n$ for all $u\in U$ and $j=1,2$ by (FR4).
Moreover, (FR4) and (FR7) together imply that $d_G(x,U)\ge D/3\ge \eps' n$ for each $x\in R_3=B$.
Thus we can apply Lemma~\ref{lma:partition2} with $\eps'$ playing the role of $\eps$ to obtain a partition 
$U_1, \dots, U_K$ of $U$. We let $A_i:=U_i$  for all $i\le K.$
Then the $A_i$ satisfy (P2)--(P5) and
\begin{equation}\label{eq:eAiB}
e_G(A_i, B)=(e_G(A, B) \pm \eps_2 \max\{n, e_G(A, B)\})/K=(1\pm \eps_2)e_G(A,B)/K.
\end{equation}
Further, Lemma~\ref{lma:partition2}(vi) implies that
$$d_F(v, A_i)=(d_F(v, A)\pm \eps_1 n)/K$$  for all $1 \le i\le K$ and $v\in V(G)$.

For the second application of Lemma~\ref{lma:partition2} we let $U:=B,$ $R_1:=B_0$, $R_2:=A_0$ 
and $R_j:=A_{j-2}$ for all $3 \le j \le K+2$. As before, $\Delta(G[U]) \le  \eps' n $ by (FR5)
and $d_G(u,R_j)\le \eps n\le \eps' n$ for all $u\in U$ and  $j=1,2$ by (FR4).
Moreover, (FR4) and (FR7) together imply that $d_G(x,U)\ge D/3\ge \eps' n$ for all $3 \le j \le K+2$ and each $x\in R_j=A_{j-2}$.
Thus we can apply Lemma~\ref{lma:partition2} with $\eps'$ playing the role of $\eps$ to obtain a partition 
$U_1, \dots, U_K$ of $U$. Let $B_i:=U_i$  for all $i \le K.$ Then the $B_i$ satisfy (P2)--(P5)
with $A$ replaced by~$B$, $A_i$ replaced by $B_i$, and so on. Moreover, for all $1 \leq i,j \leq K$,
\begin{eqnarray*}
e_G(A_i, B_j) & = & (e_G(A_i, B) \pm \eps_2 \max\{n, e_G(A_i, B)\})/K\\
& \stackrel{(\ref{eq:eAiB})}{=} & ((1\pm \eps_2)e_G(A,B) \pm \eps_2 (1+ \eps_2)e_G(A,B))/K^2\\
& = & (e_G(A, B) \pm 3\eps_2 e_G(A, B))/K^2,
\end{eqnarray*}
i.e.~(P6) holds. Since clearly (P1) holds as well, $A_0,A_1,\dots,A_K$ and $B_0,B_1,\dots,B_K$
together form a $(K,m,\eps,\eps_1,\eps_2)$-partition for $G$.
Further, Lemma~\ref{lma:partition2}(vi) implies that
$$d_F(v, B_i)=(d_F(v, B)\pm \eps_1 n)/K$$  for all $1 \le i\le K$ and $v\in V(G)$.
\endproof

The next lemma gives a decomposition of $G[A']$ and $G[B']$ into suitable smaller edge-disjoint subgraphs $H_{ij}^A$ and $H_{ij}^B$.
We say that the graphs $H_{ij}^A$ and $H_{ij}^B$ guaranteed by Lemma~\ref{rnd_slice} are \emph{localized slices} of~$G$.
Note that the order of the indices $i$ and $j$ matters here, i.e.~$H^A_{ij}\neq H^A_{ji}$.
Also, we allow $i=j$. %
\COMMENT{Deryk added this and osthus added lower bound $1 \le i,j\le$}

\begin{lemma}\label{rnd_slice}
Let $0<1/n\ll \eps\ll \eps'\ll \eps_1\ll \eps_2\ll 1/K\ll 1$.
Suppose that $(G,A,A_0,B,B_0)$ is an $(\eps,\eps',K,D)$-framework with $|G|=n$ and $D\ge n/200$.
Let $A_0,A_1,\dots,A_K$ and $B_0,B_1,\dots,B_K$ be a $(K,m,\eps,\eps_1,\eps_2)$-partition for $G$. 
Then for all $1 \le i,j\le K$ there are graphs $H_{ij}^A$ and $H_{ij}^B$ with the following properties:
\begin{enumerate}
\item[(i)] $H_{ij}^A$ is a spanning subgraph of  $G[A_0, A_i\cup A_j]\cup G[A_i,A_j]\cup G[A_0]$; 
\item[(ii)] The sets $E(H_{ij}^A)$ over all $1 \le i,j\le K$ form a partition of the edges of $G[A']$;
\item[(iii)] $e(H_{ij}^A)=(e(A') \pm 9\eps_2\max\{n, e(A')\})/K^2$ for all $1 \le i, j \le K$;
\item[(iv)] $e_{H_{ij}^A}(A_0, A_i\cup A_j)=(e(A_0, A)\pm 2\eps_2\max\{n, e(A_0, A)\})/K^2$ for all $1 \le i, j \le K$;
\item[(v)] $e_{H_{ij}^A}(A_i, A_j)=(e(A)\pm  2\eps_2 \max\{n,e(A)\})/K^2$ for all $1 \le i, j \le K$;
\item[(vi)] %Given a vertex $v \in A'$ and a set $S\subseteq V(H^A_{ij})$, let $d_{ij}(v, S):=d_{H_{ij}}(v, S)$.
For all $1 \le i, j \le K$ and all $v\in A_0$ we have $d_{H_{ij}^A}(v)=d_{H_{ij}^A}(v, A_i\cup A_j)+d_{H_{ij}^A}(v, A_0) =(d(v, A)\pm 4\eps_1 n)/K^2$.
\end{enumerate} 
The analogous assertions hold if we replace $A$ by $B$, $A_i$ by $B_i$, and so on.
\end{lemma}
\proof In order to construct the graphs $H^A_{ij}$
we perform the following procedure:
\begin{itemize}
\item Initially each $H_{ij}^A$ is an empty graph with vertex set $A_0\cup A_i\cup A_j$.
\item For all $1 \le i\le K$ choose a random partition $E(A_0, A_i)$ into $K$ sets $U_j$ of equal size
and let $E(H_{ij}^A):=U_j$. 
(If $E(A_0, A_i)$ is not divisible by $K$, first distribute up to $K-1$ edges arbitrarily among the $U_j$ to achieve divisibility.)%
\COMMENT{We need to mention divisibility as the number of edges may be small.}
\item For all $i\le K$, we add all the edges in $E(A_i)$ to $H_{ii}^A$. 
\item For all $i, j\le K$ with $i\neq j$, half of the edges in $E(A_i, A_j)$ are added to $H_{ij}^A$ and
the other half is added to  $H_{ji}^A$ (the choice of the edges is arbitrary). 
\item The edges in $G[A_0]$ are distributed equally amongst the $H^A _{ij}$.
(So $e_{H^A _{ij}} (A_0) = e(A_0)/K^2 \pm 1$.)
\end{itemize}
Clearly, the above procedure ensures that properties~(i) and (ii) hold.
(P5) implies (iv) and (P3) and (P4) imply (v).%
\COMMENT{The `2' could actually be omitted as the partition is now exact (if divisibility holds) all the numbers are large, so divisibility doesn't affect the numbers.
But it's probably easier to see as it is.}

Consider any $v \in A_0$.
To prove (vi), note that we may assume that $d(v,A) \ge  \eps_1 n/K^2$.
Let $X:=d_{H_{ij}^A}(v, A_i \cup A_j)$. Note that (P2) implies that%
    \COMMENT{Daniela replaced lots of $\ex [X]$ by $\ex (X)$}
$\ex (X)=(d(v, A)\pm 2\eps_1 n)/K^2$ and note that $\ex(X) \le n$. 
So the Chernoff-Hoeffding bound for the hypergeometric distribution in Proposition~\ref{chernoff} implies that
$$
\prob (|X-\ex (X)| > \eps_1 n/K^2) \le \prob (|X- \ex (X)| > \eps_1 \ex (X)/K^2 ) \le 2 e^{-\eps_1^2 \ex (X)/3K^4} \le 1/n^2. 
$$
Since $d_{H_{ij}^A}(v, A_0)\leq |A_0|\leq \eps _1 n/K^2$,
 a union bound implies the desired result.
Finally, observe that for any $a,b_1,\dots,b_4>0$, we have
$$
\sum_{i=1}^4 \max\{a,b_i\}\le 4\max\{a,b_1,\dots,b_4\} \le 4\max\{a,b_1+\dots +b_4\}.
$$
So (iii) follows from~(iv),~(v) and the fact that $e_{H^A _{ij}} (A_0) = e(A_0)/K^2 \pm 1$.
\endproof
Note that the construction implies that if $i\neq j$, then $H^A_{ij}$ will  contain edges between $A_0$ and $A_i$
but not between $A_0$ and $A_j$. However, this additional information is not needed in the subsequent argument.

\subsection{Decomposing the localized slices}\label{sec:slice} 
Suppose that $(G,A,A_0,B,B_0)$ is an $(\eps,\eps',K,D)$-framework. 
Recall that $a=|A_0|$, $b=|B_0|$ and $a\ge b$. Since $G$ is $D$-balanced by~(FR2),
we have $e(A')-e(B')=(a-b)D/2.$ 
So there are an integer $q \ge -b$ and a constant $0\le c<1$ such that
\begin{equation} \label{aqc}
e(A')=(a+q+c)D/2 \ \ \mbox{ and } \ \ e(B')=(b+q+c)D/2.
\end{equation} 
The aim of this subsection is to prove Lemma~\ref{matchingdec}, 
which guarantees a decomposition of each localized slice $H_{ij}^A$ into path systems
(which will be extended into $A_0B_0$-path systems in Section~\ref{besconstruct})%
\COMMENT{osthus modified this. Andy: deleted full stop}
and a sparse (but not too sparse) leftover graph $G_{ij}^A$.%
\COMMENT{AL: added `localized $A_0B_0$-'}

The following two results will be used in the proof of Lemma~\ref{matchingdec}.
\begin{lemma}\label{simple}
Let $0<1/n \ll \alpha, \beta , \gamma $ so that $\gamma<1/2$.  Suppose that $G$ is a graph on $n$ vertices such that
$\Delta (G) \leq \alpha n$ and $e(G) \geq \beta n$. Then $G$ contains a spanning subgraph
$H$ such that $e(H)= \lceil (1- \gamma) e(G)\rceil$ and $\Delta (G-H) \leq 6 \gamma \alpha n/5$.
\end{lemma}
\proof
Let $H'$ be a spanning subgraph of $G$ such that
\begin{itemize}
\item $\Delta (H') \leq 6\gamma \alpha n/5$;
\item $e(H') \geq \gamma e(G)$.
\end{itemize}
To see that such a graph $H'$ exists, consider a random subgraph of $G$ obtained by including each edge of
$G$ with probability $11 \gamma /10$.
Then $\mathbb E (\Delta (H')) \leq 11 \gamma \alpha n/10$ and  
$\mathbb E (e (H')) = 11 \gamma e(G)/10$. Thus applying Proposition~\ref{chernoff} we have that,
with high probability, $H'$ is as desired.%
   \COMMENT{We apply Chernoff in the usual way here.
For example, consider any $v \in V(G)$. If $d_G (v) \leq 6\gamma \alpha n/5$ then certainly $d_{H'} (v) \leq 6 \gamma \alpha n/5$.
Otherwise,
\begin{align*}
\mathbb P (d_{H'} (v) \geq 6\gamma \alpha n/5) & \leq \mathbb P \left(|d_{H'} (v) -\frac{11\gamma}{10} d_G (v)|\geq \frac{1}{11} 
\left ( \frac{11\gamma}{10} d_G (v) \right ) \right ) \\ &
\leq 2 \exp{\{-\frac{1}{363} \times \frac{11\gamma}{10} d_G (v)\} } \leq 2 
\exp{\{\frac{-11\gamma}{3630} \times 6\gamma \alpha n/5\} } \ll 1/n.
\end{align*}}

Define $H$ to be a spanning subgraph of $G$ such that $H\supseteq G-H'$%
	\COMMENT{AL: replaced $\subseteq$ with $\supseteq$.}
 and $e(H)= \lceil (1- \gamma) e(G)\rceil$.
Then $\Delta (G-H)\leq \Delta (H') \leq 6 \gamma \alpha n/5$, as required.
\endproof

\begin{lemma}\label{splittrick}
Suppose that $G$ is a graph such that $\Delta (G) \leq D-2$ where $D \in \mathbb N$ is even.
Suppose $A_0,A$ is a partition of $V(G)$ such that $d_G(x) \leq D/2-1$ for all $x \in A$ and
$\Delta (G[A_0]) \leq D/2-1$. 
Then $G$ has a decomposition into $D/2$ edge-disjoint path
systems $P_1, \dots , P_{D/2}$ such that the following conditions hold:
\begin{itemize}
\item[(i)] For each $i \leq D/2$, any internal vertex on a path in $P_i$ lies in $A_0$;
\item[(ii)] $|e(P_i)-e(P_j)|\leq 1 $ for all $i,j \leq D/2$.
\end{itemize}
\end{lemma}
\proof
Let $G_1$ be a maximal spanning subgraph of $G$ under the constraints that $G[A_0] \subseteq G_1$
and $\Delta (G_1) \leq D/2-1$.
Note that $G[A_0]\cup G[A] \subseteq G_1$. Set $G_2 := G-G_1$. So $G_2$ only contains $A_0A$-edges.
Further, since $\Delta (G) \leq D-2$, the maximality of $G_1$ implies that $\Delta (G_2) \leq
D/2-1$.

Define an auxiliary graph $G'$, obtained from $G_1$ as follows:
write $A_0=\{a_1, \dots , a_m\}$. Add a new vertex set $A'_0=\{ a'_1, \dots , a'_m\}$ to $G_1$.
For each $i \leq m$ and $x \in A$, we add an edge between $a'_i$ and $x$ if and only if $a_i x$
is an edge in $G_2$.

Thus $G'[A_0 \cup A]$ is isomorphic to $G_1$ and $G'[A'_0 , A]$ is isomorphic to $G_2$.
By construction and since $d_G (x) \leq D/2-1$ for all $x \in A$,  we have that $\Delta (G') \leq D/2-1$. Hence, Proposition~\ref{basic_matching_dec} implies
that $E(G')$ can be decomposed into $D/2$ edge-disjoint matchings $M_1, \dots , M_{D/2}$ such that
$||M_i|-|M_j||\leq 1$ for all $i,j \leq D/2$.

By identifying each vertex $a'_i \in A'_0$ with the corresponding vertex $a_i \in A_0$,
$M_1, \dots , M_{D/2}$ correspond to edge-disjoint subgraphs $P_1, \dots, P_{D/2}$ of $G$ such that
\begin{itemize}
\item $P_1, \dots, P_{D/2}$ together cover all the edges in $G$;
\item  $|e(P_i)-e(P_j)|\leq 1 $ for all $i,j \leq D/2$.
\end{itemize}
Note that $d_{M_i} (x) \leq 1$ for each $x \in V(G')$. Thus $d_{P_i} (x) \leq 1$ for each $x \in A$ and $d_{P_i} (x) \leq 2$ for each $x \in A_0$. This implies that any cycle in $P_i$ must
lie in $G[A_0]$. However, $M_i$ is a matching and $G'[A'_0]\cup G'[A_0,A'_0]$ contains no edges.%
   \COMMENT{Daniela replaced $G'[A'_0]$ by $G'[A'_0]\cup G'[A_0,A'_0]$}
Therefore, $P_i$ contains no cycle,
and so $P_i$ is a path system such that any internal vertex on a path in $P_i$ lies in $A_0$.
Hence $P_1, \dots , P_{D/2}$ satisfy (i) and (ii).
\endproof

\begin{lemma} \label{matchingdec}
Let $0<1/n\ll \eps\ll \eps'\ll \eps_1\ll \eps_2\ll \eps_3\ll \eps_4\ll 1/K\ll 1$.
Suppose that $(G,A,A_0,B,B_0)$ is an $(\eps,\eps',K,D)$-framework with $|G|=n$ and $D\ge n/200$.
Let $A_0,A_1,\dots,A_K$ and $B_0,B_1,\dots,B_K$ be a $(K,m,\eps,\eps_1,\eps_2)$-partition for $G$. 
Let $H_{ij}^A$ be a localized slice of $G$ as guaranteed by Lemma~\ref{rnd_slice}.
Define $c$ and $q$ as in~(\ref{aqc}).
Suppose that $t:=(1-20\eps_4)D/2K^2\in \mathbb{N}$.%
   \COMMENT{Later one we will also require that $t$ is divisible by $K^2$.}
If $e(B')\geq \eps _3 n$, set $t^*$ to be the largest integer which is at most $ct$ and is divisible by $K^2$. 
Otherwise, set $t^*:=0$.
Define
$$
\ell _a :=
\left\{
	\begin{array}{ll}
		0 & \mbox{if } e(A') < \eps _3 n ;\\
        a-b & \mbox{if } e(A') \geq \eps _3 n \mbox{ but } e(B') < \eps_ 3 n;\\
		a+q+c & \mbox{otherwise}  
	\end{array}
\right.
$$
and 
$$
\ell _b :=
\left\{
	\begin{array}{ll}
		0 & \mbox{if } e(B') < \eps _3 n ;\\
		b+q+c & \mbox{otherwise.}  
	\end{array}
\right.
$$
Then 
$H_{ij}^A$ has a decomposition into $t$ edge-disjoint path systems $P_1,\dots,P_{t}$ and a spanning subgraph $G_{ij}^A$
with the following properties:
\begin{itemize}
\item[(i)] For each $s \leq t$, any internal vertex on a path in $P_s$ lies in $A_0$;
\item[(ii)] $e(P_1)=\dots = e(P_{t^*}) =\lceil \ell_a \rceil$ and $e(P_{t^*+1}) = \dots =e(P_{t})= \lfloor \ell_a \rfloor$; 
\item[(iii)] $e(P_s)\le \sqrt{\eps} n$ for every $s\le t$;
\item[(iv)] $\Delta(G_{ij}^A)\le 13\eps_4 D/K^2$. %and $e(G_{ij}^A)\ge 9\eps_4\ell_a D /K^2$.
\end{itemize}
The analogous assertion (with $\ell_a$ replaced by $\ell_b$ and $A_0$ replaced by $B_0$) holds for each localized slice~$H_{ij}^B$ of~$G$.
Furthermore, $\lceil \ell _a \rceil -\lceil \ell _b \rceil= \lfloor \ell _a \rfloor -\lfloor \ell _b \rfloor =a-b$.
\end{lemma}
\proof Note that~(\ref{aqc}) and  (FR3) together imply that $\ell_a D/2\leq (a+q+c)D/2= e(A') \le \eps n^2$ and so
$\lceil \ell_a\rceil \le \sqrt{\eps} n$. Thus~(iii) will follow from~(ii). So it remains to prove~(i), (ii) and~(iv). We split the proof into three cases.

\medskip

\noindent
{\bf Case 1. $e(A') < \eps _3 n$}

(FR2) and (FR4) imply that $e(A')-e(B') =(a-b)D/2 \geq 0$. So $e(B') \leq e(A') < \eps _3 n$.
Thus $\ell _a =\ell _b =0$. Set $G^A _{ij}:=H^A _{ij}$ and $G^B _{ij}:=H^B _{ij}$.
Therefore, (iv) is satisfied as $\Delta(H_{ij}^A) \leq e(A') < \eps _3 n \leq 13\eps_4 D/K^2$.
Further, (i) and (ii) are vacuous (i.e. we set each $P_s$ to be the empty graph on $V(G)$). 

Note that $a=b$ since otherwise $a>b$ and therefore (FR2) implies that $e(A') \geq (a-b)D/2 \geq D/2 > \eps_3 n$, a contradiction.
Hence, $\lceil \ell _a \rceil -\lceil \ell _b \rceil= \lfloor \ell _a \rfloor -\lfloor \ell _b \rfloor =0=a-b$.

\medskip

\noindent
{\bf Case 2.} $e(A') \geq \eps _3 n$ and $e(B') < \eps _3 n$

Since $\ell _b=0$ in this case, we set $G^B _{ij}:=H^B _{ij}$ and each $P_s$ to be the empty graph on $V(G)$. Then as in Case~1, 
(i), (ii) and (iv) are satisfied with respect to $H^B _{ij}$. Further, clearly 
$\lceil \ell _a \rceil -\lceil \ell _b \rceil= \lfloor \ell _a \rfloor -\lfloor \ell _b \rfloor =a-b$.

Note that $a>b$ since otherwise $a=b$ and thus $e(A')=e(B')$ by (FR2), a contradiction
to the case assumptions. Since $e(A') -e(B') =(a-b)D/2$ by (FR2), Lemma~\ref{rnd_slice}(iii)
implies that 
\begin{align}
\nonumber
e (H^A _{ij})& \geq (1-9 \eps _2) {e(A')}/{K^2}-9 \eps _2 n /K^2  \geq 
(1-9 \eps _2) (a-b)D/(2K^2) - 9 \eps _2 n /K^2 \\ 
& \geq (1- \eps _3)(a-b) D/(2K^2) > (a-b)t. \label{gambound}
\end{align}
%(The penultimate inequality follows since $\eps _3 (a-b) D/4 \geq  \eps _3 D/4 >9 \eps _2 n$.)
Similarly, Lemma~\ref{rnd_slice}(iii) implies that 
\begin{align}\label{gamup}
e (H^A _{ij}) \leq (1+\eps _4)(a-b)D/(2K^2). 
\end{align}
Therefore, (\ref{gambound}) implies that there exists a constant $\gamma >0$ such that
$$(1-\gamma)e (H^A _{ij}) =(a-b)t.$$
Since $(1-19\eps _4 )(1-\eps _3) >(1-20\eps _4)$, (\ref{gambound}) implies that
$\gamma > 19 \eps _4 \gg 1/n$. Further, since $(1+\eps _4)(1-21 \eps _4) < (1-20 \eps _4)$,
(\ref{gamup}) implies that $\gamma < 21 \eps _4$.

Note that (FR5), (FR7) and Lemma~\ref{rnd_slice}(vi) imply that 
\begin{align}\label{boundup}
\Delta (H^A _{ij}) \leq (D/2+5 \eps _1 n)/K^2.
\end{align}
Thus Lemma~\ref{simple} implies that $H^A _{ij}$ contains a spanning subgraph $H$ such that
$e(H)=(1- \gamma ) e(H^A _{ij})=(a-b)t$ and 
$$\Delta (H^A _{ij} -H) \leq 6 \gamma (D/2+5 \eps _1 n)/(5K^2)
\leq 13 \eps _4 D/K^2,$$
where the last inequality follows since $ \gamma <21 \eps _4$ and $\eps _1 \ll 1$.
Setting $G^A _{ij}:=H^A _{ij} -H$ implies that (iv) is satisfied.

Our next task is to decompose $H$ into $t$ edge-disjoint path systems so that (i) and (ii) 
are satisfied. Note that (\ref{boundup}) implies that
$$\Delta (H) \leq \Delta (H^A _{ij}) \leq (D/2+5 \eps _1 n)/K^2 <2t-2.$$
Further, (FR4) implies that $\Delta (H[A_0]) \leq |A_0| \leq \eps n < t-1$ and 
(FR5) implies that $d_H (x ) \leq \eps ' n <t-1$ for all $x \in A$. 
Since $e(H)=(a-b)t$,  Lemma~\ref{splittrick} implies that $H$ has a decomposition into $t$ edge-disjoint path systems $P_1,\dots,P_{t}$ satisfying (i) and so that $e(P_s)=a-b=\ell _a$ for all
$ s\leq t$. In particular, (ii) is satisfied.

\medskip

\noindent
{\bf Case 3.} $e(A'), e(B') \geq \eps _3 n$ 

By definition of $\ell _a $ and $\ell _b$, we have that
$\lceil \ell _a \rceil -\lceil \ell _b \rceil= \lfloor \ell _a \rfloor -\lfloor \ell _b \rfloor =a-b$.
Notice that since $e(A') \geq \eps _3 n$ and $\eps _2 \ll \eps _3$, certainly $\eps _3 e(A')/(2K^2) > 9 \eps _2 n/K^2$.
Therefore, Lemma~\ref{rnd_slice}(iii) implies that 
\begin{align}
\nonumber
e (H^A _{ij})& \geq (1-9 \eps _2) {e(A')}/{K^2}-9 \eps _2 n /K^2 \\ & \geq 
(1- \eps _3) e(A')/K^2 \label{gambound2} \\ 
& \geq \eps _3 n/(2K^2). \nonumber
\end{align}
Note that $1/n \ll \eps _3/(2K^2)$.\COMMENT{split calculation up into 2 parts here since it is important to show that
there are many edges in $H^A_{ij}$ so that we can apply Lemma~\ref{simple}.}
Further, (\ref{aqc}) and (\ref{gambound2}) imply that
\begin{align}
\nonumber
e (H^A _{ij})& \geq 
(1- \eps _3) e(A')/K^2  \\ 
& = (1- \eps _3) (a+q+c)D/(2K^2) > (a+q)t+t^* \label{gambound3} . 
\end{align}
Similarly, Lemma~\ref{rnd_slice}(iii) implies that 
\begin{align}\label{gambound4}
e (H^A _{ij}) \leq (1+\eps _3) (a+q+c) D/(2K^2).
\end{align}
By (\ref{gambound3}) there exists a constant $\gamma >0$ such that
$$(1-\gamma)e (H^A _{ij}) =(a+q)t+t^*.$$
Note that  (\ref{gambound3}) implies that $1/n \ll 19 \eps _4 < \gamma$ and (\ref{gambound4}) implies that $\gamma < 21 \eps _4$.%
\COMMENT{the upper bound is not completely obvious this time around since $t^*$ is a `floor': Indeed, since $e(A') \geq \eps _3n$,
(\ref{gambound4}) implies that $(a+q+c)D/2K^2 \geq \eps _3 n/K^2$. Thus,
$$(1-21 \eps _4)(1+\eps _3) (a+q+c) D/(2K^2) < (1-20 \eps _4) (a+q+c) D/(2K^2)-K^2 \leq (a+q)t+t^*.$$ }
Moreover, as in Case~2, (FR5), (FR7) and Lemma~\ref{rnd_slice}(vi) together show that 
\begin{align}\label{boundup1}
\Delta (H^A _{ij}) \leq (D/2+5 \eps _1 n)/K^2.
\end{align}
Thus (as in Case~2 again), Lemma~\ref{simple} implies that $H^A _{ij}$ contains a spanning subgraph $H$ such that
$e(H)=(1- \gamma ) e(H^A _{ij})=(a+q)t+t^* $ and%
\COMMENT{Deryk shortened calculation and referred to case 2:
where the last inequality follows since $ \gamma <21 \eps _4$ and $\eps _1 \ll 1$.}
$$\Delta (H^A _{ij} -H) \leq 6 \gamma (D/2+5 \eps _1 n)/(5K^2)
\leq 13 \eps _4 D/K^2.$$
Setting $G^A _{ij}:=H^A _{ij} -H$ implies that (iv) is satisfied.
Next we decompose $H$ into $t$ edge-disjoint path systems so that (i) and (ii) 
are satisfied. Note that (\ref{boundup1}) implies that
$$\Delta (H) \leq \Delta (H^A _{ij}) \leq (D/2+5 \eps _1 n)/K^2 <2t-2.$$
Further, (FR4) implies that $\Delta (H[A_0]) \leq |A_0| \leq \eps n < t-1$ and 
(FR5) implies that $d_H (x ) \leq \eps ' n <t-1$ for all $x \in A$. 
Since $e(H)=(a+q)t+t^*$,  Lemma~\ref{splittrick} implies that $H$ has a decomposition into $t$ edge-disjoint path systems $P_1,\dots,P_{t}$ satisfying (i) and (ii). 
An identical argument implies that (i), (ii) and (iv) are satisfied with respect to $H^B _{ij}$ also. 
\endproof

%%%%%%%%%%%%%%%%%%%%%%%%%%%%%%%%%%%%%%%%%%%%%%%%%%%%%%%%%%%%%%%%%%%%

\subsection{Decomposing the global graph}\label{sec:global} 
Let $G_{glob}^A$  be the union of the graphs $G_{ij}^A$ guaranteed by Lemma~\ref{matchingdec} over all $1\le i,j\le K$.  
Define $G_{glob}^B$ similarly.
The next lemma gives a decomposition of both $G_{glob}^A$ and $G_{glob}^B$ into suitable path systems.
Properties~(iii) and~(iv) of the lemma guarantee that one can pair up each such path system $Q_A\subseteq G_{glob}^A$ with
a different path system $Q_B\subseteq G_{glob}^B$ such that $Q_A \cup Q_B$ is $2$-balanced 
(in particular $e(Q_A)-e(Q_B)=a-b$). This
property will then enable us to apply Lemma~\ref{extendpaths} to 
extend $Q_A\cup Q_B$ into a Hamilton cycle using only edges between $A'$ and $B'$.

\begin{lemma}\label{G-glob} 
Let $0<1/n\ll \eps\ll \eps'\ll \eps_1\ll \eps_2\ll \eps_3\ll \eps_4\ll 1/K\ll 1$.
Suppose that $(G,A,A_0,B,B_0)$ is an $(\eps,\eps',K,D)$-framework with $|G|=n$ and such that $D\ge n/200$
and $D$ is even.
Let $A_0,A_1,\dots,A_K$ and $B_0,B_1,\dots,B_K$ be a $(K,m,\eps,\eps_1,\eps_2)$-partition for $G$. 
Let $G_{glob}^A$  be the union of the graphs $G_{ij}^A$ guaranteed by Lemma~\ref{matchingdec} over all $1\le i,j\le K$.  
Define $G_{glob}^B$ similarly. Suppose that $k := 10 \eps_4 D\in\mathbb{N}$. Then the following properties hold:
\begin{itemize} 
\item[(i)] There is an integer $q'$ and a real number  $0\le c'<1$ so that
$e(G_{glob}^A)=(a+q'+c')k$ and  $e(G_{glob}^B)=(b+q'+c')k$.
\item[(ii)] $\Delta(G_{glob}^A), \Delta(G_{glob}^B)< 3k/2$.
\item[(iii)] Let $k^*:=c'k$. Then $G_{glob}^A$ has a decomposition into $k^*$ path systems, each containing $a+q'+1$ edges, and $k-k^*$ path systems,
each containing $a+q'$ edges. Moreover, each of these $k$ path systems~$Q$ satisfies $d_Q(x)\le 1$ for all $x\in A$.
\item[(iv)] $G_{glob}^B$ has a decomposition into $k^*$ path systems, each containing $b+q'+1$ edges, and $k-k^*$ path systems,
each containing $b+q'$ edges. Moreover, each of these $k$ path systems~$Q$ satisfies $d_Q(x)\le 1$ for all $x\in B$.
\item[(v)] Each of the path systems guaranteed in~(iii) and~(iv) contains at most $\sqrt{\eps}n$ edges.
\end{itemize}
\end{lemma}
Note%
\COMMENT{osthus added sentence. Andy: changed A to B}
that in Lemma~\ref{G-glob} and several later statements the parameter $\eps_3$ is implicitly defined by the application
of Lemma~\ref{matchingdec} which constructs the graphs $G_{glob}^A$ and $G_{glob}^B$.
\proof
Let $t^*$ and $t$ be as defined in Lemma~\ref{matchingdec}. Our first task is to show that (i) is satisfied.
If $e(A'),e(B') < \eps _3 n$ then $G^A _{glob}=G[A']$ and $G^B _{glob} =G[B']$. Further, $a=b$ in this case since 
otherwise (FR4) implies that $a>b$ and so (FR2) yields that $e(A') \geq (a-b)D/2\geq D/2 > \eps _3 n$, a contradiction.
Therefore, (FR2) implies that
\begin{align*}
e(G_{glob}^A)  -e(G_{glob}^B) & =e(A')-e(B') {=} (a-b)D/2=0=(a-b)k.
\end{align*}

If $e(A')\ge \eps _3 n$ and $e(B') < \eps _3 n$ then $G^B _{glob} =G[B']$. Further, $G^A _{glob}$ is obtained from
$G[A']$ by removing $tK^2$ edge-disjoint path systems, each of which contains precisely $a-b$ edges.
Thus (FR2) implies that 
$$e(G_{glob}^A)  -e(G_{glob}^B)  =e(A')-e(B')-tK^2(a-b) =(a-b)(D/2-tK^2)=(a-b)k.$$

Finally, consider the case when $e(A'), e(B') > \eps _3 n$.
Then $G^A _{glob}$ is obtained from $G[A']$ by removing $t^*K^2$ edge-disjoint path systems, each of which contain exactly $a+q+1$ edges,
and by removing $(t-t^*)K^2$ edge-disjoint  path systems, each of which contain exactly $a+q$ edges.
Similarly, $G^B _{glob}$ is obtained from $G[B']$ by removing $t^*K^2$ edge-disjoint path systems, each of which contain exactly $b+q+1$ edges,
and by removing $(t-t^*)K^2$ edge-disjoint path systems, each of which contain exactly $b+q$ edges. So (FR2) implies that
\begin{align*}
e(G_{glob}^A)  -e(G_{glob}^B)  =e(A')-e(B') -(a-b)tK^2= (a-b)k.
\end{align*}
Therefore, in every case,
\begin{align}\label{abk}
e(G_{glob}^A)  -e(G_{glob}^B)  =(a-b)k.
\end{align}
Define the integer $q'$ and $0 \le c' <1$ by $e(G_{glob}^A)=(a+q'+c')k$. 
Then~(\ref{abk}) implies that $e(G_{glob}^B)=(b+q'+c')k$. This proves (i).
To prove (ii), note that Lemma~\ref{matchingdec}(iv) implies that
$\Delta(G_{glob}^A)\le 13\eps_4 D<3k/2$ and similarly $\Delta(G_{glob}^B)<3k/2$. 

Note that (FR5) implies that $d_{G_{glob}^A} (x) \leq \eps 'n < k-1$ for all $x \in A$ 
and $\Delta (G_{glob}^A[A_0]) \leq |A_0| \leq \eps n < k-1$.
Thus Lemma~\ref{splittrick} together with (i) implies that (iii) is satisfied.
(iv) follows from Lemma~\ref{splittrick} analogously.

(FR3) implies that $e(G^A_{glob}) \leq e_G (A') \leq \eps n^2$ and $e(G^B_{glob}) \leq e_G (B') \leq \eps n^2$.
Therefore, each path system from (iii) and (iv) contains at most $\lceil \eps n^2 /k \rceil \leq \sqrt{\eps} n$ edges.%
   \COMMENT{Daniela replaced "has size" by "contains ... edges"}
So (v) is satisfied.
\endproof

We say that a path system $P\subseteq G[A']$ is \emph{$(i,j,A)$-localized} if%
\COMMENT{ANDREW changed the definition here. This is needed because, in proof of
Lemma~\ref{balmatchextend} we need the fact that edges e.g. don't lie in $G[A_i]$ so that
condition (BES2) is satisfied.}
\begin{itemize}
\item[(i)] $E(P) \subseteq E(G[A_0, A_i \cup A_j]) \cup E(G[A_i,A_j])\cup E(G[A_0])$;
\item[(ii)] Any internal vertex on a path in $P$ lies in $A_0$.
\end{itemize}
We%
 \COMMENT{Previously had the stronger condition that $M\subseteq H^A_{ij}$ instead.}
introduce an analogous notion of \emph{$(i,j,B)$-localized} for path systems $P\subseteq G[B']$.

The following result is a straightforward consequence of Lemmas~\ref{rnd_slice}, \ref{matchingdec} and \ref{G-glob}.
It gives a decomposition of $G[A'] \cup G[B']$ into pairs of paths systems so that most of these are localized and so that each pair
can be extended into a Hamilton cycle by adding $A'B'$-edges. %
\COMMENT{Deryk added sentnce}
\begin{cor} \label{finalcor}
Let $0<1/n\ll \eps\ll \eps'\ll \eps_1\ll \eps_2\ll \eps_3\ll \eps_4\ll 1/K\ll 1$.
Suppose that $(G,A,A_0,B,B_0)$ is an $(\eps,\eps',K,D)$-framework with $|G|=n$ and such that $D\ge n/200$
and $D$ is even.
Let $A_0,A_1,\dots,A_K$ and $B_0,B_1,\dots,B_K$ be a $(K,m,\eps,\eps_1,\eps_2)$-partition for $G$. 
Let $t_K:=(1-20\eps_4)D/2K^4$ and $k:=10\eps_4 D$. Suppose that $t_K\in\mathbb{N}$.
Then there are $K^4$ sets $\cM_{i_1i_2i_3i_4}$, one for each $1 \le i_1,i_2,i_3,i_4 \le K$, such that each $\cM_{i_1i_2i_3i_4}$
consists of $t_K$ pairs of path systems and satisfies the following properties:
\begin{itemize} 
\item[(a)]  Let $(P,P')$ be a pair of path systems which forms an element of $\cM_{i_1i_2i_3i_4}$. Then
\begin{itemize}
\item[(i)] $P$ is an $(i_1,i_2,A)$-localized path system and $P'$ is an $(i_3,i_4,B)$-localized path system;
\item[(ii)] $e(P)-e(P')=a-b$;
\item[(iii)]  $e(P),e(P') \le \sqrt{\eps}n$.
\end{itemize}
\item[(b)] The $2t_K$ path systems in the pairs belonging to $\cM_{i_1i_2i_3i_4}$ are all pairwise edge-disjoint.
\item[(c)] Let $G(\cM_{i_1i_2i_3i_4})$ denote the spanning subgraph of $G$ whose edge set is the union of all the
path systems in the pairs belonging to $\cM_{i_1i_2i_3i_4}$. 
Then the $K^4$ graphs $G(\cM_{i_1i_2i_3i_4})$ are edge-disjoint.
Further, each $x \in A_0$ satisfies
$d_{G(\cM_{i_1i_2i_3i_4})}(x)\ge (d_G(x,A)-15\eps_4 D)/K^4$ while each $y \in B_0$ satisfies
$d_{G(\cM_{i_1i_2i_3i_4})}(y)\ge (d_G(y,B)-15\eps_4 D)/K^4$.%
    \COMMENT{Daniela: previously had "The analogous assertion holds for each $x \in B_0$ with $A$ replaced by $B$."}
\item[(d)] Let $G_{glob}$ be the subgraph of $G[A'] \cup G[B']$ obtained by removing all edges contained in $G(\cM_{i_1i_2i_3i_4})$ for all 
$1\le i_1,i_2,i_3,i_4 \le K$. Then $\Delta(G_{glob})\le 3k/2$. Moreover, $G_{glob}$ has a decomposition into
$k$ pairs of path systems $(Q_{1,A},Q_{1,B}),\dots,(Q_{k,A},Q_{k,B})$ so that
\begin{itemize}
\item[(i$'$)] $Q_{i,A}\subseteq G_{glob}[A']$ and $Q_{i,B}\subseteq G_{glob}[B']$ for all $i \le k$;
\item[(ii$'$)] $d_{Q_{i,A}}(x)\le 1$ for all $x\in A$ and $d_{Q_{i,B}}(x)\le 1$ for all $x\in B$;
\item[(iii$'$)] $e(Q_{i,A})-e(Q_{i,B})=a-b$ for all $i \le k$;
\item[(iv$'$)] $e(Q_{i,A}),e(Q_{i,B}) \le \sqrt{\eps}n$ for all $i \le k$.
\end{itemize}
\end{itemize}
\end{cor}
\proof
Apply Lemma~\ref{rnd_slice} to obtain localized slices $H^A_{ij}$ and $H^B_{ij}$ (for all $i,j\le K$).
Let $t:=K^2 t_K$ and let $t^*$ be as defined in Lemma~\ref{matchingdec}. Since $t/K^2, t^*/K^2\in\mathbb{N}$ we have $(t-t^*)/K^2\in\mathbb{N}$.
For all $i_1,i_2\le K$, let $\cM_{i_1i_2}^A$ be the set of $t$ path systems  in $H_{i_1i_2}^A$ guaranteed by Lemma~\ref{matchingdec}.
We call the $t^*$ path systems in $\cM_{i_1i_2}^A$ of size $\lceil \ell_a \rceil$ \emph{large} and the others \emph{small}.
We define $\cM_{i_3i_4}^B$ as well as large and small path systems in $\cM_{i_3i_4}^B$ analogously (for all $i_3,i_4 \leq K$).

We now construct the sets $\cM_{i_1i_2i_3i_4}$ as follows:
For all $i_1,i_2 \le K$, 
consider a random partition of the set of all large path systems in $\cM_{i_1i_2}^A$ into $K^2$ sets of equal size $t^*/K^2$ and assign
(all the path systems in) each of these sets to 
one of the $\cM_{i_1i_2i_3i_4}$ with $i_3,i_4 \le K$.
Similarly, randomly partition the set of small path systems in $\cM_{i_1i_2}^A$ into $K^2$ sets, each containing $(t-t^*)/K^2$ path systems.
Assign each of these $K^2$ sets to one of the $\cM_{i_1i_2i_3i_4}$ with $i_3,i_4 \le K$.
Proceed similarly for each $\cM_{i_3i_4}^B$ in order to assign each of its path systems randomly to some $\cM_{i_1i_2i_3i_4}$.
Then to each $\cM_{i_1i_2i_3i_4}$ we have assigned exactly $t^*/K^2$ large path systems from both $\cM_{i_1i_2}^A$ and $\cM_{i_3i_4}^B$.
Pair these off arbitrarily. Similarly, pair off the small path systems assigned to $\cM_{i_1i_2i_3i_4}$ arbitrarily.
Clearly, the sets $\cM_{i_1i_2i_3i_4}$ obtained in this way satisfy (a) and (b).

We now verify (c). By construction, the $K^4$ graphs $G(\cM_{i_1i_2i_3i_4})$ are edge-disjoint.%
   \COMMENT{Daniela: new sentence}
So consider any vertex $x\in A_0$ and
write $d:=d_G(x,A)$. Note that $d_{H_{i_1i_2}^A}(x)\ge (d-4 \eps_1 n)/K^2$
by Lemma~\ref{rnd_slice}(vi). Let $G(\cM_{i_1i_2}^A)$ be the spanning subgraph of $G$ whose edge set is the union of all the path systems
in $\cM_{i_1i_2}^A$. Then Lemma~\ref{matchingdec}(iv) implies that
$$
d_{G(\cM_{i_1i_2}^A)}(x)\ge d_{H_{i_1 i_2}^A}(x)-  \Delta (G^A _{i_1 i_2} )  \ge \frac{d-4\eps_1 n}{K^2} -\frac{13\eps_4D}{K^2} 
\ge \frac{d-14\eps_4 D}{K^2}.
$$
So a Chernoff-Hoeffding estimate for the hypergeometric distribution (Proposition~\ref{chernoff}) implies that%
$$
d_{G(\cM_{i_1i_2i_3i_4})}(x)\ge \frac{1}{K^2} \left( \frac{d-14\eps_4 D}{K^2} \right)  -\eps n
\ge \frac{d-15\eps_4 D}{K^4}.
$$
(Note that we only need to apply the Chernoff-Hoeffding bound if $d \ge \eps n$ say, as (c) is vacuous otherwise.)

It remains to check condition~(d). First note that $k\in\mathbb{N}$ since $t_K, D/2\in\mathbb{N}$.
Thus we can apply Lemma~\ref{G-glob} to obtain a decomposition of both $G^A_{glob}$ and $G^B_{glob}$
into path systems.
Since $G_{glob}=G_{glob}^A\cup G_{glob}^B$, (d) is an immediate consequence of Lemma~\ref{G-glob}(ii)--(v).
\endproof

\subsection{Constructing the localized balanced exceptional systems} \label{besconstruct}

The localized path systems obtained from Corollary~\ref{finalcor} do not yet cover all of the exceptional vertices.
This is achieved via the following lemma: we extend the path systems to achieve this additional property, while maintaining the property of being balanced.%
   \COMMENT{Deryk added 2 sentences}
More precisely, let
$$\cP:=\{A_0,A_1,\dots,A_K,B_0,B_1,\dots,B_K\}$$
be a $(K,m,\eps)$-partition%
    \COMMENT{Daniela replaced $(K,m,\eps_0)$-partition by $(K,m,\eps)$-partition to make clear that the $\eps$ need
not be the same as the parameter $\eps_0$ of the BES}
of a set $V$ of $n$ vertices.
Given $1\leq \I \leq K$ and $\eps_0>0$, an \emph{$\i$-balanced exceptional system with respect to $\cP$ and parameter~$\eps_0$}
is a path system $J$ with $V(J)\subseteq A_0\cup B_0\cup A_{i_1} \cup A_{i_2} \cup B_{i_3} \cup B_{i_4}$
such that the following conditions hold:
\begin{itemize}
\item[(BES$1$)] Every vertex in $A_0\cup B_0$ is an internal vertex of a path in~$J$.
Every vertex $v\in A_{i_1} \cup A_{i_2} \cup B_{i_3} \cup B_{i_4}$ satisfies $d_J(v)\le 1$.%
    \COMMENT{Can write that $v$ is an endpoint of a (possible trivial) path in~$J$ if we allow
$V(J)\subseteq A_0\cup B_0\cup A_{i_1} \cup A_{i_2} \cup B_{i_3} \cup B_{i_4}$ (instead of
$V(J)=A_0\cup B_0\cup A_{i_1} \cup A_{i_2} \cup B_{i_3} \cup B_{i_4}$). Also,
the BES we construct in Lemma~\ref{balmatchextend} have the property that none of these paths has both endpoints
in $A_{i_1} \cup A_{i_2}$ and its midoint in $A_0$ (and the analogue holds with for $B$).}   
\item[(BES$2$)]
Every edge of $J[A \cup B]$ is either an $A_{i_1}A_{i_2}$-edge or a $B_{i_3}B_{i_4}$-edge.%
	\COMMENT{AL: had `Every (maximal) path of length one in~$J$ has either one endpoint in $A_{i_1}$ and the other endpoint in $A_{i_2}$
or one endpoint in $B_{i_3}$ and the other endpoint in $B_{i_4}$.'
}
\item[(BES$3$)] The edges in $J$ cover precisely the same number of vertices in $A$ as in $B$.%
   \COMMENT{Cannot write $J$ instead of "edges in $J$" here since we don't want to count the vertices of degree 0 which lie in $V(J)$.}
\item[(BES$4$)] $e(J)\le \eps_0 n$.
\end{itemize}
To shorten the notation, we will often refer to $J$ as an \emph{$(i_1,i_2, i_3,i_4)$-BES}.
If $V$ is the vertex set of a graph $G$ and $J\subseteq G$, we also say that $J$ is an \emph{$(i_1,i_2, i_3,i_4)$-BES in~$G$}. 
Note that (BES2) implies that an $(i_1,i_2, i_3,i_4)$-BES does not contain edges between $A$ and $B$.
Furthermore, an $\i$-BES is also, for example, an $(i_2,i_1, i_4,i_3)$-BES.
We will sometimes omit the indices $i_1,i_2, i_3,i_4$ and just refer to a balanced exceptional system (or a BES for short).%
   \COMMENT{Daniela deleted: Throughout the paper, usually our balanced exceptional systems will have parameter $\eps_0$,
so in this case we will often omit mentioning the parameter.}
We will sometimes also omit the partition $\cP$, if it is clear from the context.

(BES1) implies that each balanced exceptional system is an $A_0B_0$-path system as defined before Proposition~\ref{balpathcheck}.
(However, the converse is not true since, for example, a 2-balanced $A_0B_0$-path system need not satisfy~(BES4).)
So (BES3) and Proposition~\ref{balpathcheck} imply that each balanced exceptional system is also $2$-balanced.

We now extend each set $\cM_{i_1i_2i_3i_4}$ obtained from Corollary~\ref{finalcor} into a set
$\mathcal{J}_{i_1i_2i_3i_4}$ of $(i_1,i_2,i_3,i_4)$-BES.

\begin{lemma}\label{balmatchextend}
Let $0<1/n\ll \eps\ll \eps_0\ll \eps'\ll \eps_1\ll \eps_2\ll \eps_3\ll \eps_4\ll 1/K\ll 1$.
Suppose that $(G,A,A_0,B,B_0)$ is an $(\eps,\eps',K,D)$-framework with $|G|=n$ and such that $D\ge n/200$
and $D$ is even. Let $\cP:=\{A_0,A_1,\dots,A_K,B_0,B_1,\dots,B_K\}$ be a $(K,m,\eps,\eps_1,\eps_2)$-partition for $G$. 
Suppose that $t_K:=(1-20\eps_4)D/2K^4\in \mathbb{N}$. Let
$\cM_{i_1i_2i_3i_4}$ be the sets returned by Corollary~\ref{finalcor}.
Then for all $1\leq \I \leq K$ there is a set $\mathcal{J}_{i_1i_2i_3i_4}$ which satisfies the
following properties:
\begin{itemize}
\item[(i)] $\mathcal{J}_{i_1i_2i_3i_4}$ consists of $t_K$ edge-disjoint $(i_1,i_2,i_3,i_4)$-BES in $G$ with respect to $\cP$
and with parameter $\eps_0$.
\item[(ii)] For each of the $t_K$ pairs of path systems $(P,P')\in \cM_{i_1i_2i_3i_4}$, 
there is a unique $J\in \mathcal{J}_{i_1i_2i_3i_4}$ which contains all the edges in $P\cup P'$. Moreover, all edges in
$E(J)\setminus E(P\cup P')$ lie in $G[A_0,B_{i_3}]\cup G[B_0,A_{i_1}]$.
\item[(iii)] Whenever $(i_1,i_2,i_3,i_4)\neq (i'_1,i'_2,i'_3,i'_4)$, $J \in \mathcal{J}_{i_1i_2i_3i_4}$ and $J' \in \mathcal{J}_{i'_1i'_2i'_3i'_4}$,
then $J$ and $J'$ are edge-disjoint.
\end{itemize}
\end{lemma}
We let $\mathfrak{J}$ denote the union of the sets $\mathcal{J}_{i_1i_2i_3i_4}$ over all $1\leq \I \leq K$.

\proof
We will construct the sets $\mathcal{J}_{i_1i_2i_3i_4}$ greedily by extending each pair of path systems $(P,P')\in \cM_{i_1i_2i_3i_4}$
in turn into an $(i_1,i_2,i_3,i_4)$-BES containing $P\cup P'$. For this, 
consider some arbitrary ordering of the $K^4$ $4$-tuples $(i_1,i_2,i_3,i_4)$.  
Suppose that we have already constructed the sets ${\mathcal J}_{i'_1i'_2i'_3i'_4}$ for all 
$(i'_1,i'_2,i'_3,i'_4)$ preceding $(i_1,i_2,i_3,i_4)$ so that (i)--(iii) are satisfied. So our aim now is to construct ${\mathcal J}_{i_1i_2i_3i_4}$.
Consider an enumeration $(P_1,P'_1),\dots,(P_{t_K},P'_{t_K})$ of the pairs of path systems%
    \COMMENT{Daniela replaced matchings by paths systems}
in $\cM_{i_1i_2i_3i_4}$.
Suppose that for some $i\le t_K$ we%
   \COMMENT{Daniela added for some $i\le t_K$}
have already constructed edge-disjoint $(i_1,i_2,i_3,i_4)$-BES $J_1,\dots,J_{i-1}$, so that for each $i'<i$
the following conditions hold:
\begin{itemize}
\item $J_{i'}$ contains the edges in $P_{i'}\cup P'_{i'}$;
\item all edges in $E(J_{i'})\setminus E(P_{i'}\cup P'_{i'})$ lie in $G[A_0,B_{i_3}]\cup G[B_0,A_{i_1}]$;
\item $J_{i'}$ is edge-disjoint from all the balanced exceptional systems in $\bigcup_{(i'_1,i'_2,i'_3,i'_4)} {\mathcal J}_{i'_1i'_2i'_3i'_4}$,
where the union is over all $(i'_1,i'_2,i'_3,i'_4)$ preceding $(i_1,i_2,i_3,i_4)$.
\end{itemize}
We will now construct $J:=J_i$. For this, we need to add suitable edges to $P_i\cup P'_i$ to ensure that all vertices  of $A_0 \cup B_0$ 
have degree two. We start with $A_0$. Recall that $a=|A_0|$ and write  $A_0=\{x_1,\dots,x_a\}$. 
Let $G'$ denote the subgraph of $G[A',B']$ obtained by removing all the edges lying in $J_1,\dots,J_{i-1}$ as well as all
those edges lying in the balanced exceptional systems belonging to $\bigcup_{(i'_1,i'_2,i'_3,i'_4)} {\mathcal J}_{i'_1i'_2i'_3i'_4}$ (where as before the union
is over all $(i'_1,i'_2,i'_3,i'_4)$ preceding $(i_1,i_2,i_3,i_4)$). 
We will choose the new edges incident to $A_0$ in $J$ inside $G'[A_0,B_{i_3}]$.

Suppose we have already found suitable edges for $x_1,\dots,x_{j-1}$ and let $J(j)$ be the set of all these edges.
We will first show that the degree of $x_j$ inside $G'[A_0,B_{i_3}]$ is still large.
Let $d_j:=d_G(x_j,A')$. Consider any $(i'_1,i'_2,i'_3,i'_4)$ preceding $(i_1,i_2,i_3,i_4)$.
Let $G({\mathcal J}_{i'_1i'_2i'_3i'_4})$ denote the union of the $t_K$ balanced exceptional systems belonging to ${\mathcal J}_{i'_1i'_2i'_3i'_4}$.
Thus $d_{G({\mathcal J}_{i'_1i'_2i'_3i'_4})}(x_j)=2t_K$. However, Corollary~\ref{finalcor}(c) implies that
$d_{G({\mathcal M}_{i'_1i'_2i'_3i'_4})}(x_j)\ge (d_j-15\eps_4 D)/K^4$. So altogether, when constructing (the balanced exceptional systems in)
${\mathcal J}_{i'_1i'_2i'_3i'_4}$, we have added at most $2t_K-(d_j-15\eps_4 D)/K^4$ new edges
at $x_j$, and all these edges join $x_j$ to vertices in $B_{i'_3}$. Similarly, when constructing
$J_1,\dots,J_{i-1}$, we have added at most $2t_K-(d_j-15\eps_4 D)/K^4$ new edges at $x_j$.
Since the number of $4$-tuples $(i'_1,i'_2,i'_3,i'_4)$ with $i'_3=i_3$ is $K^3$, it follows that%
   \COMMENT{Daniela: previously had $\le $ instead of the first $=$}
\begin{align*}
d_G(x_j,B_{i_3})-d_{G'}(x_j,B_{i_3})  
& \le K^3 \left(2t_K - \frac{d_j-15\eps_4 D}{K^4}\right) \\
& = \frac{1}{K} \left( (1-20\eps_4)D -d_j+15\eps_4D \right) \\
& = \frac{1}{K} \left( D-d_j -5 \eps_4D \right). 
\end{align*}
Also, (P2) with $A$ replaced by~$B$ implies that%
  \COMMENT{AL:changed $d_G(x_j,B_{i_3})=\frac{d_G(x_j,B)\pm \eps_1 n}{K}$ to $d_G(x_j,B_{i_3})\ge\frac{d_G(x_j,B)- \eps_1 n}{K}$}
$$d_G(x_j,B_{i_3}) \ge \frac{d_G(x_j,B) -  \eps_1 n}{K}\ge \frac{d_G(x_j)-d_G(x_j,A')-\eps_1 n}{K} = \frac{D-d_j -\eps_1 n}{K},
$$
where here we use (FR2) and (FR6).
So altogether, we have%
\COMMENT{AL: changed $\eps_4 n/25K$ with $\eps_4 n/50K$}
$$
d_{G'}(x_j,B_{i_3})  \ge (5\eps_4 D- \eps_1n)/K\ge \eps_4 n/50K.
$$
Let $B'_{i_3}$ be the set of vertices in $B_{i_3}$ not covered by the edges of $J(j)\cup P'_i$.%
\COMMENT{AL: added edges of}
Note that $|B'_{i_3}| \ge |B_{i_3}|- 2|A_0|-2e(P'_i)\ge |B_{i_3}|-3\sqrt{\eps}n$ since 
$a=|A_0|\le \eps n$ by (FR4)
and $e(P'_i)\le \sqrt{\eps}n$ by Corollary~\ref{finalcor}(a)(iii).
So $d_{G'}(x_j,B'_{i_3}) \ge \eps_4 n/51K$. We can add up to two of these edges to~$J$ in order to ensure that $x_j$ has degree two in~$J$.
This completes the construction of the edges of $J$ incident to $A_0$. The edges incident to $B_0$ are found similarly.

Let $J$ be the graph on $A_0\cup B_0 \cup A_{i_1} \cup A_{i_2} \cup B_{i_3} \cup B_{i_4}$ whose edge set is constructed in this way.
By construction, $J$ satisfies (BES1) and (BES2) since $P_j$ and $P'_j$ are
$(i_1,i_2,A)$-localized and $(i_3,i_4,B)$-localized respectively.%
	\COMMENT{AL: deleted follows}
We now verify (BES3). As mentioned before the statement of the lemma, 
(BES1) implies that $J$ is an $A_0B_0$-path system (as defined before Proposition~\ref{balpathcheck}).
Moreover, Corollary~\ref{finalcor}(a)(ii) implies that $P_i \cup P'_i$ is a path system which satisfies (B1)
in the definition of $2$-balanced.
Since $J$ was obtained by adding only $A'B'$-edges, (B1) is preserved in $J$.
Since by construction $J$ satisfies (B2), it follows that $J$ is $2$-balanced.
So Proposition~\ref{balpathcheck} implies (BES3). 

 Finally, we verify (BES4). For this, note that Corollary~\ref{finalcor}(a)(iii) implies
that $e(P_i),e(P'_i)\le \sqrt{\eps}n$. Moreover, the number of edges added to $P_i\cup P'_i$ when constructing $J$
is at most $2(|A_0|+|B_0|)$, which is at most $2\eps n$ by (FR4).
Thus $e(J)\le 2\sqrt{\eps}n+2\eps n\le \eps_0 n$.
\endproof

\subsection{Covering $G_{glob}$ by edge-disjoint Hamilton cycles}

We now find a set of edge-disjoint Hamilton cycles covering the edges of the `leftover' graph obtained from $G-G[A,B]$ by deleting all
those edges lying in balanced exceptional systems belonging to $\mathfrak{J}$.

\begin{lemma}\label{lem:HCglob}
Let $0<1/n\ll \eps\ll \eps_0\ll \eps'\ll \eps_1\ll \eps_2\ll \eps_3\ll \eps_4\ll 1/K\ll 1$.
Suppose that $(G,A,A_0,B,B_0)$ is an $(\eps,\eps',K,D)$-framework with $|G|=n$ and such that $D\ge n/200$ and%
   \COMMENT{Daniela: had $\delta (G) \geq D\ge n/200$}
$D$ is even. Let $\cP:=\{A_0,A_1,\dots,A_K,B_0,B_1,\dots,B_K\}$ be a $(K,m,\eps,\eps_1,\eps_2)$-partition for $G$. 
Suppose that $t_K:=(1-20\eps_4)D/2K^4\in \mathbb{N}$. Let $\mathfrak{J}$ be as defined after Lemma~\ref{balmatchextend}
and let $G(\mathfrak{J})\subseteq G$ be the union of all the balanced exceptional systems lying in $\mathfrak{J}$.
Let $G^*:=G-G(\mathfrak{J})$, let $k:=10\eps_4 D$ and let $(Q_{1,A},Q_{1,B}),\dots,(Q_{k,A},Q_{k,B})$ be as in Corollary~\ref{finalcor}(d).
\begin{itemize}
\item[(a)] The graph $G^*-G^*[A,B]$ can be decomposed into $k$ $A_0B_0$-path systems $Q_1,\dots,Q_k$ which are $2$-balanced and satisfy the
following properties:
\begin{itemize}
\item[(i)] $Q_i$ contains all edges of $Q_{i,A} \cup Q_{i,B}$; 
\item[(ii)]  $Q_1,\dots,Q_k$ are pairwise edge-disjoint;
\item[(iii)] $e(Q_i) \le 3\sqrt{\eps}n$.
\end{itemize}
\item[(b)] Let $Q_1,\dots,Q_k$ be as in~(a). Suppose that $F$ is a graph on $V(G)$ such that $G\subseteq F$, $\delta(F)\ge 2n/5$ and such that $F$ satisfies
(WF5) with respect to $\eps'$. Then there are edge-disjoint Hamilton cycles $C_1,\dots,C_k$ in $F-G(\mathfrak{J})$
such that $Q_i \subseteq C_i$ and  $C_i \cap G$ is $2$-balanced for each $i\le k$. 
\end{itemize}
\end{lemma}
\proof
We first prove (a). The argument is similar to that of Lemma~\ref{G-glob}. Roughly speaking, we will extend
each $Q_{i,A}$ into a path system $Q_{i,A}'$ by adding suitable $A_0B$-edges which ensure that 
every vertex in $A_0$ has degree exactly two in $Q_{i,A}'$. 
Similarly, we will extend each $Q_{i,B}$ into $Q'_{i,B}$ by adding suitable $AB_0$-edges. 
We will ensure that no vertex is an endvertex of both an edge in $Q'_{i,A}$ and an edge in $Q'_{i,B}$ and take $Q_i$ to
be the union of these two path systems. We first construct all the $Q_{i,A}'$.

\medskip

\noindent
{\bf Claim~1.} \emph{$G^*[A']\cup G^*[A_0,B]$ has a decomposition into edge-disjoint path systems $Q'_{1,A},\dots,Q'_{k,A}$ such that
\begin{itemize}
\item $Q_{i,A}\subseteq Q'_{i,A}$ and $E(Q'_{i,A})\setminus E(Q_{i,A})$ consists of $A_0B$-edges in $G^*$ (for each $i\le k$);
\item $d_{Q'_{i,A}}(x)=2$ for every $x\in A_0$ and $d_{Q'_{i,A}}(x)\le 1$ for every $x\notin A_0$;
\item no vertex is an endvertex of both an edge in $Q'_{i,A}$ and an edge in $Q_{i,B}$ (for each $i\le k$).%
    \COMMENT{Cannot simply say that $Q'_{i,A}$ and $Q_{i,B}$ are vertex-disjoint since we allow trivial paths in a path system
and so we might e.g. have that $V(Q_{i,B})=B'$.}
\end{itemize}
}

\medskip

\noindent
To prove Claim~1, let $G_{glob}$ be as defined in Corollary~\ref{finalcor}(d). Thus $G_{glob}[A']=Q_{1,A}\cup \dots \cup Q_{k,A}$.
On the other hand, Lemma~\ref{balmatchextend}(ii) implies that $G^*[A']=G_{glob}[A']$.
Hence,
\begin{equation}\label{eq:G*A'}
G^*[A']=G_{glob}[A']=Q_{1,A}\cup \dots \cup Q_{k,A}.
\end{equation}
Similarly, $G^*[B']=G_{glob}[B']=Q_{1,B}\cup \dots \cup Q_{k,B}$. Moreover, $G_{glob}=G^*[A']\cup G^*[B']$.
Consider any vertex $x \in A_0$. Let $d_{glob}(x)$ denote the degree of $x$ in $Q_{1,A}\cup \dots \cup Q_{k,A}$.
So $d_{glob}(x)=d_{G^*}(x,A')$ by~(\ref{eq:G*A'}). Let
\begin{align}
d_{loc}(x) & :=d_G(x,A')-d_{glob}(x)\label{eq:dloc1}\\
& =d_G(x,A')-d_{G^*}(x,A')=d_{G(\mathfrak{J})}(x,A')\label{eq:dloc2}.
\end{align}
Then
\begin{equation}\label{eq:dsum}
d_{loc}(x)+d_G(x,B')+d_{glob}(x)\stackrel{(\ref{eq:dloc1})}{=}d_G(x)=D,
\end{equation}
where the final equality follows from (FR2).
Recall that $\mathfrak{J}$ consists of $K^4t_K$ edge-disjoint balanced exceptional systems.
Since $x$ has two neighbours in each of these balanced exceptional systems, 
the degree of $x$ in $G(\mathfrak{J})$ is $2K^4 t_K=D-2k$.
%Moreover, by Lemma~\ref{balmatchextend}(ii) all but $d_{loc}(x)$ of the edges incident
%to $x$ in $G(\mathfrak{J})$ are chosen from $G[A',B']$.
Altogether this implies that
\begin{eqnarray} \label{dB'x}
d_{G^*}(x,B') & = & d_{G}(x,B')-d_{G(\mathfrak{J})}(x,B')=d_{G}(x,B')-(d_{G(\mathfrak{J})}(x)-d_{G(\mathfrak{J})}(x,A'))\nonumber\\
& \stackrel{(\ref{eq:dloc2})}{=} & d_{G}(x,B')-(D-2k-d_{loc}(x))\stackrel{(\ref{eq:dsum})}{=} 2k-d_{glob}(x).
\end{eqnarray}
Note that this is precisely the total number of edges at $x$ which we need to add to $Q_{1,A},\dots,Q_{k,A}$
in order to obtain $Q'_{1,A},\dots,Q'_{k,A}$ as in Claim~1.

We can now construct the path systems $Q'_{i,A}$.  For each $x \in A_0$, let $n_i(x)=2-d_{Q_{i,A}}(x)$.
So $0\le n_i(x)\le 2$ for all $i\le k$.
Recall that $a:=|A_0|$ and consider an ordering $x_1,\dots,x_{a}$ of the vertices in $A_0$.
Let $G^*_j:=G^*[\{x_1,\dots,x_j\},B]$. Assume that for some 
$0 \le j < a$, we have already found a decomposition of $G^*_j$ into edge-disjoint path systems $Q_{1,j},\dots,Q_{k,j}$
satisfying the following properties (for all $i\le k$): 
\begin{itemize}
\item[(i$'$)] no vertex is an endvertex of both an edge in $Q_{i,j}$ and an edge in $Q_{i,B}$;
\item[(ii$'$)] $x_{j'}$ has degree $n_{i}(x_{j'})$ in $Q_{i,j}$ for all $j' \le j$
and all other vertices have degree at most one in $Q_{i,j}$.
\end{itemize}
We call this assertion ${\mathcal A}_j$. We will show that ${\mathcal A}_{j+1}$ holds 
(i.e.~the above assertion also holds with $j$ replaced by $j+1$).
This in turn implies Claim~1 if we let $Q'_{i,A}:=Q_{i,a}\cup Q_{i,A}$ for all $i\le k$.

To prove ${\mathcal A}_{j+1}$, consider the following bipartite auxiliary graph $H_{j+1}$.
The vertex classes of $H_{j+1}$ are $N_{j+1}:=N_{G^*}(x_{j+1})\cap B$ and $Z_{j+1}$, where
$Z_{j+1}$ is a multiset whose elements are chosen from $Q_{1,B},\dots,Q_{k,B}$.%
\COMMENT{AL: joined the two sentences together.}
Each $Q_{i,B}$ is included exactly $n_i(x_{j+1})$ times in $Z_{j+1}$. Note that
$N_{j+1}=N_{G^*}(x_{j+1})\cap B'$ since $e(G[A_0,B_0])=0$ by (FR6).
Altogether this implies that
\begin{eqnarray}
|Z_{j+1}| & = & \sum_{i=1}^k n_i(x_{j+1})=2k-\sum_{i=1}^k d_{Q_{i,A}}(x_{j+1})=2k-d_{glob}(x_{j+1})\label{eq:Zjplus1}\\
& \stackrel{(\ref{dB'x})}{=} & d_{G^*}(x_{j+1},B') =|N_{j+1}| \ge k/2\nonumber.
\end{eqnarray}
The final inequality follows from (\ref{dB'x}) since
$$d_{glob}(x_{j+1})\stackrel{(\ref{eq:G*A'})}{\le} \Delta(G_{glob}[A']) \le 3k/2$$
by Corollary~\ref{finalcor}(d).
We include an edge in $H_{j+1}$ between $v \in N_{j+1}$ and $Q_{i,B} \in Z_{j+1}$ if $v$ is not an endvertex of an edge in 
$Q_{i,B}\cup Q_{i,j}$.

\medskip

\noindent
{\bf Claim~2.} \emph{$H_{j+1}$ has a perfect matching $M'_{j+1}$.}

\medskip

\noindent
Given the perfect matching guaranteed by the claim, we construct $Q_{i,j+1}$ 
from $Q_{i,j}$ as follows:
the edges of $Q_{i,j+1}$ incident to $x_{j+1}$ are precisely the edges $x_{j+1}v$ where $vQ_{i,B}$ is an edge of $M'_{j+1}$
(note that there are up to two of these). Thus Claim~2 implies that ${\mathcal A}_{j+1}$ holds. 
(Indeed, (i$'$)--(ii$'$) are immediate from the definition of~$H_{j+1}$.)
 
\medskip

\noindent
To prove Claim~2, consider any vertex $v \in N_{j+1}$. Since $v \in B$, the number of path systems $Q_{i,B}$ containing an edge at $v$ is 
at most $d_{G}(v,B')$. The number of indices $i$ for which $Q_{i,j}$ contains an edge at $v$ is at most $d_{G}(v,A_0) \le |A_0|$. 
Since each path system $Q_{i,B}$ occurs at most twice
in the multiset $Z_{j+1}$, it follows that the degree of $v$ in $H_{j+1}$ is at least
$|Z_{j+1}|-2d_{G}(v,B')-2|A_0|$. Moreover, $d_G(v,B') \le \eps' n \le k/16$ (say) by (FR5).
Also, $|A_0| \le \eps n\le k/16$ by (FR4).
So $v$ has degree at least $|Z_{j+1}|-k/4\ge |Z_{j+1}|/2$ in $H_{j+1}$.

Now consider any path system $Q_{i,B} \in Z_{j+1}$. Recall that $e(Q_{i,B})\le \sqrt{\eps}n\le k/16$ (say),
where the first inequality follows from Corollary~\ref{finalcor}(d)(iv$'$).
Moreover, $e(Q_{i,j}) \le 2|A_0|\le 2 \eps n \le k/8$,
where the second inequality follows from (FR4). 
Thus the degree of $Q_{i,B}$ in $H_{j+1}$ is at least
$$
|N_{j+1}|-2e(Q_{i,B})-e(Q_{i,j})\ge |N_{j+1}|-k/4\ge |N_{j+1}|/2.
$$ 
Altogether this implies that $H_{j+1}$ has a perfect matching $M'_{j+1}$, as required.

\medskip
This completes the construction of $Q'_{1,A},\dots,Q'_{k,A}$. Next we construct $Q'_{1,B},\dots,Q'_{k,B}$ using the same approach.%
\COMMENT{Deryk+AL: added sentences}

\medskip 

\noindent
{\bf Claim~3.} \emph{$G^*[B']\cup G^*[B_0,A]$ has a decomposition into edge-disjoint path systems $Q'_{1,B},\dots,Q'_{k,B}$ such that
\begin{itemize}
\item $Q_{i,B}\subseteq Q'_{i,B}$ and $E(Q'_{i,B})\setminus E(Q_{i,B})$ consists of $B_0A$-edges in $G^*$ (for each $i\le k$);
\item $d_{Q'_{i,B}}(x)=2$ for every $x\in B_0$ and $d_{Q'_{i,B}}(x)\le 1$ for every $x\notin B_0$;
\item no vertex is an endvertex of both an edge in $Q'_{i,A}$ and an edge in $Q'_{i,B}$ (for each $i\le k$).
\end{itemize}
}
\medskip

\noindent
The proof of Claim~3 is similar to that of Claim~1.%
\COMMENT{osthus changed 2 to 1}
The only difference is that when constructing $Q'_{i,B}$, we need to avoid the endvertices of all the edges in $Q'_{i,A}$
(not just the edges in $Q_{i,A}$). However, $e(Q'_{i,A}-Q_{i,A})\le 2|A_0|$, so this does not affect the calculations
significantly. 

\medskip

\noindent
We now take $Q_i:=Q'_{i,A}\cup Q'_{i,B}$ for all $i\le k$. Then the $Q_i$ are pairwise edge-disjoint
and $$e(Q_i)\le e(Q_{i,A})+e(Q_{i,B})+2|A_0\cup B_0|\le 2\sqrt{\eps}n+2\eps n\le 3\sqrt{\eps}n$$
by Corollary~\ref{finalcor}(d)(iv$'$) and (FR4).
Moreover, Corollary~\ref{finalcor}(d)(iii$'$) implies that
\begin{equation}\label{eq:edgediff}
e_{Q_i}(A')-e_{Q_i}(B')=e(Q_{i,A})-e(Q_{i,B})=a-b.
\end{equation}
Thus each $Q_i$ is a 2-balanced $A_0B_0$-path system. Further, $Q_1,\dots,Q_k$ form a decomposition of
$$G^*[A']\cup G^*[A_0,B]\cup G^*[B']\cup G^*[B_0,A]=G^*-G^*[A,B].$$
(The last equality follows since $e(G[A_0,B_0])=0$ by (FR6).)
This completes the proof of~(a).

To prove (b), note that $(F,G,A,A_0,B,B_0)$ is an $(\eps,\eps',D)$-pre-framework, i.e.~it satisfies (WF1)--(WF5).
Indeed, recall that (FR1)--(FR4) imply (WF1)--(WF4) and that (WF5) holds by assumption.
So we can apply Lemma~\ref{extendpaths} (with $Q_1$ playing the role of $Q$) 
to extend $Q_1$ into a Hamilton cycle $C_1$.
Moreover, Lemma~\ref{extendpaths}(iii) implies that
$C_1 \cap G$ is $2$-balanced, as required. (Lemma~\ref{extendpaths}(ii) guarantees%
   \COMMENT{Daniela added ref to Lemma~\ref{extendpaths}(ii)}
that $C_1$ is edge-disjoint from $Q_2, \dots , Q_k$ and $G(\mathfrak{J})$.)

Let $G_1:=G-C_1$ and $F_1:=F-C_1$.%
\COMMENT{Daniela: previously had "Let $G_1$ be obtained from $G$ by removing the edges of $C_1$ and $F_1$ be obtained by the removing the edges of $C_1$."}
Proposition~\ref{WFpreserve} (with $C_1$ playing the role of $H$)
implies that $(F_1,G_1,A,A_0,B,B_0)$ is an $(\eps,\eps',D-2)$-pre-framework.
So we can now apply Lemma~\ref{extendpaths}%
    \COMMENT{Daniela replaced Lemma~\ref{coverA0B0} by Lemma~\ref{extendpaths}}
to $(F_1,G_1,A,A_0,B,B_0)$ to extend $Q_2$ into a Hamilton cycle $C_2$,
where $C_2 \cap G$ is also $2$-balanced.

We can continue this way to find $C_3,\dots, C_k$.
Indeed, suppose that we have found $C_1,\dots,C_i$ for $i<k$.
Then we can still apply Lemma~\ref{extendpaths}
since $\delta(F)-2i \ge \delta(F)-2k\ge n/3$.%
   \COMMENT{Daniela deleted " and since $D-2k\ge n/300$" - don't see why we need this} 
Moreover, $C_j \cap G$ is $2$-balanced for all $j \le i$, so $(C_1 \cup \dots \cup C_i) \cap G$ is $2i$-balanced.
This in turn means that 
Proposition~\ref{WFpreserve} (applied with $C_1 \cup \dots \cup C_i$ playing the role of $H$)
implies that after removing $C_1,\dots,C_i$, we still have an $(\eps,\eps',D-2i)$-pre-framework and can find $C_{i+1}$.
\endproof

 We can now put everything together to find a set of localized balanced exceptional systems
 and a set of Hamilton cycles which altogether cover all edges of $G$ outside $G[A,B]$.
 The localized balanced exceptional systems will be extended to Hamilton cycles later on.%
 \COMMENT{Deryk added sentences}
\begin{cor}\label{BEScor}
Let $0<1/n\ll \eps\ll \eps_0\ll \eps'\ll \eps_1\ll \eps_2\ll \eps_3\ll \eps_4\ll 1/K\ll 1$.
Suppose that $(G,A,A_0,B,B_0)$ is an $(\eps,\eps',K,D)$-framework with $|G|=n$ and such that $D\ge n/200$ and%
   \COMMENT{Daniela: had $\delta (G) \geq D\ge n/200$}
$D$ is even. Let $\cP:=\{A_0,A_1,\dots,A_K,B_0,B_1,\dots,B_K\}$ be a $(K,m,\eps,\eps_1,\eps_2)$-partition for $G$. 
Suppose that $t_K:=(1-20\eps_4)D/2K^4\in \mathbb{N}$ and let $k:=10\eps_4 D$.
Suppose that $F$ is a graph on $V(G)$ such that $G\subseteq F$, $\delta(F)\ge 2n/5$ and such that $F$ satisfies
(WF5) with respect to $\eps '$. Then there are $k$ edge-disjoint Hamilton cycles $C_1,\dots,C_k$ in $F$
and for all $1\leq \I \leq K$ there is a set $\mathcal{J}_{i_1i_2i_3i_4}$ such that the
following properties are satisfied:
\begin{itemize}
\item[(i)] $\mathcal{J}_{i_1i_2i_3i_4}$ consists of $t_K$ $(i_1,i_2,i_3,i_4)$-BES in $G$ with respect to $\cP$ and
with parameter $\eps_0$ which are edge-disjoint from each other
and from $C_1\cup\dots\cup C_k$.
\item[(ii)] Whenever $(i_1,i_2,i_3,i_4)\neq (i'_1,i'_2,i'_3,i'_4)$, $J \in \mathcal{J}_{i_1i_2i_3i_4}$ and $J' \in \mathcal{J}_{i'_1i'_2i'_3i'_4}$,
then $J$ and $J'$ are edge-disjoint.
\item[(iii)] Given any $i \leq k$ and $v \in A_0 \cup B_0$, the two edges incident 
to $v$ in $C_i$ lie in $G$.
\item[(iv)] Let $G^\diamond$ be the subgraph of $G$ obtained by deleting the edges of all the $C_i$ and all
the balanced exceptional systems in $\mathcal{J}_{i_1i_2i_3i_4}$ (for all $1\leq \I \leq K$).
Then $G^\diamond$ is bipartite with vertex classes $A'$, $B'$ and $V_0=A_0 \cup B_0$%
\COMMENT{osthus added $A_0 cup B_0$} 
is an isolated set in $G^\diamond$.%
\COMMENT{Originally it was written that:
Moreover, $G^\diamond$ is $(D-2k-2K^4t_K)$-balanced with respect to $(A,A_0,B,B_0)$.
But note this is vacuous (i.e. $0$-balanced).}
\end{itemize}
\end{cor}
\proof
This follows immediately from Lemmas~\ref{balmatchextend} and~\ref{lem:HCglob}(b). Indeed, clearly~(i)--(iii) are satisfied.
To check~(iv), note that $G^\diamond$ is obtained from the graph $G^*$ defined in Lemma~\ref{lem:HCglob}
by deleting all the edges of the Hamilton cycles $C_i$. But Lemma~\ref{lem:HCglob} implies that the $C_i$ together cover all the edges in
$G^*-G^*[A,B]$. Thus this implies that $G^\diamond$ is bipartite with vertex classes $A'$, $B'$
and $V_0$ is an isolated set in $G^\diamond$. 
\endproof

%%%%%%%%%%%%%%%%%%%%%%%%%%%%%%%%%%%%%%%%%%%%%%%%%%%%%%%%%%%%%%%%%%%%%%%%%%%%%%%%%%%%%%

\section{Special factors and balanced exceptional factors}\label{sec:spec}

As discussed in the proof sketch, the proof of Theorem~\ref{1factbip} proceeds as follows.
First we find an approximate decomposition of the given graph $G$ and finally we find a decomposition of the (sparse) leftover from the approximate decomposition
(with the aid of a `robustly decomposable' graph we removed earlier).
Both the approximate decomposition as well as the actual decomposition steps assume that we work with a bipartite graph
on $A \cup B$ (with $|A|=|B|$). So in both steps,
we would need $A_0 \cup B_0$ to be empty, which we clearly cannot assume. On the other hand, in both steps, one can specify 
`balanced exceptional path systems' (BEPS) in $G$ with the following crucial property:
one can replace each BEPS with a path system BEPS$^*$ so that
\begin{itemize}
\item[($\alpha_1$)] BEPS$^*$ is bipartite with vertex classes $A$ and $B$;
\item[($\alpha_2$)] a Hamilton cycle $C^*$ in $G^*:=G[A,B] + {\rm BEPS}^*$ which%
    \COMMENT{Daniela replaced $G^*:=G[A,B] \cup {\rm BEPS}^*$ by $G^*:=G[A,B] + {\rm BEPS}^*$}
contains BEPS$^*$ corresponds to a Hamilton cycle $C$ in $G$
which contains BEPS (see Section~\ref{BESstar}).
\end{itemize}
Each BEPS will contain one of the balanced exceptional sequences BES constructed in Section~\ref{findBES}.
BEPS$^*$ will then be obtained by replacing the edges in BES by suitable `fictive' edges (i.e.~which are not necessarily contained in $G$). 

So, roughly speaking, this allows us to work with 
$G^*$ rather than $G$ in the two steps.
A convenient way of specifying and handling these balanced exceptional path systems is to combine them into `balanced exceptional factors' BF
(see Section~\ref{sec:BEPS} for the definition).

One complication is that the `robust decomposition lemma' (Lemma~\ref{rdeclemma}) we use from~\cite{Kelly}
deals with digraphs rather than undirected graphs. 
So to be able to apply it, we need a suitable orientation of  the edges of $G$ and so we will actually consider directed path systems
BEPS$^*_{\rm dir}$ instead of BEPS$^*$ above (whereas the path systems BEPS are undirected).

The formulation of the robust decomposition lemma is quite general and rather than guaranteeing ($\alpha_2$) directly, it
 assumes the existence of certain directed `special paths systems' SPS which are combined into `special factors' SF.
These are introduced in Section~\ref{sec:SF}.
Each of the Hamilton cycles produced by the lemma then contains exactly one of these special path systems.
So to apply the lemma, it suffices to check separately that each BEPS$^*_{\rm dir}$ satisfies the conditions required of a special path system 
and that it also satisfies ($\alpha_2$).%
\COMMENT{Andy: changed title of section 6.1}

\subsection{Constructing the graphs $J^*$ from the balanced exceptional systems~$J$} \label{BESstar}

Suppose that $J$ is a balanced exceptional system in a graph~$G$ with respect to a $(K,m,\eps_0)$-partition
$\cP=\{A_0,A_1,\dots,A_K,B_0,B_1,\dots,B_K\}$ of $V(G)$. We will now use $J$ to define an auxiliary matching $J^*$.
Every edge of $J^*$ will have one endvertex in $A$ and its other endvertex in $B$.
We will regard $J^*$ as being edge-disjoint from the original graph~$G$. So 
even if both $J^*$ and $G$ have an edge between the same pair of endvertices, we will regard these as different edges.
The edges of such a $J^*$ will be called \emph{fictive edges}. Proposition~\ref{CES-H}(ii) below shows that a Hamilton cycle in $G[A\cup B]+ J^*$ containing
all edges of $J^*$ in a suitable order will correspond to a Hamilton cycle in~$G$ which contains~$J$.
So when finding our Hamilton cycles, this property will enable us to ignore all the vertices in $V_0=A_0\cup B_0$ and to consider
a bipartite (multi-)graph between $A$ and $B$ instead.

We construct $J^{*}$ in two steps.
First we will construct a matching $J^*_{AB}$ on $A \cup B$
and then $J^{*}.$%
	   \COMMENT{AL: I have rephrased the definition/construction of $J^*$. $J^*$ is still the same as before. The construction is similar to the two cliques case.  This is from paper 4.}
Since each maximal path in $J$ has endpoints in $A \cup B$ and internal vertices in $V_0$ by (BES1), a balanced exceptional system $J$ naturally induces a matching $J^*_{AB}$ on $A \cup B$.
More precisely, if $P_1, \dots ,P_{\ell'}$ are the non-trivial paths in~$J$ and $x_i, y_i$ are the endpoints of $P_i$, then we define $J^*_{AB} := \{x_iy_i : i  \le \ell'\}$. 
Thus $J^*_{AB}$ is a matching by~(BES1) and $e(J^*_{AB}) \le e(J)$.%
    \COMMENT{Daniela deleted " by~(BES4)"}
Moreover, $J^*_{AB}$ and $E(J)$ cover exactly the same vertices in $A$. 
Similarly, they cover exactly the same vertices in $B$. 
So (BES3) implies that $e(J^*_{AB}[A])=e(J^*_{AB}[B])$.
We can write $E(J^*_{AB}[A])=\{x_1x_2, \dots, x_{2s-1}x_{2s}\}$,
$E(J^*_{AB}[B])=\{y_1y_2, \dots, y_{2s-1}y_{2s}\}$ and $E(J^*_{AB}[A,B])=\{x_{2s+1}y_{2s+1}, \dots, x_{s'}y_{s'}\}$, where $x_i \in A$ and $y_i \in B$.
Define $J^*:= \{ x_iy_i : 1 \le i \le s' \}$.
Note that $	e(J^*) =  e(J^*_{AB}) \le e(J)$.
All edges of $J^*$ are called \emph{fictive edges}.

As mentioned before, we regard $J^*$ as being edge-disjoint from the original graph~$G$. 
Suppose that $P$ is an orientation of a subpath of (the multigraph) $G[A\cup B]+J^*$. 
We say that $P$ is \emph{consistent with $J^{*}$} if $P$ contains all the edges of $J^{*}$
and $P$ traverses the vertices $x_1,y_1,x_2,\dots,y_{s'-1},x_{s'},y_{s'}$ in this order.%
     \COMMENT{we need to prescribe a vertex rather than an edge ordering as it is important in the next lemma}
(This ordering will be crucial for the vertices $x_1,y_1,\dots,x_{2s},y_{2s}$,%
   \COMMENT{Daniela replaced "added in step (BES$^*$1)" by "$x_1y_1,\dots,x_{2s}y_{2s}$". Also defined when a cycle is consistent}
but it is also convenient to have an ordering
involving all vertices of~$J^*$.) 
Similarly, we say that a cycle $D$ in $G[A\cup B]+J^*$ is \emph{consistent with $J^{*}$} if $D$ contains all the edges of $J^{*}$
and there exists some orientation of $D$ which
traverses the vertices $x_1,y_1,x_2,\dots,y_{s'-1},x_{s'},y_{s'}$ in this order.

The next result shows that if $J$ is a balanced exceptional system and $C$ is a Hamilton cycle on
$A \cup B$ which is consistent with $J^*$, then the graph obtained from $C$ by replacing $J^*$
with $J$ is a Hamilton cycle on $V(G)$ which contains~$J$, see Figure~\ref{fig3}.
When choosing our Hamilton cycles, this property will enable us ignore all the vertices in $V_0$ and edges in $A$ and $B$ and to consider
the (almost complete) bipartite graph with vertex classes $A$ and $B$ instead.%
   \COMMENT{Daniela rephrased (ii) and also adjusted the proof to bring it in line with Allan's changed def of BES}

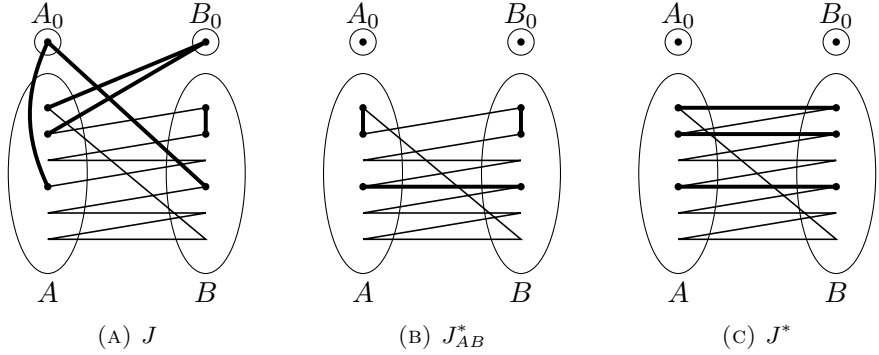
\begin{figure}[tbp]
\centering
\subfloat[$J$]{
\begin{tikzpicture}[scale=0.35]
% 			\node at (0,-6) {$J$};
			\draw  (-3,0) ellipse (1.5  and 3.8 );
			\draw (3,0) ellipse (1.5  and 3.8 );
			\node at (-3,-4.5) {$A$};
			\node at (3,-4.5) {$B$}; 			
			\draw (-3,5) circle(0.5);
			\draw (3,5) circle (0.5);
			\node at (-3,6) {$A_0$};
			\node at (3,6) {$B_0$};
			\fill (-3,5) circle (4pt);
			\fill (3,5) circle (4pt);
			\begin{scope}[start chain]
			\foreach \i in {2.5,1.5,-0.5}
			\fill (-3,\i) circle (4pt);
			\end{scope}
			\begin{scope}[start chain]
			\foreach \i in {2.5,1.5,-0.5}
			\fill (3,\i) circle (4pt);
			\end{scope}

			\begin{scope}[line width=1.5pt]
			\draw (-3,-0.5) to [out=115, in=-115] (-3,5)--(3,-0.5);
			\draw (-3,1.5)--(3,5)--(-3,2.5);
			\draw (3,2.5)--(3,1.5);
			\end{scope}
			
			\begin{scope}[line width=0.6pt]
			\draw (3,2.5)-- (-3,1.5);
			\draw (3,1.5)-- (-3,0.5)--(3,0.5)  -- (-3,-0.5);
			\draw (3,-0.5)--(-3,-1.5) -- (3,-1.5) --(-3,-2.5) -- (3,-2.5) -- (-3,2.5);
			\end{scope}
\end{tikzpicture}
}
\qquad
\subfloat[$J_{AB}^*$]{
\begin{tikzpicture}[scale=0.35]
% 			\node at (0,-6) {$J$};
			\draw  (-3,0) ellipse (1.5  and 3.8 );
			\draw (3,0) ellipse (1.5  and 3.8 );
			\node at (-3,-4.5) {$A$};
			\node at (3,-4.5) {$B$}; 			
			\draw (-3,5) circle(0.5);
			\draw (3,5) circle (0.5);
			\node at (-3,6) {$A_0$};
			\node at (3,6) {$B_0$};
			\fill (-3,5) circle (4pt);
			\fill (3,5) circle (4pt);
			\begin{scope}[start chain]
			\foreach \i in {2.5,1.5,-0.5}
			\fill (-3,\i) circle (4pt);
			\end{scope}
			\begin{scope}[start chain]
			\foreach \i in {2.5,1.5,-0.5}
			\fill (3,\i) circle (4pt);
			\end{scope}

			\begin{scope}[line width=1.5pt]
			\draw (-3,-0.5) --(3,-0.5);
			\draw (-3,1.5)--(-3,2.5);
			\draw (3,2.5)--(3,1.5);
			\end{scope}
			
			\begin{scope}[line width=0.6pt]
			\draw (3,2.5)-- (-3,1.5);
			\draw (3,1.5)-- (-3,0.5)--(3,0.5)  -- (-3,-0.5);
			\draw (3,-0.5)--(-3,-1.5) -- (3,-1.5) --(-3,-2.5) -- (3,-2.5) -- (-3,2.5);
			\end{scope}
\end{tikzpicture}
}
\qquad
\subfloat[$J^*$]{
\begin{tikzpicture}[scale=0.35]
% 			\node at (0,-6) {$J$};
			\draw  (-3,0) ellipse (1.5  and 3.8 );
			\draw (3,0) ellipse (1.5  and 3.8 );
			\node at (-3,-4.5) {$A$};
			\node at (3,-4.5) {$B$}; 			
			\draw (-3,5) circle(0.5);
			\draw (3,5) circle (0.5);
			\node at (-3,6) {$A_0$};
			\node at (3,6) {$B_0$};
			\fill (-3,5) circle (4pt);
			\fill (3,5) circle (4pt);
			\begin{scope}[start chain]
			\foreach \i in {2.5,1.5,-0.5}
			\fill (-3,\i) circle (4pt);
			\end{scope}
			\begin{scope}[start chain]
			\foreach \i in {2.5,1.5,-0.5}
			\fill (3,\i) circle (4pt);
			\end{scope}

			\begin{scope}[line width=1.5pt]
			\draw (-3,-0.5)--(3,-0.5);
			\draw (-3,1.5)--(3,1.5);
			\draw (3,2.5)--(-3,2.5);
			\end{scope}

			\begin{scope}[line width=0.6pt]
			\draw (3,2.5)-- (-3,1.5);
			\draw (3,1.5)-- (-3,0.5)--(3,0.5)  -- (-3,-0.5);
			\draw (3,-0.5)--(-3,-1.5) -- (3,-1.5) --(-3,-2.5) -- (3,-2.5) -- (-3,2.5);
			\end{scope}
\end{tikzpicture}
}
\caption{The thick lines illustrate the edges of $J$, $J_{AB}^*$ and $J^*$ respectively.}
\label{fig3}
\end{figure}

\begin{prop}\label{CES-H}
Let $\cP=\{A_0,A_1,\dots,A_K,B_0,B_1,\dots,B_K\}$ be a $(K,m,\eps)$-partition of a vertex set $V$.
Let $G$ be a graph on $V$ and let $J$ be a balanced exceptional system with respect to~$\cP$.
\begin{itemize}
\item[{\rm (i)}] Assume that $P$ is an orientation of a subpath of $G[A\cup B]+ J^*$ such that $P$ is consistent with~$J^{*}.$ 
Then the graph obtained from $P-J^{*}+J$ by ignoring the orientations of the edges is a path on
$V(P) \cup V_0$ whose endvertices are the same as those of~$P$.
\item[{\rm (ii)}] If $J\subseteq G$ and
$D$ is a Hamilton cycle of $G[A\cup B]+ J^*$ which is consistent with~$J^{*},$ 
then $D-J^*+J$ is a Hamilton cycle of $G$.
\end{itemize}
\end{prop}
\proof
We first prove~(i). 
Let $s:=e(J^*_{AB}[A]) = e(J^*_{AB}[B])$ and $J^\diamond:=\{x_1y_1,\dots,x_{2s}y_{2s}\}$ (where the $x_i$ and $y_i$ are
as in the definition of $J^*$). So
$J^*:=J^\diamond\cup \{x_{2s+1}y_{2s+1},\dots,x_{s'}y_{s'}\}$, where $s':=e(J^*)$.
Let $P^c$ denote the path obtained from $P=z_1\dots z_2$ by reversing its direction. (So $P^c=z_2\dots z_1$ traverses
the vertices $y_{s'}, x_{s'}, y_{2s'-1}, \dots, x_2, y_1,x_1$ in this order.) 
First note
$$
P':=z_1Px_1x_2P^cy_1y_2Px_3x_4P^cy_3y_4\dots x_{2s-1}x_{2s}P^cy_{2s-1}y_{2s}Pz_2
$$
is a path on $V(P)$.
Moreover, the underlying undirected graph of $P'$ is%
\COMMENT{osthus added $=P-J^*+J^*_{AB}.$} 
precisely 
$$P-J^\diamond+(J^*_{AB}[A]\cup J^*_{AB}[B])=P-J^*+J^*_{AB}.
$$
In particular, $P'$ contains $J^*_{AB}$.
Now recall that if $w_1w_2$ is an edge in $J^*_{AB}$, then the vertices $w_1$ and $w_2$ are the endpoints of some path $P^*$ in $J$
(where the internal vertices on $P^*$ lie in $V_0$).
Clearly, $P'-w_1w_2+P^*$ is also a path. Repeating this step for every edge $w_1w_2$ of $J^*_{AB}$
gives a path $P''$ on $V(P)\cup V_0$. Moreover, $P''=P-J^{*}+J$. This completes the proof of~(i).

 (ii) now follows immediately from~(i).
\endproof

%%%%%%%%%%%%%%%%%%%%%%%%%%%%%%%%%%%%%%%%%%%%%%%%%%%%%%%%%%%%%%%%%%%%%%%%%%%%%%%

\subsection{Special path systems and special factors}\label{sec:SF}
As mentioned earlier, in order to apply Lemma~\ref{rdeclemma}, we first need  to prove the existence of certain `special path systems'.
These are defined below.

Suppose that 
$$
\mathcal{P}=\{A_0,A_1,\dots,A_K,B_0,B_1,\dots,B_K\}
$$ is a $(K,m,\eps_0)$-partition of a vertex set~$V$ and $L,m/L\in\mathbb{N}$.
We say that $(\mathcal{P},\mathcal{P'})$%
\COMMENT{the roles of $\mathcal{P},\mathcal{P'}$ were interchanged in a recent version. There may be places where this was overlooked.}
 is a \emph{$(K, L, m, \epszero)$-partition of~$V$} if $\mathcal{P'}$ is 
obtained from $\mathcal{P}$ by partitioning  $A_i$ 
into $L$ sets $A_{i,1},\dots,A_{i,L}$ of size~$m/L$ for all $1 \leq i \leq K$ and partitioning  $B_i$  into $L$ sets $B_{i,1},\dots,B_{i,L}$ of size~$m/L$ for all $1 \leq i \leq K$.
(So $\cP'$ consists of the exceptional sets $A_0$, $B_0$, the $KL$ clusters $A_{i,j}$ and the $KL$ clusters $B_{i,j}$.)
Unless stated otherwise, whenever considering a $(K, L, m, \epszero)$-partition $(\mathcal{P},\mathcal{P'})$ of a vertex set $V$
we use the above notation to denote the elements of $\mathcal P$ and $\mathcal P'$.

Let $(\mathcal{P},\mathcal{P'})$ be a $(K,L,m, \epszero)$-partition of~$V$.
Consider a spanning cycle $C = A_1 B_1 \dots A_K B_K$ on the clusters of $\mathcal{P}$.
Given an integer $f$ dividing $K$, the \emph{canonical interval partition} $\mathcal{I}$ of $C$ into $f$ intervals
consists of the intervals $$A_{(i-1)K/f+1} B_{(i-1)K/f+1} A_{(i-1)K/f+2} \dots B_{iK/f} A_{iK/f+1}$$ for all $i\le f$.
(Here $A_{K+1}:=A_1$.)

Suppose that $G$ is a digraph on~$V\setminus V_0$ and%
    \COMMENT{Have $V\setminus V_0$ instead of $V$ here since this is what we use later.}
$h\le L$. Let $I=A_jB_jA_{j+1}\dots A_{j'}$ be an interval in~$\mathcal{I}$.
A \emph{special path system $SPS$ of style~$h$ in~$G$ 
spanning the interval $I$} consists of precisely $m/L$ (non-trivial)
vertex-disjoint directed paths $P_1,\dots,P_{m/L}$ such that the following conditions hold:
\begin{itemize}
\item[(SPS1)] Every $P_s$ has its initial vertex in $A_{j,h}$ and its final vertex in $A_{j',h}$.
\item[(SPS2)] $SPS$ contains a matching ${\rm Fict}(SPS)$ such that all the edges in ${\rm Fict}(SPS)$ avoid the
endclusters $A_j$ and $A_{j'}$ of $I$ and such that $E(P_s)\setminus {\rm Fict}(SPS)\subseteq E(G)$.
\item[(SPS3)] The vertex set of $SPS$ is $A_{j,h}\cup B_{j,h}\cup A_{j+1,h}\cup \dots \cup B_{j'-1,h}\cup A_{j',h}$.
\end{itemize}
The edges in ${\rm Fict}(SPS)$ are called \emph{fictive edges of} $SPS$.

Let $\mathcal{I}=\{I_1,\dots,I_f\}$ be the canonical interval partition of $C$ into $f$ intervals.%
	\COMMENT{AL: added `the canonical interval partition of $C$ into $f$ intervals'}
A \emph{special factor $SF$ with parameters $(L,f)$ in $G$ (with respect to $C$, $\mathcal P'$)}
is a $1$-regular digraph on $V\setminus V_0$
which is the union of $Lf$ digraphs $SPS_{j,h}$ (one for all $j\le f$ and $h\le L$) such that each $SPS_{j,h}$ is
a special path system of style $h$ in $G$ which spans~$I_j$.
We write ${\rm Fict}(SF)$ for the union of the sets ${\rm Fict}(SPS_{j,h})$ over all $j\le f$ and $h\le L$
and call the edges in ${\rm Fict}(SF)$ \emph{fictive edges of $SF$}.

We will always view fictive edges as being distinct from each other and from the edges in other digraphs.
So if we say that special factors $SF_1,\dots,SF_r$ are pairwise edge-disjoint from each other and from some digraph $Q$ on $V\setminus V_0$,
then%
   \COMMENT{It's not enough to have this for all $Q\subseteq G$ since in the robust decomposition lemma $H$ need not be a subgraph of $G$.}
this means that $Q$ and all the $SF_i- {\rm Fict}(SF_i)$
are pairwise edge-disjoint, but for example there could be an edge from $x$ to $y$ in $Q$ as well as in ${\rm Fict}(SF_i)$ 
for several indices $i\le r$.
But these are the only instances of multiedges that we allow, i.e.~if there is more than one edge from $x$ to $y$, then all but
at most one of these edges are fictive edges.

\subsection{Balanced exceptional path systems and balanced exceptional factors}\label{sec:BEPS}
We now define balanced exceptional path systems BEPS. 
It will turn out that they (or rather their bipartite directed versions BEPS$^*_{\rm dir}$ involving fictive edges)
will satisfy the conditions of the special path systems defined above. 
Moreover, (bipartite) Hamilton cycles containing BEPS$^*_{\rm dir}$ correspond to Hamilton cycles in the `original' graph $G$
(see Proposition~\ref{prop:CEPSbiparite}).

Let $(\mathcal{P},\mathcal{P'})$ be a $(K,L,m, \epszero)$-partition of a vertex set $V$.
Suppose that $K/f\in\mathbb{N}$ and $h\le L$.
Consider a spanning cycle $C = A_1 B_1 \dots A_K B_K$ on the clusters of $\mathcal{P}$.
Let $\mathcal{I}$ be the canonical interval partition of $C$ into~$f$ intervals of equal size.
Suppose that $G$ is an oriented bipartite graph with vertex classes $A$ and $B$.
Suppose that $I=A_jB_j\dots A_{j'}$ is an interval in~$\mathcal{I}$.%
	\COMMENT{AL: changed line, Daniela added in~$\mathcal{I}$}
A \emph{balanced exceptional path system $BEPS$ of style~$h$ for $G$ spanning  $I$}
consists of precisely $m/L$ (non-trivial) vertex-disjoint undirected paths $P_1,\dots,P_{m/L}$ such that the following conditions hold:
\begin{enumerate}
\item[(BEPS1)] Every $P_s$ has one endvertex in $A_{j,h}$ and its other endvertex in $A_{j',h}$.
\item[(BEPS2)] $J := BEPS - BEPS[A,B]$ is a balanced exceptional system with respect to~$\cP$%
\COMMENT{AL: we cannot write 'There is a balanced exceptional system $J$\dots'. Because we want ALL balanced exceptional systems in BEPS to avoid $A_{j,h}$ and $A_{j',h}$.}
such that $P_1$ contains all edges of~$J$ and
so that the edge set of $J$ is disjoint from $A_{j,h}$ and $A_{j',h}$.%
    \COMMENT{Cannot just write $J\subseteq P_1$ since $J$ might contain vertices of degree 0 which don't belong to $P_1$.}
Let $P_{1,{\rm dir}}$ be the path obtained by orienting $P_1$ towards its endvertex in $A_{j',h}$
and let $J_{\rm dir}$ be the orientation of $J$ obtained in this way. Moreover, let $J^*_{\rm dir}$ be obtained from $J^*$ by orienting every edge in $J^*$
towards its endvertex in $B$. Then $P^*_{1,{\rm dir}}:=P_{1,{\rm dir}}-J_{\rm dir}+J^*_{\rm dir}$ is a directed path from $A_{j,h}$ to $A_{j',h}$ which is consistent with $J^*$.
\item[(BEPS3)] The vertex set of $BEPS$ is $V_0\cup A_{j,h}\cup B_{j,h}\cup A_{j+1,h}\cup \dots \cup B_{j'-1,h}\cup A_{j',h}$.
\item[(BEPS4)] For each $2\le s\le m/L$, define $P_{s,{\rm dir}}$ similarly as $P_{1,{\rm dir}}$.
Then $E(P_{s,{\rm dir}})\setminus E(J_{\rm dir}) \subseteq E(G)$ for every $1\le s\le m/L$.  
\end{enumerate}

Let%
    \COMMENT{Daniela deleted: Recall (BES2) that a balanced exceptional system does not contain $AB$-edges.
Since $G$ is bipartite, (BEPS2) and (BEPS4) implies that $J := BEPS - BEPS[A,B]$.} 
$BEPS^*_{\rm dir}$ be the path system consisting of $P^*_{1,{\rm dir}},P_{2,{\rm dir}},\dots,P_{m/L,{\rm dir}}$.
Then $BEPS^*_{\rm dir}$ is a special path system of style $h$ in $G$ which spans the interval~$I$ and such that ${\rm Fict}(BEPS^*_{\rm dir})=J^*_{\rm dir}$.

Let $\mathcal{I}=\{I_1,\dots,I_f\}$ be the canonical interval partition of $C$ into $f$ intervals.%
	\COMMENT{AL: added `the canonical interval partition of $C$ into $f$ intervals'}
A \emph{balanced exceptional factor $BF$ with parameters $(L,f)$ for $G$ (with respect to $C$, $\mathcal P'$)} is 
the union of $Lf$ undirected graphs $BEPS_{j,h}$ (one for all $j\le f$ and $h\le L$) such that each $BEPS_{j,h}$ is
a balanced exceptional path system of style $h$ for $G$ which spans~$I_j$. We write $BF^*_{\rm dir}$ for the union of $BEPS_{j,h,{\rm dir}}^*$ over all
$j\le f$ and $h\le L$. Note that $BF^*_{\rm dir}$ is a special factor with parameters $(L,f)$ in $G$ (with respect to $C$, $\mathcal P'$)
such that ${\rm Fict}(BF^*_{\rm dir})$ is the union of $J^*_{j,h,{\rm dir}}$ over all $j\le f$ and $h\le L$, where
$J_{j,h}=BEPS_{j,h}-BEPS_{j,h}[A,B]$ is%
    \COMMENT{Daniela added $=BEPS_{j,h}-BEPS_{j,h}[A,B]$}
the balanced exceptional system contained in $BEPS_{j,h}$ (see condition (BEPS2)). In particular, $BF^*_{\rm dir}$ is a $1$-regular digraph on $V\setminus V_0$ 
while  $BF$ is an undirected graph on $V$ with 
\begin{align}\label{EFdeg}
	d_{BF}(v) = 2 \ \ \text{for all } v \in V \setminus V_0 \ \ \ \textrm{ and } \ \ \ d_{BF}(v) = 2Lf \ \ \text{for all } v \in V_0.	
\end{align}

Given a balanced exceptional path system $BEPS$, let $J$ be as in (BEPS2) and
let $BEPS^*:=BEPS-J+J^*$. So $BEPS^*$ consists of $P^*_1:=P_1-J+J^*$ as well as $P_2,\dots,P_{m/L}$.
The following is an immediate consequence of (BEPS2) and Proposition~\ref{CES-H}.

\begin{prop} \label{prop:CEPSbiparite}
Let $(\cP,\cP')$ be a $(K,L, m , \epszero)$-partition of a vertex set $V$.
Suppose that $G$ is a graph on~$V\setminus V_0$, that $G_{\rm dir}$ is an orientation of $G[A,B]$ and that $BEPS$ is a balanced exceptional
path system for~$G_{\rm dir}$. Let $J$ be as in (BEPS2). Let $C$ be a Hamilton cycle of $G+J^*$ which
contains $BEPS^*$. Then $C - BEPS^*+BEPS$ is a Hamilton cycle of $G\cup J$.
\end{prop}
\proof
Note that $C - BEPS^*+BEPS=C-J^*+J$. Moreover, (BEPS2) implies that $C$ contains all edges of $J^*$ and is
consistent with $J^*$.%
   \COMMENT{Daniela: had "traverses the
vertices $x_1,y_1,x_2,\dots,y_{s'-1},x_{s'},y_{s'}$ in this order" instead of "is consistent with $J^*$"}
So the proposition follows from Proposition~\ref{CES-H}(ii)
applied with $G\cup J$ playing the role of~$G$.
\endproof

%%%%%%%%%%%%%%%%%%%%%%%%%%%%%%%%%%%%%%%%%%%%%%%%%%%%%%%555

\subsection{Finding balanced exceptional factors in a scheme}\label{sec:findBF}
The following definition of a `scheme' captures the `non-exceptional' part of the graphs we are working with.%
\COMMENT{osthus added para}
For example, this will be the structure within which we find the edges needed to extend a balanced exceptional system into a balanced exceptional path system.

Given an oriented graph $G$ and partitions $\mathcal P$ and $\cP'$ of a vertex set $V$, we call $(G, \mathcal{P},\mathcal{P}')$ a
\emph{$[K,L,m,\eps_0,\eps]$-scheme} if the following properties hold:
\begin{itemize}
\item[(Sch$1'$)] $(\mathcal{P},\mathcal{P}')$ is a $(K,L,m,\eps_0)$-partition of $V$. Moreover, $V(G)=A\cup B$.
\item[(Sch$2'$)] Every edge of $G$ has one endvertex in $A$ and its other endvertex in~$B$.
\item[(Sch$3'$)] $G[A_{i,j},B_{i',j'}]$ and $G[B_{i',j'}, A_{i,j}]$ are $[\eps, 1/2]$-superregular
for all $i, i'\le K$ and all $j, j' \leq L$. Further, $G[A_i,B_j]$ and $G[B_j,A_i]$ are $[\eps, 1/2]$-superregular
for all $i,j \leq K$.
\item[(Sch$4'$)] $|N_{G}^+(x)\cap N_{G}^-(y)\cap B_{i,j}|\ge (1-\eps) m/5L$ for all distinct $x,y\in A$, all $i\le K$ and all $j\le L$.
Similarly, $|N_{G}^+(x)\cap N_{G}^-(y)\cap A_{i,j}|\ge(1-\eps)  m/5L$ for all distinct $x,y\in B$, all $i\le K$ and all $j\le L$.\COMMENT{NOTE! added error term here!}
\end{itemize}
If $L=1$ (and so $\cP=\cP'$), then (Sch$1'$) just says that
$\mathcal{P}$ is a $(K,m,\eps_0)$-partition of $V(G)$.%
    \COMMENT{AL:removed notation for $[K,m, \eps_0, \eps]$}

The next lemma allows us to extend a suitable balanced exceptional system into a balanced exceptional path system.%
   \COMMENT{It does not seem to be possible to unify this in a simple way with the balancing lemma of the approximate cover chapter}
Given $h\le L$, we say that an $(i_1, i_2 , i_3 , i_4)$-BES $J$ has \emph{style $h$ (with respect to the $(K,L,m,\eps_0)$-partition
$(\mathcal{P},\mathcal{P}')$)} if all the edges of $J$
have their endvertices in $V_0\cup A_{i_1,h}\cup A_{i_2,h}\cup B_{i_3,h}\cup B_{i_4,h}$.

\begin{lemma} \label{lma:bipartite:CEPS}
Suppose that $K, L, n, m/L \in \mathbb{N}$, that $0 <1/n  \ll \eps, \eps_0\ll 1$ and $\eps_0\ll 1/K,1/L$.
Let $(G,\mathcal{P}, \mathcal{P}')$ be a $[K, L, m,\eps_0,\eps]$-scheme with $|V(G)\cup V_0|=n$.
Consider a spanning cycle $C = A_1 B_1 \dots A_K B_K$ on the clusters of $\mathcal{P}$ and let 
$I  = A_j B_j A_{j+1} \dots  A_{j'}$ be an interval on~$C$ of length at least~$10$.
Let $J$ be an $(i_1, i_2 , i_3 , i_4)$-BES of style $h\le L$ with parameter $\eps_0$%
    \COMMENT{We use $\eps_0$ for both the $(K,m, \epszero)$-partition $\cP'$ and the bound on $e(J)$. But this seems
to be ok for our applications. Also, can probably replace the 10 for the interval length by something smaller, but 10 seems safe.}
(with respect to $(\cP,\cP')$), for some $i_1, i_2, i_3 , i_4 \in \{j+1, \dots, j'-1\}$.
Then there exists a balanced exceptional path system of style $h$ for $G$
which spans the interval $I$ and contains all edges in~$J$.
\end{lemma}
\begin{proof}
For each $k\le 4$, let $m_k$ denote the number of vertices in $A_{i_k,h} \cup B_{i_k,h}$ which are incident to edges of~$J$.%
     \COMMENT{Cannot just write $m_k := |V(J) \cap (A_{i_k,h} \cup B_{i_k,h})|$ since $V(J)$ contains vertices of degree zero.}
We only consider the case when $i_1$, $i_2$, $i_3$ and $i_4$ are distinct and $m_k>0$ for each $k \le 4$, as the other cases can be
proved by similar arguments.%
\COMMENT{The argument is essentially identical in the other cases. Only now, we don't consider 
$m_1,\dots, m_4$ but a subset of them. For example, consider the case when
$i_1,\dots, i_4$ are all distinct except $i_3=i_4$ and each $m_i >0$. Then note that
$m_3=m_4$. So we now follow the argument as before but only use $m_1,m_2,m_3$.
So for example, we now have
\begin{align*}
	 |V(P^*_{1,{\rm dir}}) \cap A_{i,h}| &= 
	 \begin{cases}
	1	& \textrm{for $i \in \{j, \dots, j'\} \setminus \{i_1, i_2, i_3 \}$,}\\
	m_k	& \textrm{for $i = i_k$ and $k \le 3$,}\\
	0	& \textrm{otherwise.}
	 \end{cases}
\end{align*}
Further, now for each $k \le 3$, we choose $m_k-1$ paths $P_1^{k}, \dots, P_{m_{k}-1}^{k}$ in $G$. So now $Q$ is a path system consisting of $m_1+m_2+m_3 - 2 $ vertex-disjoint directed paths from $A_{j,h}$ to $A_{j',h}$.
\\
In cases where one of the $m_i=0$, we similarly `ignore' $m_i$ and follow the corresponding
calculations. } 
Clearly $m_1+\dots+m_4 \le  2\epszero n$ by (BES4).
For every vertex $x \in A$, we define $B(x)$ to be the cluster $B_{i,h}\in\mathcal{P}'$ such that $A_i$ contains $x$.
Similarly, for every $y \in B$, we define $A(y)$ to be the cluster $A_{i,h}\in\mathcal{P}'$ such that $B_i$ contains~$y$.%
     \COMMENT{Daniela replaced $\mathcal{P}$ by $\mathcal{P}'$ (twice)}

Let $x_1y_1, \dots, x_{s'} y_{s'}$ be the edges of $J^*$, with $x_i \in A$ and $y_i \in B$ for all $i \le s'$.
(Recall that the ordering of these edges is fixed in the definition of $J^*$.)
Thus $s' = (m_1+\dots+m_4)/2 \le  \epszero n $. Moreover, our assumption that $\eps_0\ll 1/K,1/L$ implies that
$\eps_0 n\le m/100L$ (say). Together with (Sch$4'$) this in turn ensures that 
for every $r \le s'$, we can pick vertices $w_r \in B (x_r) $ and $z_r \in A(y_r)$ such that $w_rx_r$, $y_rz_r$ and $z_r w_{r+1}$ are
(directed) edges in $G$
and such that all the $4s'$ vertices $x_r,y_r,w_r,z_r$ (for $r\le s'$) are distinct from each other.
Let $P_1'$ be the path $w_1 x_1 y_1 z_1 w_2 x_2 y_2 z_2 w_3 \dots y_{s'} z_{s'}$.
Thus $P_1'$ is a directed path from $B$ to $A$ in $G + J^*_{\rm dir}$ which is consistent with~$J^*$. 
(Here $J^*_{\rm dir}$ is obtained from $J^*$ by orienting every edge towards~$B$.%
\COMMENT{Deryk})
Note that $|V(P'_1) \cap A_{i_k,h}| = m_k = |V(P'_1) \cap B_{i_k,h}|$ for all $k \le 4$.
(This follows from our assumption that $i_1$, $i_2$, $i_3$ and $i_4$ are distinct.) Moreover,
$V(P'_1) \cap ( A_i \cup B_i ) = \emptyset$ for all $i \notin \{i_1, i_2, i_3, i_4 \}$.

Pick a vertex $z'$ in $A_{j,h}$ so that $z' w_1 $ is an edge of $G$.
Find a path $P''_1$ from $z_{s'}$ to $A_{j',h}$ in $G$ such that the vertex set of $P''_1$ consists of $z_{s'}$ and precisely one vertex in
each $A_{i,h}$ for all $i \in \{j+1, \dots, j' \} \setminus \{i_1, i_2, i_3, i_4 \}$ 
and one vertex in each $B_{i,h}$ for all $i \in \{j, \dots, j'-1 \} \setminus \{i_1, i_2, i_3, i_4 \}$ and no other vertices.
(Sch$4'$) ensures that this can be done greedily.%
    \COMMENT{Daniela replaced (Sch$3'$) by (Sch$4'$)}
Define $P^*_{1,{\rm dir}}$ to be the concatenation of $z'w_1$, $P'_1$ and $P''_1$.
Note that $P^*_{1,{\rm dir}}$ is a directed path from $A_{j,h}$ to $A_{j',h}$ in $G + J^*_{\rm dir}$ which is consistent with $J^*$.
Moreover, $V(P^*_{1,{\rm dir}}) \subseteq \bigcup_{i \le K} A_{i,h} \cup B_{i,h}$,%
\COMMENT{AL: added $V(P^*_{1,{\rm dir}}) \subseteq \bigcup_{i \le K} A_{i,h} \cup B_{i,h}$}
\begin{align*}
	 |V(P^*_{1,{\rm dir}}) \cap A_{i,h}| &= 
	 \begin{cases}
	1	& \textrm{for $i \in \{j, \dots, j'\} \setminus \{i_1, i_2, i_3, i_4 \}$,}\\
	m_k	& \textrm{for $i = i_k$ and $k \le 4$,}\\
	0	& \textrm{otherwise,}
	 \end{cases}
\end{align*}
while	 
\begin{align*}
	  	 |V(P^*_{1,{\rm dir}}) \cap B_{i,h}| & = 
	 \begin{cases}
	1	& \textrm{for $i \in \{j, \dots, j'-1\} \setminus \{i_1, i_2, i_3, i_4 \}$},\\
	m_k	& \textrm{for $i = i_k$ and $k \le 4$},\\
	0	& \textrm{otherwise.}
	 \end{cases}
\end{align*}
(Sch$4'$) ensures that for each $k \le 4$, there exist $m_k-1$ (directed) paths $P_1^{k}, \dots, P_{m_{k}-1}^{k}$ in $G$ such that 
\begin{itemize}
	\item $P_r^{k}$ is a path from $A_{j,h}$ to $A_{j',h}$ for each $r\le m_k-1$ and $k \le 4$;
	\item each $P_r^{k}$ contains precisely one vertex in $A_{i,h}$ for each $i\in \{j, \dots, j'\} \setminus \{i_k\}$,
one vertex in $B_{i,h}$ for each $i\in \{j, \dots, j'-1\} \setminus \{i_k\}$ and no other vertices;
	\item $P^*_{1,{\rm dir}},P_1^{1},\dots, P_{ m_{1}-1}^{1},P_1^2, \dots, P_{m_4-1}^4$  are vertex-disjoint.
\end{itemize}
Let $Q$ be the union of $P^*_{1,{\rm dir}}$ and all the $P_r^{k}$ over all $k \le 4$ and $r \le m_k-1$.
Thus $Q$ is a path system consisting of $m_1+\dots+m_4 - 3 $ vertex-disjoint directed paths from $A_{j,h}$ to $A_{j',h}$.
Moreover, $V(Q)$ consists of precisely $m_1+\dots+m_4 - 3 \le  2\epszero n$ vertices in $A_{i,h}$
for every $j\le i\le j'$ and precisely $m_1+\dots+m_4-3$ vertices in $B_{i,h}$ for every $j\le i<j'$. 
Set $A'_{i,h}: = A_{i,h} \setminus V(Q)$ and $B'_{i,h}:= B_{i,h} \setminus V(Q)$ for all $i \le K$.
Note that, for all $j \leq i \leq j'$,
\begin{equation}\label{eq:sizeAih}
|  A'_{i,h}|= \frac{m}{L}-(m_1+\dots+m_4-3)\ge \frac{m}{L}-2\eps_0 n\ge \frac{m}{L}-5\eps_0 mK\ge (1-\sqrt{\eps_0})\frac{m}{L}
\end{equation}
since $\eps_0\ll 1/K,1/L$. Similarly, $|  B'_{i,h}| \geq (1-\sqrt{\eps_0}){m}/{L}$ for all
 $j \leq i < j'$.
Pick a new constant $\eps'$ such that $\eps,\eps_0 \ll \eps'\ll 1$.
Then (Sch$3'$) and (\ref{eq:sizeAih}) together with Proposition~\ref{superslice} imply that 
$G[A'_{i,h}, B'_{i,h}]$ is still $[\eps',1/2]$-superregular 
and so we can find a perfect matching in $G[A_{i,h}',B_{i,h}']$
for all $j \le i < j'$. Similarly, we can find a perfect matching in $G[B_{i,h}',A_{i+1,h}']$ for all $j \le i < j'$.
The union $Q'$ of all these matchings forms $m/L-(m_1+\dots+m_4) +3$ vertex-disjoint directed paths.

Let $P_1$ be the undirected graph obtained from $P^*_{1,{\rm dir}}-J^*_{\rm dir}+J$ by ignoring the
directions of all the edges. Proposition~\ref{CES-H}(i) implies that $P_1$
is a path on $V(P^*_{1,{\rm dir}})\cup V_0$ with the same endvertices as $P^*_{1,{\rm dir}}$.
Consider the path system obtained from $(Q\cup Q') \setminus \{P^*_{1,{\rm dir}}\}$ by ignoring the directions of the edges on all
the paths. Let $BEPS$ be the union of this path system and $P_1$. Then $BEPS$ is a balanced exceptional path system for $G$, as required.
\end{proof}

The next lemma shows that we can obtain many edge-disjoint balanced exceptional factors by extending balanced exceptional systems with suitable properties.

\begin{lemma} \label{lma:EF-bipartite}
Suppose that $L,f,q,n,m/L,K/f \in \mathbb{N}$, that $K/f\ge 10$, that $0 <1/n \ll \eps,\eps_0  \ll 1$,
that $\eps_0\ll 1/K,1/L$ and $Lq/m\ll 1$.
Let $(G,\mathcal{P}, \mathcal{P}')$ be a $[K, L, m,\eps_0,\eps]$-scheme with $|V(G)\cup V_0|=n$.
Consider a spanning cycle $C = A_1 B_1 \dots A_K B_K$ on the clusters of $\mathcal{P}$.
%Let $\mathcal{I} = \{I_1, \dots, I_{f}\}$ be the canonical interval partition of $C$ into $f$ intervals.
Suppose that there exists a set $\mathcal{J}$ of $Lf q $ edge-disjoint balanced exceptional systems
with parameter $\eps_0$ such that 
\begin{itemize}
	\item for all $i \le f$ and all $h \le L$, $\mathcal{J}$ contains precisely $q$ $(i_1,i_2,i_3,i_4)$-BES of style $h$
(with respect to $(\cP,\cP')$) for which $i_1,i_2,i_3,i_4 \in \{ (i-1)K/f+2, \dots, iK/f \}$.
\end{itemize}
Then there exist $q$ edge-disjoint balanced exceptional factors with parameters~$(L,f)$ for $G$ (with respect to $C$, $\mathcal P'$) covering all
edges in~$\bigcup\mathcal{J}$.
\end{lemma}

Recall that the canonical interval partition $\mathcal{I}$ of $C$ into $f$ intervals consists of the intervals
$$A_{(i-1)K/f+1} B_{(i-1)K/f+1} A_{(i-1)K/f+2} \dots A_{iK/f+1}$$ for all $i\le f$. So the condition on~$\mathcal{J}$ ensures
that for each interval $I\in \mathcal{I}$ and each $h\le L$, the set $\mathcal{J}$ contains precisely $q$
balanced exceptional systems of style $h$ whose edges are only incident to vertices in $V_0$ and vertices
belonging to clusters in the interior of $I$. We will use Lemma~\ref{lma:bipartite:CEPS} to extend each such balanced exceptional system into
a balanced exceptional path system of style $h$ spanning~$I$.

\removelastskip\penalty55\medskip\noindent{\bf Proof of Lemma~\ref{lma:EF-bipartite}. }
Choose a new constant $\eps'$ with $\eps,Lq/m\ll \eps'\ll 1$.
Let $\mathcal{J}_{1}, \dots, \mathcal{J}_{q }$ be a partition of $\mathcal{J}$ such that for all $j\le q$, $h \le L$ and $i \le f$, 
the set $\mathcal{J}_j$ contains precisely one $(i_1,i_2,i_3,i_4)$-BES of style $h$ with $i_1,i_2,i_3,i_4 \in \{ (i-1)K/f+2, \dots, iK/f \}$.
Thus each $\mathcal{J}_j$ consists of $Lf$ balanced exceptional systems.
For each $j\le q$ in turn, we will choose a balanced exceptional factor $EF_j$ with parameters $(L,f)$ for $G$
such that $BF_j$ and $BF_{j'}$ are edge-disjoint for all $j' < j$ and $BF_j$ contains all edges
of the balanced exceptional systems in $\mathcal{J}_j$.
Assume that we have already constructed $BF_1, \dots, BF_{j-1}$. In order to construct $BF_j$, we will choose the $Lf$ balanced exceptional path systems
forming $BF_j$ one by one, such that each of these balanced exceptional path systems is edge-disjoint from $BF_1,\dots,BF_{j-1}$ and
contains precisely one of the balanced exceptional systems in $\mathcal{J}_j$. Suppose that we have already chosen some of these
balanced exceptional path systems and that next we wish to choose a balanced exceptional path system of style $h$ which spans the interval $I\in \mathcal{I}$
of $C$ and contains $J\in \mathcal{J}_j$.
Let $G'$ be the oriented graph obtained from $G$ by deleting all the edges in the balanced path systems already chosen for $BF_j$
as well as deleting all the edges in $BF_1,\dots,BF_{j-1}$.
Recall from (Sch1$'$) that $V(G)=A\cup B$. Thus $\Delta(G - G') \le 2j < 3q $ by~\eqref{EFdeg}.%
\COMMENT{Deryk added ref to Sch1'}
Together with Proposition~\ref{superslice} this implies that $(G',\mathcal{P},\mathcal{P}')$ is still a $[K, L, m, \epszero, \eps']$-scheme.
(Here we use that $\Delta(G - G') < 3 q=3Lq/m\cdot m/L$ and
$\eps,Lq/m\ll \eps'\ll 1$.) So we can apply Lemma~\ref{lma:bipartite:CEPS} with $\eps'$ playing the role of $\eps$ to obtain a
balanced exceptional path system of style $h$ for $G'$ (and thus for $G$) which spans $I$ and contains all edges of $J$.
This completes the proof of the lemma.
\endproof

%%%%%%%%%%%%%%%%%%%%%%%%%%%%%%%%%%%%%%%%%%%%%%%%%%%%%%%%%%%%%%%%%%%%%%%%%%%%%%%%%%%%%%%%%%%%%%%%%%%%%%%%%%%
%%%%%%%%%%%%%%%%%%%%%%%%%%%%%%%%%%%%%%%%%%%%%%%%%%%%%%%%%%55
%%%%%%%%%%%%%%%%%%%%%%%%%%%%%%%%%%%%%%%%%%%%%%%%%%%%%%%%

\section{The robust decomposition lemma}\label{sec:robust}

The robust decomposition lemma (Corollary~\ref{rdeccor}) allows us to transform an approximate Hamilton decomposition into an exact one.
As discussed in Section~\ref{sec:sketch}, it will only be used in the proof of Theorem~\ref{1factbip} (and not in the proof of Theorem~\ref{NWmindegbip}).
In the next subsection, we introduce the necessary concepts. In particular, Corollary~\ref{rdeccor} relies on the existence of a
so-called bi-universal walk.%
    \COMMENT{Deryk added new para and osthus added final sentence from 2clique paper}
The (proof of the) robust decomposition lemma then uses
edges guaranteed by this universal walk to `balance out' edges of the graph $H$ when constructing the Hamilton decomposition of
$G^{\rm rob}+H$.
\subsection{Chord sequences and bi-universal walks}
Let $R$ be a digraph whose vertices are $V_1,\dots,V_k$ and suppose that $C=V_1\dots V_k$ is a Hamilton cycle of $R$.
(Later on the vertices of $R$ will be clusters. So we denote them by capital letters.)

A \emph{chord sequence $CS(V_i,V_j)$} from $V_i$ to $V_j$ in $R$ is an ordered sequence of edges of the form
\[ CS(V_i,V_j) = (V_{i_1-1} V_{i_2}, V_{i_2-1} V_{i_3},\dots, V_{i_t-1} V_{i_{t+1}}),\]
where $V_{i_1}=V_i$, $V_{i_{t+1}} = V_j$ and the edge $V_{i_s-1} V_{i_{s+1}}$ belongs to $R$ for each $s\le t$.

If $i=j$ then we consider the empty set to be a chord sequence from $V_i$ to $V_j$. 
Without loss of generality, we may assume that $CS(V_i,V_j)$ does not contain any edges of $C$.
(Indeed, suppose that  $V_{i_s-1} V_{i_{s+1}}$ is an edge of $C$.
Then $i_s=i_{s+1}$ and so we can obtain a chord sequence from $V_i$ to $V_j$ with fewer edges.) 
For example, if $V_{i-1}V_{i+2}\in E(R)$, then the edge $V_{i-1}V_{i+2}$ is a chord sequence from $V_i$ to $V_{i+2}$.

The crucial property of chord sequences is that they satisfy a `local balance' condition. 
Suppose that $CS$ is obtained by concatenating several chord sequences 
$$CS(V_{i_1},V_{i_2}),CS(V_{i_2},V_{i_3}),\dots,CS(V_{i_{\ell-1}},V_{i_\ell}), CS(V_{i_{\ell}},V_{i_{\ell+1}})$$ 
where $V_{i_1}=V_{i_{\ell+1}}$.
Then for every $V_i$, the number of edges of $CS$ leaving $V_{i-1}$ equals the number of edges entering $V_i$.
We will not use this property explicitly, but it underlies the proofs of e.g.~Lemma~\ref{rdeclemma} 
and appears implicitly e.g.~in~(BU3) below.

A closed walk $U$ in $R$ is a \emph{bi-universal walk for $C$
with parameter $\ell'$} if the following conditions hold:
\begin{itemize}
\item[(BU1)] The edge set of $U$ has a partition into $U_{\rm odd}$ and $U_{\rm even}$.
For every $1\le i\le k$ there is a chord sequence $ECS^{\rm bi}(V_i,V_{i+2})$
from $V_i$ to $V_{i+2}$ such that $U_{\rm even}$ contains all edges of all these chord sequences for even $i$ (counted with multiplicities)
and $U_{\rm odd}$ contains all edges of these chord sequences for odd $i$.
All remaining edges of $U$ lie on $C$.
\item[(BU2)] Each $ECS^{\rm bi}(V_i,V_{i+2})$ consists of at most $\sqrt{\ell'}/2$ edges.
\item[(BU3)] $U_{\rm even}$ enters every cluster
$V_i$ exactly $\ell'/2$ times and it leaves every cluster $V_i$ exactly $\ell'/2$ times.
The same assertion holds for $U_{\rm odd}$.%
\COMMENT{So need to make sure $\ell'$ is even when we apply this.}
\end{itemize} 
Note that condition~(BU1) means that if an edge $V_iV_j\in E(R)\setminus E(C)$ occurs in total 5 times (say) in
$ECS^{{\rm bi}}(V_1,V_3),\dots,ECS^{{\rm bi}}(V_{k},V_2)$ then it occurs precisely 5 times in $U$. We will identify each occurrence of $V_iV_j$ in
$ECS^{{\rm bi}}(V_1,V_3),\dots,ECS^{{\rm bi}}(V_{k},V_2)$ with a (different) occurrence of $V_iV_j$ in $U$. 
Note that the edges of $ECS^{{\rm bi}}(V_i,V_{i+2})$ are allowed to appear in a different order % within $ECS^{{\rm bi}}(V_i,V_{i+2})$ and 
within $U$.

\begin{lemma}\label{lem:univwalk}
Let $R$ be a digraph with vertices $V_1,\dots,V_k$ where $k \geq 4$ is even.
 Suppose that $C=V_1\dots V_k$ is a Hamilton cycle of $R$
and that $V_{i-1}V_{i+2}\in E(R)$ for every $1\le i\le k$. Let $\ell'\ge 4$ be an even integer.%
\COMMENT{need 4 rather than 2 to ensure $1 \le \sqrt{\ell'}/2$.}
 Let $U_{{\rm bi},\ell'}$ denote the multiset obtained from
$\ell'-1$ copies of $E(C)$ by adding
$V_{i-1}V_{i+2}\in E(R)$ for every $1\le i\le k$. Then the edges in $U_{{\rm bi},\ell'}$ can be ordered so that the
resulting sequence forms a bi-universal walk for $C$ with parameter~$\ell'$.
\end{lemma}
In the remainder of the paper, we will also write $U_{{\rm bi},\ell'}$ for the bi-universal walk guaranteed by Lemma~\ref{lem:univwalk}. 

\proof
Let us first show that the edges in $U_{{\rm bi},\ell'}$ can be ordered so that the resulting sequence forms a closed walk in~$R$.
To see this, consider the multidigraph $U$ obtained from $U_{{\rm bi},\ell'}$ by deleting one copy of~$E(C)$.
Then $U$ is $(\ell'-1)$-regular and thus has a decomposition into 1-factors. We order the edges of $U_{{\rm bi},\ell'}$ as follows: 
We first traverse all cycles of the 1-factor decomposition of $U$ which contain the cluster $V_1$.
Next, we traverse the edge $V_1V_2$ of $C$. Next we traverse all those cycles of the 1-factor decomposition which contain
$V_2$ and which have not been traversed so far. Next we traverse the edge $V_2V_3$ of $C$ and so on
until we reach $V_1$ again. 

Recall that, for each $1\le i\le k$, the edge $V_{i-1}V_{i+2}$ is a chord sequence from $V_i$ to $V_{i+2}$. Thus we can take
$ECS^{{\rm bi}}(V_i,V_{i+2}):=V_{i-1}V_{i+2}$. Then $U_{{\rm bi},\ell'}$ satisfies (BU1)--(BU3).
Indeed, (BU2) is clearly satisfied. Partition one of the copies of $E(C)$ in $U_{{\rm bi},\ell'}$
into $E_{\rm even}$ and $E_{\rm odd}$ where $E_{\rm even}=\{V_iV_{i+1} | \ i \text{ even}\}$
and $E_{\rm odd}=\{V_iV_{i+1} | \ i \text{ odd}\}$. Note that the union of $E_{\rm even}$ together
with all $ECS^{{\rm bi}}(V_i,V_{i+2})$ for even $i$ is a $1$-factor in $R$. Add $\ell'/2-1$ of the
remaining copies of $E(C)$ to this $1$-factor to obtain $U_{\rm even}$. Define $U_{\rm odd}$
to be $E(U_{{\rm bi},\ell'})\setminus U_{\rm even}$. By construction of $U_{\rm even}$ and $U_{\rm odd}$,
(BU1) and (BU3) are satisfied.
\endproof

\subsection{Bi-setups and the robust decomposition lemma} 
The%
\COMMENT{osthus adapted para from 2clique paper}
aim of this subsection is to state the robust decomposition lemma (Lemma~\ref{rdeclemma}, proved in~\cite{Kelly})
and derive Corollary~\ref{rdeccor}, which we shall use later on.
The robust decomposition lemma guarantees the existence of a `robustly decomposable' digraph $G^{\rm rob}_{\rm dir}$
within a `bi-setup'. Roughly speaking, a bi-setup is a digraph $G$ together with its `reduced digraph' $R$,
which contains a Hamilton cycle $C$ and a universal walk $U$.
In our application, $G[A,B]$ will play the role of $G$ and $R$ will be the complete bipartite digraph.
To define a bi-setup formally, we first need to define certain `refinements' of partitions.

Given a digraph $G$ and a partition $\cP$ of $V(G)$ into $k$ clusters $V_1,\dots,V_k$ of equal size,
we say that a partition $\cP'$ of $V(G)$%
   \COMMENT{Daniela: had $V$ instead of $V(G)$}
is an \emph{$\ell'$-refinement of $\cP$} if $\cP'$ is obtained by splitting each $V_i$
into $\ell'$ subclusters of equal size. (So $\cP'$ consists of $\ell'k$ clusters.)
$\cP'$ is an \emph{$\eps$-uniform $\ell'$-refinement}
of $\cP$ if it is an $\ell'$-refinement of $\cP$ which satisfies the following condition:
Whenever $x$ is a vertex of $G$, $V$ is a cluster in $\cP$ and $|N^+_G(x)\cap V|\ge \eps |V|$
then $|N^+_G(x)\cap V'|=(1\pm \eps)|N^+_G(x)\cap V|/\ell'$%
\COMMENT{AL: added prime, Daniela added two more primes}
 for each cluster $V'\in \cP'$ with $V'\subseteq V$.
The inneighbourhoods of the vertices of $G$ satisfy an analogous condition.%
   \COMMENT{Daniela added new lemma since we need it in the proof of Lemma~\ref{lem:bisetup}, the lemma in~\cite{Kelly} actually allows
to have $V_0$, so I changed it slightly}
We will use the following lemma from~\cite{Kelly}.   

\begin{lemma}\label{randompartition}
Suppose that $0<1/m \ll 1/k,\eps \ll \eps', d,1/\ell \le 1$ and that $k,\ell, m/\ell\in\mathbb{N}$.
Suppose that $G$ is a digraph and that $\cP$ is a partition of $V(G)$ into
$k$ clusters of size $m$. Then there exists an $\eps$-uniform $\ell$-refinement of $\cP$. Moreover, any
$\eps$-uniform $\ell$-refinement $\cP'$ of $\cP$ automatically satisfies
the following condition:
\begin{itemize}
\item Suppose that $V$, $W$ are clusters in $\cP$ and $V',W'$ are clusters in $\cP'$ with $V'\subseteq V$ and
$W'\subseteq W$. If $G[V,W]$ is $[\eps,d']$-superregular for some $d'\ge d$ then $G[V',W']$ is $[\eps',d']$-superregular.
\end{itemize}
\end{lemma}

We will also need the following definition from~\cite{Kelly}, which describes the structure within which the robust decomposition lemma
finds the robustly decomposable graph.%
\COMMENT{Deryk added sentence}
$(G,\cP,\cP',R,C,U,U')$ is called an \emph{$(\ell',k,m,\eps,d)$-bi-setup} if the following properties are satisfied:
\begin{itemize}
\item [(ST1)] $G$ and $R$ are digraphs. $\mathcal{P}$ is a partition of $V(G)$ into
$k$ clusters of size $m$ where $k$ is even. The vertex set of $R$ consists of these clusters.
\item[(ST2)] For every edge $VW$ of $R$, the corresponding pair $G[V,W]$ is $(\eps,\ge d)$-regular.
\item[(ST3)] $C=V_1\dots V_{k}$ is a Hamilton cycle of $R$ and for every edge $V_iV_{i+1}$ of $C$ the corresponding pair $G[V_i,V_{i+1}]$ is $[\eps,\ge d]$-superregular.
%\item[(ST4)] $V_0$ forms an independent set in $G$.
\item[(ST4)] $U$ is a bi-universal walk for $C$ in $R$ with parameter~$\ell'$ and $\cP'$ is an $\eps$-uniform $\ell'$-refinement%
   \COMMENT{It was decided by DK+DO we need this notion of refinement here}
of $\cP$.
\item[(ST5)] Let $V_j^1,\dots,V_j^{\ell'}$ denote the clusters in $\cP'$ which are contained
in $V_j$ (for each $1\le j\le k$). Then $U'$ is a closed walk on the clusters in $\cP'$ which is obtained from $U$ as follows:
When $U$ visits $V_j$ for the $a$th time, we let $U'$ visit the subcluster $V_j^a$ (for all $1\le a\le \ell'$).
\item[(ST6)] For every edge $V_{i}^jV_{i'}^{j'}$ of $U'$ the corresponding pair $G[V_{i}^j,V_{i'}^{j'}]$ is $[\eps,\ge d]$-superregular.%
	\COMMENT{AL: added more details}
\end{itemize}
In~\cite{Kelly}, in a bi-setup, the digraph $G$ could also contain an exceptional set, but since we are only using
the definition in the case when there is no such exceptional set, we have only stated it in this special case.

Suppose that $(G,\mathcal{P},\mathcal{P}')$ is a $[K,L,m,\eps_0,\eps]$-scheme and that $C=A_1B_1\dots A_KB_K$ is a spanning cycle
on the clusters of $\mathcal{P}$. Let $\mathcal{P}_{{\rm bi}}:=\{A_1,\dots,A_K,B_1,\dots,B_K\}$. Suppose that $\ell',m/\ell'\in\mathbb{N}$ with
$\ell'\ge 4$. Let $\mathcal{P}''_{{\rm bi}}$ be an $\eps$-uniform $\ell'$-refinement of $\cP_{{\rm bi}}$ (which exists by Lemma~\ref{randompartition}).%
    \COMMENT{Daniela added brackets}
Let $C_{{\rm bi}}$ be the directed cycle obtained from $C$
in which the edge $A_1B_1$ is oriented towards $B_1$ and so on. Let $R_{{\rm bi}}$ be the complete bipartite digraph
whose vertex classes are $\{A_1,\dots,A_K\}$ and $\{B_1,\dots,B_K\}$.%
   \COMMENT{Daniela: had "whose vertices are the clusters in $\cP$"}
Let $U_{{\rm bi},\ell'}$ be a bi-universal walk for $C$
with parameter $\ell'$ as defined in Lemma~\ref{lem:univwalk}. Let $U'_{{\rm bi},\ell'}$ be the closed walk
obtained from $U_{{\rm bi},\ell'}$ as described in~(ST5). We will call
$$
(G,\cP_{{\rm bi}},\cP''_{{\rm bi}},R_{{\rm bi}},C_{{\rm bi}},U_{{\rm bi},\ell'},U'_{{\rm bi},\ell'})$$ the \emph{bi-setup
associated to~$(G,\mathcal{P},\mathcal{P}')$}. The following lemma shows that it is indeed a bi-setup.%
   \COMMENT{Daniela changed lemma below to bring it in line with Lemma~\ref{randompartition}}

\begin{lemma}\label{lem:bisetup}
Suppose that $K,L,m/L, \ell', m/\ell'\in\mathbb{N}$ with $\ell'\ge 4$, $K \geq 2$ and $0<1/m \ll 1/K,\eps \ll \eps',1/\ell'$.
Suppose that $(G,\mathcal{P},\mathcal{P}')$ is a $[K,L,m,\eps_0,\eps]$-scheme and that $C=A_1B_1\dots A_KB_K$ is a spanning cycle
on the clusters of $\mathcal{P}$. Then $$(G,\cP_{{\rm bi}},\cP''_{{\rm bi}},R_{{\rm bi}},C_{{\rm bi}},U_{{\rm bi},\ell'},U'_{{\rm bi},\ell'})$$
is an $(\ell',2K,m,\eps',1/2)$-bi-setup.
\end{lemma}
\proof
Clearly, $(G,\cP_{{\rm bi}},\cP''_{{\rm bi}},R_{{\rm bi}},C_{{\rm bi}},U_{{\rm bi},\ell'},U'_{{\rm bi},\ell'})$ satisfies~(ST1).
(Sch$3'$) implies that (ST2) and~(ST3) hold.%
    \COMMENT{This is now immediate because of change of def of (Sch3$'$)}
Lemma~\ref{lem:univwalk} implies~(ST4). (ST5) follows from the definition of $U'_{{\rm bi},\ell'}$. Finally, (ST6) follows from (Sch3$'$) and
Lemma~\ref{randompartition} since $\cP''_{{\rm bi}}$ is an $\eps$-uniform $\ell'$-refinement of~$\cP_{{\rm bi}}$.
\endproof

We now state the robust decomposition lemma from~\cite{Kelly}.
This guarantees the existence of a `robustly decomposable' digraph $G^{\rm rob}_{\rm dir}$, whose crucial property
is that $H + G^{\rm rob}_{\rm dir}$ has a Hamilton decomposition for any sparse bipartite regular digraph~$H$
which is edge-disjoint from $G^{\rm rob}_{\rm dir}$.%
    \COMMENT{Daniela added "which is..."}

$G^{\rm rob}_{\rm dir}$ consists of digraphs $CA_{{\rm dir}}(r)$ (the `chord absorber') and $PCA_{{\rm dir}}(r)$ 
(the `parity extended cycle switcher')
together with some special factors. $G^{\rm rob}_{\rm dir}$ is constructed in two steps:
given a suitable set $\mathcal{SF}$ of special factors, the lemma first `constructs' $CA_{{\rm dir}}(r)$ and then,
given another suitable set $\mathcal{SF}'$ of special factors, the lemma `constructs' $PCA_{{\rm dir}}(r)$.
The reason for having two separate steps is that in~\cite{Kelly}, it is not clear how to construct $CA_{{\rm dir}}(r)$
after constructing $\mathcal{SF}'$ (rather than before), as the removal of $\mathcal{SF}'$ from the digraph under consideration 
affects its properties considerably.
\begin{lemma} \label{rdeclemma}
Suppose that $0<1/m\ll 1/k\ll \eps \ll 1/q \ll 1/f \ll r_1/m\ll d\ll 1/\ell',1/g\ll 1$ 
where $\ell '$ is even and%
   \COMMENT{In the Kelly paper have $1/n$ instead of $1/m$ in the hierarchy. But since $V_0=\emptyset$ in our setting,
it doesn't make that much sense to introduce $n$ here. So I replaced $1/n$ by $1/m$.}
that $rk^2\le m$. Let
$$r_2:=96\ell'g^2kr, \ \ \ r_3:=rfk/q, \ \ \ r^\diamond:=r_1+r_2+r-(q-1)r_3, \ \ \ s':=rfk+7r^\diamond
$$
and suppose that $k/14, k/f, k/g, q/f, m/4\ell', fm/q, 2fk/3g(g-1) \in \mathbb{N}$.
Suppose that $(G,\cP,\cP',R,C,U,U')$ is an $(\ell',k,m,\eps,d)$-bi-setup and $C=V_1\dots V_k$.
Suppose that $\cP^*$ is a $(q/f)$-refinement
of $\cP$ and that $SF_1,\dots, SF_{r_3}$ are edge-disjoint special factors with parameters $(q/f,f)$ 
with respect to $C$, $\cP^*$ in $G$. Let $\mathcal{SF}:=SF_1+\dots +SF_{r_3}$.
Then there exists a digraph $CA_{{\rm dir}}(r)$ for which the following holds:
\begin{itemize}
\item[\rm (i)] $CA_{{\rm dir}}(r)$ is an $(r_1+r_2)$-regular spanning subdigraph of $G$ which is edge-disjoint from $\mathcal{SF}$.
\item[\rm (ii)] Suppose that $SF'_1,\dots, SF'_{r^\diamond}$ are special factors with parameters $(1,7)$
with respect to $C$, $\cP$ in $G$ which are edge-disjoint from each other and from $CA_{{\rm dir}}(r)+ \mathcal{SF}$.%
   \COMMENT{In the Kelly paper we write $CA_{{\rm dir}}(r)\cup \mathcal{SF}$ instead of $CA_{{\rm dir}}(r)+ \mathcal{SF}$
(and similarly below). But with our def of $+$ and $\cup $ in this paper, we have to use $+$ since we allow for a fictive edge
$xy$ in $\mathcal{SF}$ to also occur in $CA_{{\rm dir}}(r)$, and in this case $CA_{{\rm dir}}(r)+ \mathcal{SF}$
willl contain two copies of that edge.} 
Let $\mathcal{SF}':=SF'_1+\dots +SF'_{r^\diamond}$.
Then there exists a digraph $PCA_{{\rm dir}}(r)$ for which the following holds:
\begin{itemize}
\item[\rm (a)] $PCA_{{\rm dir}}(r)$ is a $5r^\diamond$-regular spanning subdigraph of $G$ which
is edge-disjoint from $CA_{{\rm dir}}(r)+ \mathcal{SF}+ \mathcal{SF}'$.
\item[(b)] Let $\mathcal{SPS}$ be the set consisting of all the $s'$ special path systems
contained in $\mathcal{SF}+ \mathcal{SF}'$. Let $V_{\rm even}$ denote the union of all $V_i$ over all even $1\le i\le k$
and define $V_{\rm odd}$ similarly.
Suppose that $H$ is an $r$-regular bipartite digraph on $V(G)$ with vertex classes $V_{\rm even}$ and $V_{\rm odd}$
which is edge-disjoint from $G^{\rm rob}_{\rm dir}:=CA_{{\rm dir}}(r)+ PCA_{{\rm dir}}(r)+ \mathcal{SF}+ \mathcal{SF}'$.
Then $H+G^{\rm rob}_{\rm dir}$ has a decomposition into $s'$
edge-disjoint Hamilton cycles $C_1,\dots,C_{s'}$.
Moreover, $C_i$ contains one of the special path systems from $\mathcal{SPS}$, for each $i\le s'$.
\end{itemize}
\end{itemize}
\end{lemma}

Recall from Section~\ref{sec:SF} that we always view fictive edges in special factors as being distinct from each other and
from the edges in other graphs. So for example, saying that $CA_{{\rm dir}}(r)$ and $\mathcal{SF}$ are edge-disjoint in
Lemma~\ref{rdeclemma} still allows for a fictive edge $xy$ in $\mathcal{SF}$ to occur in $CA_{{\rm dir}}(r)$ as well
(but $CA_{{\rm dir}}(r)$ will avoid all non-fictive edges in $\mathcal{SF}$).
 
We will use the following `undirected' consequence of Lemma~\ref{rdeclemma}.

\begin{cor} \label{rdeccor}
Suppose that $0<1/m\ll \eps_0,1/K\ll \eps \ll 1/L \ll 1/f \ll r_1/m\ll 1/\ell',1/g\ll 1$ 
where $\ell '$ is even and
that $4rK^2\le m$.%
\COMMENT{osthus: changed from $4rK^2$} Let
$$r_2:=192\ell'g^2Kr, \ \ \ r_3:=2rK/L, \ \ \ r^\diamond:=r_1+r_2+r-(Lf-1)r_3, \ \ \ s':=2rfK+7r^\diamond
$$
and suppose that $L, K/7, K/f, K/g, m/4\ell', m/L, 4fK/3g(g-1) \in \mathbb{N}$.%
   \COMMENT{Daniela added $L\in \mathbb{N}$}
Suppose that $(G_{\rm dir},\cP,\cP')$ is a $[K,L,m,\eps_0,\eps]$-scheme and let $G'$ denote the underlying undirected graph
of $G_{\rm dir}$. Let $C=A_1B_1\dots A_KB_K$ be a spanning cycle on the clusters in $\cP$. 
Suppose that $BF_1,\dots, BF_{r_3}$ are edge-disjoint balanced exceptional factors with parameters $(L,f)$ for $G_{\rm dir}$ 
(with respect to $C$, $\mathcal P'$). Let $\mathcal{BF}:=BF_1+\dots + BF_{r_3}$.
Then there exists a graph $CA(r)$ for which the following holds:
\begin{itemize}
\item[(i)] $CA(r)$ is a $2(r_1+r_2)$-regular spanning subgraph of $G'$ which is edge-disjoint from $\mathcal{BF}$.
\item[(ii)] Suppose that $BF'_1,\dots, BF'_{r^\diamond}$ are balanced exceptional factors with parameters $(1,7)$ for $G_{\rm dir}$
(with respect to $C$, $\mathcal P$) which are edge-disjoint from each other and from $CA(r)+ \mathcal{BF}$.
Let $\mathcal{BF}':=BF'_1+\dots + BF'_{r^\diamond}$.
Then there exists a graph $PCA(r)$ for which the following holds:
\begin{itemize}
\item[(a)] $PCA(r)$ is a $10r^\diamond$-regular spanning subgraph of $G'$ which
is edge-disjoint from $CA(r)+ \mathcal{BF}+ \mathcal{BF}'$.
\item[(b)] Let $\mathcal{BEPS}$ be the set consisting of all the $s'$ balanced exceptional path systems
contained in $\mathcal{BF}+ \mathcal{BF}'$.
Suppose that $H$ is a $2r$-regular bipartite graph on $V(G_{\rm dir})$ with vertex classes $\bigcup_{i=1}^K A_i$ and $\bigcup_{i=1}^K B_i$
which is edge-disjoint from $G^{\rm rob}:=CA(r)+ PCA(r)+ \mathcal{BF}+ \mathcal{BF}'$.
Then $H+ G^{\rm rob}$ has a decomposition into $s'$
edge-disjoint Hamilton cycles $C_1,\dots,C_{s'}$.
Moreover, $C_i$ contains one of the balanced exceptional path systems from $\mathcal{BEPS}$, for each $i\le s'$.
\end{itemize}
\end{itemize}
\end{cor}
We remark that we write%
\COMMENT{osthus adapted sentence from paper4} 
$A_1,\dots,A_K,B_1,\dots,B_K$ for the clusters in $\cP$. Note that the vertex set of each of $\mathcal{EF}$, $\mathcal{EF}'$, $G^{\rm rob}$
includes $V_0$ while that of $G_{\rm dir}$, $CA(r)$, $PCA(r)$, $H$ does not.
Here $V_0=A_0\cup B_0$, where $A_0$ and $B_0$ are the exceptional sets of $\cP$.
%Note that $H+G^{\rm rob}$ already includes the exceptional set $V_0$ (whereas the statement of Lemma~\ref{rdeclemma}
%does not involve an exceptional set).%
 %  \COMMENT{Daniela deleted "We remark that the `moreover' part of Lemma~\ref{rdeclemma}(ii)(b) is used to prove Corollary~\ref{rdeccor}, but that the `moreover' part of 
%Corollary~\ref{rdeccor}(ii)(b) is actually not needed when we apply the corollary. This is due to the fact .." - not quite sure whether the last part is true...}

\proof
Choose new constants $\eps'$ and $d$ such that $\eps\ll \eps'\ll 1/L$ and $r_1/m\ll d\ll 1/\ell',1/g$.%
   \COMMENT{Daniela: introduced $\eps'$ which is needed since Lemma~\ref{lem:bisetup} changed. Before we had that
$(G_{\rm dir},\cP_{{\rm bi}},\cP''_{{\rm bi}},R_{{\rm bi}},C_{{\rm bi}},U_{{\rm bi},\ell'},U'_{{\rm bi},\ell'})$
is an $(\ell',2K,m,\eps^{1/2},1/2)$-bi-setup instead of just an $(\ell',2K,m,\eps',1/2)$-bi-setup. So check whether I made all the necessary changes below.}
Consider the bi-setup $(G_{\rm dir},\cP_{\rm bi},\cP''_{{\rm bi}},R_{{\rm bi}},C_{{\rm bi}},U_{{\rm bi},\ell'},U'_{{\rm bi},\ell'})$
associated to $(G_{\rm dir},\cP,\cP')$. By Lemma~\ref{lem:bisetup}, 
$(G_{\rm dir},\cP_{{\rm bi}},\cP''_{{\rm bi}},R_{{\rm bi}},C_{{\rm bi}},U_{{\rm bi},\ell'},U'_{{\rm bi},\ell'})$
is an $(\ell',2K,m,\eps',1/2)$-bi-setup and thus also an $(\ell',2K,m,\eps',d)$-bi-setup.
Let $BF^*_{i,{\rm dir}}$ be as defined in Section~\ref{sec:BEPS}.
Recall from there that, for each $i\le r_3$, $BF^*_{i,{\rm dir}}$ is a special factor
with parameters $(L,f)$ with respect to $C$, $\mathcal P'$ in $G_{\rm dir}$ such that
${\rm Fict}(BF^*_{i,{\rm dir}})$ consists of all the edges in the $J^*$ for all the $Lf$ balanced exceptional systems $J$ contained in $BF_i$.
Thus we can apply Lemma~\ref{rdeclemma} to $(G_{\rm dir},\cP_{{\rm bi}},\cP''_{{\rm bi}},R_{{\rm bi}},C_{{\rm bi}},U_{{\rm bi},\ell'},U'_{{\rm bi},\ell'})$
with $2K$, $Lf$, $\eps'$ playing the roles of $k$, $q$, $\eps$
in order to obtain a spanning subdigraph $CA_{{\rm dir}}(r)$
of $G_{\rm dir}$ which satisfies Lemma~\ref{rdeclemma}(i). Hence the underlying undirected graph $CA(r)$ of
$CA_{{\rm dir}}(r)$ satisfies Corollary~\ref{rdeccor}(i). Indeed, to check that $CA(r)$ and $\mathcal{BF}$ are edge-disjoint, by Lemma~\ref{rdeclemma}(i)%
\COMMENT{osthus added by Lemma~\ref{rdeclemma}(i)} 
it suffices to
check that $CA(r)$ avoids all edges in all the balanced exceptional systems $J$ contained in $BF_i$ (for all $i\le r_3$). But this
follows since $E(G_{\rm dir})\supseteq E(CA(r))$ consists only of $AB$-edges by (Sch2$'$) and since no balanced exceptional system
contains an $AB$-edge by~(BES2).%
   \COMMENT{Daniela added the last 2 sentences}

Now let $BF'_1,\dots, BF'_{r^\diamond}$ be balanced exceptional factors as described in Corollary~\ref{rdeccor}(ii).
Similarly as before, for each $i\le r^\diamond$, $(BF'_i)^*_{\rm dir}$ is a special factor
with parameters $(1,7)$ with respect to $C$, $\mathcal P$ in $G_{\rm dir}$ such that
${\rm Fict}((BF'_i)^*_{\rm dir})$ consists of all the edges in the $J^*$ over all the $7$ balanced exceptional systems $J$ contained in $BF'_i$.
Thus we can apply Lemma~\ref{rdeclemma} to obtain a spanning subdigraph $PCA_{{\rm dir}}(r)$
of $G_{\rm dir}$ which satisfies Lemma~\ref{rdeclemma}(ii)(a) and~(ii)(b). Hence the underlying undirected graph
$PCA(r)$ of $PCA_{{\rm dir}}(r)$ satisfies Corollary~\ref{rdeccor}(ii)(a).

It remains to check that Corollary~\ref{rdeccor}(ii)(b) holds too. Thus let $H$ be as described in Corollary~\ref{rdeccor}(ii)(b).
Let $H_{{\rm dir}}$ be an $r$-regular orientation of $H$. (To see that such an orientation exists, apply Petersen's theorem to obtain
a decomposition of $H$ into $2$-factors and then orient each $2$-factor to obtain a (directed) $1$-factor.)%
     \COMMENT{However, even if $H\subseteq G'$, this orientation might not agree with $G_{\rm dir}$, i.e.~we might not have that
$H\subseteq G_{\rm dir}$. This is the reason why we cannot assume that $H\subseteq G$ in Lemma~\ref{rdeclemma}(ii)(b).
($H\subseteq G$ would fit better to the remainder of the Kelly paper.)}
Let $\mathcal{BF}^*_{\rm dir}$ be the union of the $BF^*_{i,{\rm dir}}$ over all $i\le r_3$
and let $(\mathcal{BF}')^*_{\rm dir}$ be the union of the $(BF'_i)^*_{\rm dir}$ over all $i\le r^\diamond$.
Then Lemma~\ref{rdeclemma}(ii)(b) implies that
$H_{{\rm dir}}+ CA_{{\rm dir}}(r)+ PCA_{{\rm dir}}(r)+ \mathcal{BF}^*_{\rm dir}+ (\mathcal{BF}')^*_{\rm dir}$
has a decomposition into $s'$ edge-disjoint (directed) Hamilton cycles $C'_1,\dots,C'_{s'}$ such that each $C'_i$ contains
$BEPS^*_{i,{\rm dir}}$ for some balanced exceptional path system $BEPS_i$ from $\mathcal{BEPS}$.
Let $C_i$ be the undirected graph obtained from $C'_i-BEPS^*_{i,{\rm dir}}+BEPS_i$ by ignoring the directions of all
the edges. Then Proposition~\ref{prop:CEPSbiparite} (applied with $G'$ playing the role of $G$) implies that $C_1,\dots,C_{s'}$ is
a decomposition of $H+ G^{\rm rob}=H+ CA(r)+ PCA(r)+ \mathcal{BF}+ \mathcal{BF}'$ into edge-disjoint Hamilton cycles.
\endproof 

%%%%%%%%%%%%%%%%%%%%%%%%%%%%%%%%%%%%%%%%%%%%%%%%%%%%%%%%%%%%%%%%%%%%%%%%%%%%%%%%%%%%%%%%%%%%%%%%%%%%%%%%%%%%%%%%%%%%%%%

\section{Proof of Theorem~\ref{NWmindegbip}}\label{sec:proof1}

The proof of Theorem~\ref{NWmindegbip} is similar to that of Theorem~\ref{1factbip} except that we  do not need to apply
the robust decomposition lemma in the proof of Theorem~\ref{NWmindegbip}. For both results, we will need an approximate decomposition result
(Lemma~\ref{almostthm}), which is stated below and proved as Lemma~3.2 in~\cite{paper3}.%
    \COMMENT{osthus changed the last sentence to refer to 3.2}
The lemma extends a suitable set of balanced exceptional systems into a set of edge-disjoint Hamilton cycles covering most edges of an almost complete 
and almost balanced bipartite graph.%
\COMMENT{osthus added extra sentence}
\begin{lemma}\label{almostthm}
Suppose that $0<1/n \ll \eps_0  \ll 1/K \ll \rho  \ll 1$ and $0 \leq \mu \ll 1$,
where $n,K \in \mathbb N$ and $K$ is even.
Suppose that $G$ is a graph on $n$ vertices and $\mathcal{P}$ is a $(K, m, \eps _0)$-partition of $V(G).$
Furthermore, suppose that the following conditions hold:%
	\COMMENT{Previously we had `$(G[A,B],\mathcal{P})$ is a $(K, m, \eps _0, \eps )$-scheme'. now removed.}
\begin{itemize}
	\item[{\rm (a)}] $d(w,B_i) = (1 - 4 \mu \pm 4 /K) m $ and $d(v,A_i) = (1 - 4 \mu \pm 4 /K) m $ for all
	$w \in A$, $v \in B$ and $1\leq i \leq K$.%
	\COMMENT{Daniela swapped $v$ and $w$, now it is the same as in the proofs}
	\item[{\rm (b)}] There is a set $\mathcal J$ which consists of at most $(1/4-\mu - \rho)n$ edge-disjoint balanced exceptional systems with parameter $\eps_0$ in~$G$.
	\item[{\rm (c)}] $\mathcal J$ has a partition into $K^4$ sets $\mathcal J_{i_1,i_2,i_3,i_4}$ (one for all $1\le \I \le K$) such that each $\mathcal J_{\I}$ consists of precisely $|\mathcal J|/{K^4}$ $\i$-BES with respect to~$\cP$.
   \item[{\rm (d)}] Each $v \in A \cup B$ is incident with an edge in $J$ for at most $2 \eps_0 n $ $J \in \mathcal{J}$.%
   \COMMENT{This is a new property. This condition is implied by the fact that $(G[A,B],\mathcal{P})$ is a $(K, m, \eps _0, \eps )$-scheme.($|A_0 \cup B_0| \le \eps_0 n$
and $\Delta( G[A]),\Delta( G[B]) \le \eps_0 n$.)}
\end{itemize}
Then $G$ contains $|\mathcal J|$ edge-disjoint Hamilton cycles such that each of these Hamilton cycles contains some $J\in\mathcal{J}$.
\end{lemma}

To%
\COMMENT{osthus added para} 
prove Theorem~\ref{NWmindegbip}, we find a framework via Corollary~\ref{coverA0B02c}.
Then we choose suitable balanced exceptional systems using Corollary~\ref{BEScor}.
Finally, we extend these into Hamilton cycles using Lemma~\ref{almostthm}.

\removelastskip\penalty55\medskip\noindent{\bf Proof of Theorem~\ref{NWmindegbip}. }
\noindent\textbf{Step 1: Choosing the constants and a framework.}
By%
   \COMMENT{Daniela changed this proof substantially, eg schemes are not used anymore. So read all of it again}
making $\alpha$ smaller if necessary, we may assume that $\alpha\ll 1$.
Define new constants such that 
\begin{align}
0 & < 1/n_0 \ll \eps_{\rm ex} \ll \epszero \ll \eps'_0\ll \eps' \ll \eps_1 \ll \eps_2 \ll \eps_3 \ll
\eps_4 \ll 1/K\ll \alpha \ll \eps\ll 1, \nonumber
\end{align}
where $K\in  \mathbb{N}$ and $K$ is even.%
\COMMENT{$\eps_3$ and $\eps_4$ are implicitly used when applying Corollary~\ref{BEScor}. }

Let $G$, $F$ and $D$ be as in Theorem~\ref{NWmindegbip}. 
Apply Corollary~\ref{coverA0B02c} with $\eps_{\rm ex}$, $\eps_0$ playing the role of $\eps$, $\eps^*$
to find a set $\cC_1$ of at most $\eps _{\rm ex} ^{1/3} n$ edge-disjoint Hamilton cycles in $F$ so that the graph $G_1$ obtained from $G$
by deleting all the edges in these Hamilton cycles forms part of an
$(\eps_0,\eps',K,D_1)$-framework $(G_1,A,A_0,B,B_0)$ with $D_1 \ge D-2\eps_{\rm ex}^{1/3}n$.
Moreover, $F$ satisfies~(WF5) with respect to $\eps '$ and
\begin{equation}\label{eq:sizeC1}
|\cC_1|=(D-D_1)/2.
\end{equation}
In particular, this implies that $\delta (G_1) \geq D_1$ and that $D_1$ is even (since $D$ is even). 
Let $F_1$ be the graph obtained from $F$ by deleting all those edges lying on Hamilton cycles
in $\cC_1$. Then 
\begin{equation}\label{eq:degF1}
\delta(F_1) \ge \delta(F)-2|\cC_1|\ge (1/2-3\eps_{\rm ex}^{1/3})n.
\end{equation}
 Let
\begin{align*}
m :=\frac{|A|}{K}=\frac{|B|}{K} \qquad \text{and} \qquad t_{K}:=\frac{(1-20\eps_4)D_1}{2K^4}.
\end{align*}
By changing $\eps_4$ slightly, we may assume that $t_K\in\mathbb{N}$.

\smallskip

\noindent\textbf{Step 2: Choosing a $(K,m,\eps_0)$-partition $\cP$.}
Apply Lemma~\ref{part} to $(G_1,A,A_0,B,B_0)$ with $F_1$, $\eps_0$
playing the roles of $F$, $\eps$ in order to obtain partitions $A_1,\dots,A_{K}$ and
$B_1,\dots,B_{K}$ of $A$ and $B$ into sets of size $m$ such that together with $A_0$ and $B_0$ the sets
$A_i$ and $B_i$ form a $(K,m,\eps_0,\eps_1,\eps_2)$-partition $\cP$ for $G_1$.%

Note that by Lemma~\ref{part}(ii) and since $F$ satisfies (WF5),
for all $x \in A$ and $1 \leq j \leq K$, we have
\begin{eqnarray}\label{eq:degF'1}
d_{F_1}(x,B_j)& \ge & \frac{d_{F_1}(x,B) - \eps_1n}{K} \stackrel{{\rm (WF5)}}{\ge} \frac{d_{F_1}(x)-\eps' n -|B_0|- \eps_1 n}{K}\nonumber\\
& \stackrel{(\ref{eq:degF1})}{\ge} & \frac{(1/2-3\eps_{\rm ex}^{1/3})n- 2\eps_1 n}{K}
\ge (1-5\eps_1)m_.
\end{eqnarray}
Similarly, $d_{F_1}(y,A_i)\ge (1-5\eps_1)m$%
\COMMENT{Andy: deleted a prime}
for all $y\in B$ and $1 \leq i \leq K$.

\smallskip

\noindent\textbf{Step 3: Choosing balanced exceptional systems for the almost decomposition.}
Apply Corollary~\ref{BEScor} to the $(\eps_0,\eps',K,D_1)$-framework $(G_1,A,A_0,B,B_0)$ with
$F_1$, $G_1$, $\eps_0$, $\eps'_0$, $D_1$ playing the roles
of $F$, $G$, $\eps$, $\eps_0$, $D$. Let $\cJ'$ be the union of the sets $\cJ_{i_1i_2i_3i_4}$ guaranteed by Corollary~\ref{BEScor}.
So $\cJ'$ consists of $K^4t_{K}$ edge-disjoint balanced exceptional systems with parameter $\eps'_0$ in $G_1$ (with respect to~$\mathcal{P}$).
Let $\cC_2$ denote the set of $10 \eps _4 D_1$ Hamilton cycles guaranteed by Corollary~\ref{BEScor}.
Let $F_2$ be the subgraph obtained from $F_1$ by deleting all the Hamilton cycles in~$\cC_2$.  
Note that 
\begin{equation}\label{eq:D2}
D_2:=D_1 -2|\mathcal C_2| =(1-20 \eps _4)D_1=2K^4 t_K =2|\cJ'|.
\end{equation}
\noindent\textbf{Step 4: Finding the remaining Hamilton cycles.}
Our next aim is to apply Lemma~\ref{almostthm} with $F_2$, $\cJ'$,  $\eps'$ playing the
roles of $G$, $\cJ$,  $\eps_0$.

Clearly, condition~(c) of Lemma~\ref{almostthm} is satisfied.
In order to see that condition (a) is satisfied, let $\mu:=1/K$ and note that for all $w\in A$ we have 
$$
d_{F_2}(w,B_i)\ge d_{F_1}(w,B_i)-2|\cC_2|\stackrel{(\ref{eq:degF'1})}{\ge} (1-5\eps_1)m-20\eps_4 D_1\ge (1-1/K)m.
$$
Similarly $d_{F_2}(v,A_i)\ge (1-1/K)m$ for all $v\in B$.

To check condition~(b), note that
$$|\cJ'|\stackrel{(\ref{eq:D2})}{=} \frac{D_2}{2}\le \frac{D}{2}\le (1/2-\alpha)\frac{n}{2}\le (1/4-\mu-\alpha/3)n .$$
Thus condition~(b) of Lemma~\ref{almostthm} holds with $\alpha/3$ playing the role of $\rho$.
Since the edges in $\cJ'$ lie in $G_1$ and  $(G_1,A,A_0,B,B_0)$ is an $(\eps_0,\eps',K,D_1)$-framework,
(FR5) implies that each $v \in A \cup B$ is incident with an edge in $J$ for at most $\eps ' n+|V_0|\leq 2\eps 'n$
$J \in \cJ '$. 
(Recall that in a balanced exceptional system there are no edges between $A$ and $B$.)
So condition~(d) of Lemma~\ref{almostthm} holds with $\eps '$ playing the role of $\eps _0$.

So we can indeed apply Lemma~\ref{almostthm} to obtain a collection $\cC_3$ of $|\cJ'|$ edge-disjoint Hamilton cycles in $F_2$
which cover all edges of $\bigcup \cJ'$. Then $\cC_1\cup \cC_2\cup \cC_3$ is a set of edge-disjoint Hamilton cycles in~$F$
of size
$$|\cC_1|+|\cC_2|+|\cC_3|\stackrel{(\ref{eq:sizeC1}),(\ref{eq:D2})}{=} \frac{D-D_1}{2}+ \frac{D_1-D_2}{2}+ \frac{D_2}{2}=\frac{D}{2},$$
as required.
\endproof

%%%%%%%%%%%%%%%%%%%%%%%%%%%%%%%%%%%%%%%%%%%%%%%%%%%%%%%%%%%%%%%%%%%%%%%%%%%%%%%%%%%%%%%%%%%%%%%%%%%%%%%%%%%%%%%%%%%%%%%

\section{Proof of Theorem~\ref{1factbip}}\label{sec:proof2}

As mentioned earlier, the proof of Theorem~\ref{1factbip} is similar to that of Theorem~\ref{NWmindegbip} except that we will also need to apply 
the robust decomposition lemma. This means Steps 2--4 and Step 8 in the proof of Theorem~\ref{1factbip} do not appear
in the proof of Theorem~\ref{NWmindegbip}.
Steps 2--4 prepare the ground for the application of the robust decomposition lemma and in Step~8 we apply it to cover the leftover from the approximate decomposition step
with Hamilton cycles. Steps 5--7 contain the approximate decomposition step, using Lemma~\ref{almostthm}.
\COMMENT{osthus added extra sentence}

In our proof of Theorem~\ref{1factbip} it will be convenient to work with an undirected version of the schemes introduced in Section~\ref{sec:findBF}.
Given a graph $G$ and partitions $\mathcal P$ and $\cP'$ of a vertex set $V$, we call $(G, \mathcal{P},\mathcal{P}')$ a
\emph{$(K,L,m,\eps_0, \eps)$-scheme} if the following properties hold:%
   \COMMENT{Daniela changed def to make $G$ into a bipartite graph with vertex classes $A$ and $B$ (ie replacing the old $G$ by $G-V_0$).
Also deleted the def of a $(K,m,\eps_0, \eps)$-scheme} 
\begin{itemize}
\item[(Sch$1$)] $(\mathcal{P},\mathcal{P}')$ is a $(K,L,m,\eps_0)$-partition of $V$. Moreover, $V(G)=A\cup B$.
\item[(Sch$2$)] Every edge of $G$ joins some vertex in $A$ to some vertex in~$B$.
\item[(Sch$3$)] $d_G(v,A_{i,j}) \geq (1 - \eps) m/L $ and $d_G(w,B_{i,j}) \geq (1 - \eps) m/L $ for all
    $v \in B$, $w \in A$, $ i \leq K$ and $j\leq L$.
\end{itemize}
%Note that if $L=1$ (and so $\cP'=\cP$) then (Sch1) just says that $\mathcal{P}'$ is a $(K,m,\eps_0)$-partition of~$V$.
We will also use the following proposition.

\begin{prop}\label{lem:dirscheme}
Suppose that $K, L, n, m/L \in \mathbb{N}$ and $0 <1/n  \ll \eps, \eps_0\ll 1$.
Let $(G,\mathcal{P}, \mathcal{P}')$ be a $(K, L, m, \epszero, \eps)$-scheme with $|G| = n$.
Then there exists an orientation $G_{\rm dir}$ of $G$%
   \COMMENT{Daniela replaced $G-V_0$ by $G$ (also twice in the first para of the proof)}
such that $(G_{\rm dir}, \mathcal{P},\mathcal{P}')$ is a
$[K,L,m,\eps_0,2\sqrt{\eps}]$-scheme.
\end{prop}
\proof 
Randomly orient every edge in $G$ to obtain an oriented graph $G_{\rm dir}$. (So given any edge $xy$ in $G$
with probability $1/2$, $xy \in E(G_{\rm dir})$ and with probability $1/2$, $yx \in E(G_{\rm dir})$.)
(Sch$1'$) and (Sch$2'$) follow immediately from (Sch$1$) and (Sch$2$).

Note that Fact~\ref{simplefact} and (Sch$3$) imply that $G[A_{i,j},B_{i',j'}]$ is $[1, \sqrt{\eps}]$-superregular
with density at least $1-\eps$,
for all $i,i'\leq K$ and $j,j' \leq L$.
Using this, (Sch$3'$) follows easily from the large deviation bound in Proposition~\ref{chernoff}.
(Sch$4'$) follows from Proposition~\ref{chernoff} in a similar way.%
\endproof

\removelastskip\penalty55\medskip\noindent{\bf Proof of Theorem~\ref{1factbip}. }

\noindent\textbf{Step 1: Choosing the constants and a framework.}
Define new constants such that%
    \COMMENT{$\eps_3$ and $\eps_4$ are implicitly used when applying Corollary~\ref{BEScor}. } 
\begin{align}
0 & < 1/n_0 \ll \eps_{\rm ex} \ll \eps_* \ll \epszero \ll \eps'_0\ll \eps' \ll \eps_1 \ll \eps_2  \ll \eps_3 \ll
\eps_4\ll 1/K_2  \\
 &  \ll \gamma \ll 1/K_1 \ll  \eps''\ll 1/L \ll 1/f \ll \gamma_1  \ll 1/g \ll \eps\ll 1, \nonumber
\end{align}
where $K_1, K_2, L, f, g \in  \mathbb{N}$ and both $K_2$, $g$ are even.%
   \COMMENT{Daniela added $g$ even, we need this when applying Cor~\ref{rdeccor} with $\ell':=g$} 
Note that we can choose the constants such that
$$\frac{K_1}{28fgL}, \frac{K_2}{4gLK_1}, \frac{4fK_1}{3g(g-1)} \in \mathbb{N}.$$

Let $G$ and $D$ be as in Theorem~\ref{1factbip}. By applying Dirac's theorem to remove a suitable number of edge-disjoint Hamilton cycles if necessary,
we may assume that $D\le n/2$. 
Apply Corollary~\ref{coverA0B02c} with $G$, $\eps_{\rm ex}$, $\eps_*$, $\eps _0$, $K_2$ playing the roles of $F$, $\eps$, $\eps ^*$, $\eps'$, $K$
to find a set $\cC_1$ of 
at most $\eps _{\rm ex} ^{1/3} n$ 
edge-disjoint Hamilton cycles in $G$ so that the graph $G_1$ obtained from $G$
by deleting all the edges in these Hamilton cycles forms part of an
$(\eps_*,\eps_0,K_2,D_1)$-framework $(G_1,A,A_0,B,B_0)$, where%
   \COMMENT{Daniela added $=D-2|\cC_1|$ below, osthus added $|A|+\eps_0 n \ge n/2$ } 
\begin{equation} \label{D1eq}
|A|+\eps_0 n \ge n/2 \ge D_1 =D-2|\cC_1|\ge D-2\eps_{\rm ex}^{1/3}n\ge D-\eps_0 n \ge n/2-2\eps_0 n \ge |A|-2\eps_0 n.
\end{equation}
Note that  $G_1$ is $D_1$-regular and that $D_1$ is even since $D$ was even.
Moreover, since $K_2/LK_1\in\mathbb{N}$, $(G_1,A,A_0,B,B_0)$ is also an $(\eps_*,\eps_0,K_1L,D_1)$-framework
and thus an $(\eps_*,\eps',K_1L,D_1)$-framework.

Let%
     \COMMENT{$r^\diamond$ is $t$ in the $2$-clique case}
\begin{align*}
m_1 & :=\frac{|A|}{K_1}=\frac{|B|}{K_1}, \qquad r:=  \gamma m_1, \qquad
r_1 := \gamma_1 m_1, \qquad
%s:=rfK, \qquad
r_2:= 192g^3K_1r, \\
r_3 & := \frac{2rK_1}{L}, \qquad r^\diamond := r_1 +r_2 +r -(Lf -1)r_3, \\
D_4 & :=D_1-2(Lfr_3+7r^\diamond), \qquad t_{K_1L} :=\frac{(1-20\eps_4)D_1}{2(K_1L)^4} .
%\phi_0 := \frac{ s +7t}{n} +\epszero  = \frac{Lfr_3 +7t}n +\epszero.
\end{align*}
Note that (FR4) implies $m_1/L\in \mathbb{N}$. Moreover, 
\begin{equation}\label{eq:rs}
r_2,r_3\le \gamma^{1/2}m_1\le \gamma^{1/3} r_1,  \qquad  r_1/2\le r^\diamond\le 2r_1.
\end{equation}
Further, by changing $\gamma,\gamma_1,\eps_4$ slightly, we may assume that $r/K^2_2,r_1,t_{K_1L}\in\mathbb{N}$.
Since $K_1/L\in\mathbb{N}$ this implies that $r_3\in\mathbb{N}$.
Finally, note that 
\begin{equation} \label{D4bound}
(1+3\eps_*)|A| \ge D \ge D_4  \stackrel{(\ref{eq:rs})}{\ge} D_1-\gamma_1 n \stackrel{(\ref{D1eq})}{\ge} |A|-2\gamma_1 n \ge (1-5\gamma_1)|A|.
\end{equation}

\noindent\textbf{Step 2: Choosing a $(K_1, L, m_1, \epszero)$-partition $(\mathcal{P}_1,\mathcal{P}'_1)$.}
We now prepare the ground for the construction of the robustly decomposable graph $G^{\rm rob}$, 
which we will obtain via the robust decomposition lemma (Corollary~\ref{rdeccor}) in Step~4.

Recall that $(G_1,A,A_0,B,B_0)$ is an $(\eps_*,\eps',K_1L,D_1)$-framework.
Apply Lemma~\ref{part} with $G_1$, $D_1$, $K_1L$, $\eps_*$ playing the roles of $G$, $D$, $K$, $\eps$ to obtain partitions
$A'_1,\dots,A'_{K_1L}$ of $A$ and $B'_1,\dots,B'_{K_1L}$ of $B$
into sets of size $m_1/L$ such that together with $A_0$ and $B_0$ all
these sets $A'_i$ and $B'_i$ form a $(K_1L,m_1/L,\eps_*,\eps_1,\eps_2)$-partition
$\cP'_1$ for $G_1$. 
Note that $(1-\eps_0)n\le n-|A_0\cup B_0|=2K_1m_1\le n$ by (FR4).
For all $i\le K_1$ and all $h\le L$, let $A_{i,h}:=A'_{(i-1)L+h}$. (So this is just a relabeling of the sets $A'_i$.)
Define $B_{i,h}$ similarly and let
$A_i:= \bigcup_{h\le L} A_{i,h}$ and $B_i:= \bigcup_{h\le L} B_{i,h}$.
Let $\cP_1:=\{A_0,B_0,A_1,\dots,A_{K_1},B_1,\dots,B_{K_1}\}$ denote the
corresponding $(K_1,m_1,\eps_0)$-partition of $V(G)$. Thus $(\cP_1,\cP'_1)$ is a $(K,L,m_1,\eps_0)$-partition of $V(G)$,
as defined in Section~\ref{sec:SF}.
 
Let $G_2:=G_1[A,B]$.%
   \COMMENT{Daniela: changed def because of change in def of (..)-system, before had "Let $G_2$ be the spanning subgraph of $G_1$ obtained by
keeping all $AB$-edges and deleting all other edges". So $G_2$ included $V_0$.}
We claim that $(G_2, \mathcal{P}_1,\mathcal{P'}_1)$ is a $(K_1, L, m_1, \epszero,\eps')$-scheme. Indeed, clearly~(Sch1) and~(Sch2) hold.
To verify (Sch3), recall that that $(G_1,A,A_0,B,B_0)$ is an $(\eps_*,\eps_0,K_1L,D_1)$-framework and so%
   \COMMENT{Daniela added more detail}
by (FR5) for all $x \in B$ we have 
$$
d_{G_2}(x,A)\ge d_{G_1}(x)-d_{G_1}(x,B')-|A_0| \ge D_1 - \eps_0n -|A_0| \stackrel{(\ref{D1eq})}{\ge} |A| -4 \eps_0 n
$$ and
similarly $d_{G_2}(y,B)\ge |B| -4 \eps_0 n$ for all $y\in A$. Since $\eps_0\ll \eps'/K_1L$, this implies~(Sch3).\COMMENT{Here we are just using that $\eps_0 n \ll \eps' |A_{i,j}|$, not e.g. (P2)}

\smallskip

\noindent\textbf{Step 3: Balanced exceptional systems for the robustly decomposable graph.}
In order to apply Corollary~\ref{rdeccor}, we first need to construct suitable balanced exceptional systems.
Apply Corollary~\ref{BEScor} to the $(\eps_*,\eps',K_1L,D_1)$-framework $(G_1,A,A_0,B,B_0)$
with $G_1$, $K_1L$, $\cP'_1$, $\eps_*$ playing the roles of $F$, $K$, $\cP$, $\eps$
in order to obtain a set $\cJ$ of $(K_1L)^4t_{K_1L}$ edge-disjoint balanced exceptional systems in $G_1$ with parameter $\eps_0$
such that for all $1 \le i'_1,i'_2,i'_3,i'_4  \le K_1L$ the set $\cJ$ contains precisely $t_{K_1L}$
$(i'_1,i'_2,i'_3,i'_4)$-BES with respect to the partition $\cP'_1$.
(Note that $F$ in Corollary~\ref{BEScor} satisfies (WF5) since $G_1$ satisfies (FR5).)%
\COMMENT{osthus added bracket}
So $\cJ$ is the union of all the sets $\cJ_{i'_1i'_2i'_3i'_4}$ returned by
Corollary~\ref{BEScor}. (Note that we will not use all the balanced exceptional systems in~$\cJ$ and
we do not need to consider the Hamilton cycles guaranteed by this result. 
So we do not need the full strength of Corollary~\ref{BEScor} at this point.)

Our next aim is to choose two disjoint subsets $\mathcal{J}_{\rm CA}$ and $\mathcal{J}_{\rm PCA}$ of $\cJ$ with the
following properties:
\begin{itemize}	
\item[(a)] In total $\mathcal{J}_{\rm CA}$ contains $L f r_3$ balanced exceptional systems. For each $i\le f$ and each $h \le L$,
$\mathcal{J}_{\rm CA}$ contains precisely $r_3$ $(i_1,i_2,i_3,i_4)$-BES of style $h$ (with respect to the
$(K,L,m_1,\eps_0)$-partition $(\cP_1,\cP'_1)$) such that $i_1,i_2,i_3,i_4\in \{(i-1)K_1/f+2,\dots,iK_1/f\}$.
\item[(b)] In total $\mathcal{J}_{\rm PCA}$ contains $7r^\diamond$ balanced exceptional systems. For each $i\le 7$,
$\mathcal{J}_{\rm PCA}$ contains precisely $r^\diamond$ $(i_1,i_2,i_3,i_4)$-BES (with respect to
the partition $\cP_1$) with
$i_1,i_2,i_3,i_4\in \{(i-1)K_1/7+2,\dots,iK_1/7\}$.
\end{itemize}
(Recall that we defined in Section~\ref{sec:findBF} when an $(i_1,i_2,i_3,i_4)$-BES has style $h$ with respect to a
$(K,L,m_1,\eps_0)$-partition $(\cP_1,\cP'_1)$.)
To see that it is possible to choose $\mathcal{J}_{\rm CA}$ and $\mathcal{J}_{\rm PCA}$, split $\cJ$ into two sets $\cJ_1$ and $\cJ_2$ such that both $\cJ_1$ and $\cJ_2$
contain at least $t_{K_1L}/3$ $(i'_1,i'_2,i'_3,i'_4)$-BES with respect to~$\cP'_1$, for all $1 \le i'_1,i'_2,i'_3,i'_4  \le K_1L$.
Note that there are $(K_1/f-1)^4$ choices of 4-tuples $(i_1,i_2,i_3,i_4)$
with $i_1,i_2,i_3,i_4\in \{(i-1)K_1/f+2,\dots,iK_1/f\}$. Moreover, for each such 4-tuple $(i_1,i_2,i_3,i_4)$
and each $h\le L$ there is one 4-tuple $(i'_1,i'_2,i'_3,i'_4)$ with $1 \le i'_1,i'_2,i'_3,i'_4  \le K_1L$
and such that any $(i'_1,i'_2,i'_3,i'_4)$-BES with respect to~$\cP'_1$ is an $(i_1,i_2,i_3,i_4)$-BES of style $h$ with respect to
$(\cP_1,\cP'_1)$. Together with the fact that
$$ 
\frac{(K_1/f-1)^4t_{K_1L}}{3}\ge \frac{D_1}{7(Lf)^4}\ge \gamma^{1/2}n\stackrel{(\ref{eq:rs})}{\ge} r_3,
$$
this implies that we can choose a set $\mathcal{J}_{\rm CA}\subseteq \cJ_1$ satisfying~(a).

Similarly, there are $(K_1/7-1)^4$ choices of 4-tuples $(i_1,i_2,i_3,i_4)$
with $i_1,i_2,i_3,i_4\in \{(i-1)K_1/7+2,\dots,iK_1/7\}$. Moreover, for each such 4-tuple $(i_1,i_2,i_3,i_4)$
there are $L^4$ distinct 4-tuples $(i'_1,i'_2,i'_3,i'_4)$ with $1 \le i'_1,i'_2,i'_3,i'_4  \le K_1L$
and such that any $(i'_1,i'_2,i'_3,i'_4)$-BES with respect to~$\cP'_1$ is an $(i_1,i_2,i_3,i_4)$-BES with respect to
$\cP_1$. Together with the fact that
$$ \frac{(K_1/7-1)^4L^4t_{K_1L}}{3}\ge \frac{D_1}{7^5}\ge \frac{n}{3\cdot 7^5}\stackrel{(\ref{eq:rs})}{\ge} r^\diamond,$$
this implies that we can choose a set $\mathcal{J}_{\rm PCA}\subseteq \cJ_2$ satisfying~(b).

\smallskip

\noindent\textbf{Step 4: Finding the robustly decomposable graph.}
Recall that $(G_2, \mathcal{P}_1,\cP'_1)$ is a $(K_1, L, m_1, \epszero,\eps')$-scheme.
Apply Proposition~\ref{lem:dirscheme} with $G_2$, $\cP_1$, $\cP'_1$, $K_1$, $m_1$, $\eps'$ playing the roles of
$G$, $\cP$, $\cP'$, $K$, $m$, $\eps$ to obtain an orientation $G_{2,{\rm dir}}$ of $G_2$%
  \COMMENT{Daniela: had $G_2-V_0$, but now $G_2$ doesn't include $V_0$ anymore}
such that
$(G_{2,{\rm dir}}, \mathcal{P}_1,\cP'_1)$ is a $[K_1, L, m_1, \epszero,2\sqrt{\eps'}]$-scheme. 
Let $C=A_1B_1A_2\dots A_{K_1}B_{K_1}$ be a spanning cycle on the clusters in $\cP_1$.
%Let $\mathcal{I}_f$ be the canonical interval partition of $[K]$ into $f$ intervals of equal length.

Our next aim is to use Lemma~\ref{lma:EF-bipartite} in order to extend the balanced exceptional systems in $\mathcal{J}_{\rm CA}$ into $r_3$ edge-disjoint
balanced exceptional factors with parameters $(L,f)$ for $G_{2,{\rm dir}}$ (with respect to $C$, $\mathcal P'_1$).
For this, note that the condition on $\cJ_{CA}$ in Lemma~\ref{lma:EF-bipartite} with $r_3$ playing the role of $q$ is satisfied by~(a). 
Moreover, $Lr_3/m_1=2rK_1/m_1=2\gamma K_1\ll 1$. 
Thus we can indeed apply Lemma~\ref{lma:EF-bipartite} to $(G_{2,{\rm dir}}, \mathcal{P}_1,\cP'_1)$
with $\cJ_{CA}$, $2\sqrt{\eps'}$, $K_1$, $r_3$ playing the roles of
$\cJ$, $\eps$, $K$, $q$ in order to obtain $r_3$ edge-disjoint balanced exceptional factors $BF_1,\dots,BF_{r_3}$ with parameters $(L,f)$
for $G_{2,{\rm dir}}$ (with respect to $C$, $\mathcal P'_1$) such that together these balanced exceptional factors
cover all edges in $\bigcup \mathcal{J}_{\rm CA}$. Let $\mathcal{BF}_{\rm CA}:=BF_1+\dots+ BF_{r_3}$.

Note that $m_1/4g,m_1/L\in\mathbb{N}$ since $m_1=|A|/K_1$ and $|A|$ is divisible by $K_2$ and thus $m_1$ is divisible
by $4gL$ (since $K_2/4gLK_1\in\mathbb{N}$ by our assumption).
Furthermore, $4rK_1^2=4\gamma m_1K_1^2\le \gamma^{1/2}m_1\le m_1$.%
\COMMENT{osthus adapted calculation from error in statement of rob dec corollary}
Thus we can apply Corollary~\ref{rdeccor} to the $[K_1, L, m_1, \epszero,\eps'']$-scheme
$(G_{2,{\rm dir}}, \mathcal{P}_1,\cP'_1)$ with $K_1$, $\eps''$, $g$ playing the roles of $K$, $\eps$, $\ell'$ to obtain a spanning
subgraph $CA(r)$ of $G_2$ as described there. (Note that $G_2$ equals the graph $G'$ defined in Corollary~\ref{rdeccor}.)%
    \COMMENT{Daniela: had $G_2-V_0$ instead of $G_2$ (twice)}
In particular, $CA(r)$ is $2(r_1+r_2)$-regular and edge-disjoint from $\mathcal{BF}_{\rm CA}$.

Let $G_3$ be the graph obtained from $G_2$ by deleting all the edges of $CA(r)+ \mathcal{BF}_{\rm CA}$.
Thus $G_3$ is obtained from $G_2$ by deleting at most $2(r_1+r_2+r_3)\le 6r_1=6\gamma_1m_1$ edges at every vertex in $A\cup B=V(G_3)$.
Let $G_{3,{\rm dir}}$ be the orientation of $G_3$%
    \COMMENT{Daniela: had $G_3-V_0$ instead of $G_3$}
in which every edge is oriented in the same way as in $G_{2,{\rm dir}}$.
Then Proposition~\ref{superslice} implies that%
   \COMMENT{Daniela added ref to Proposition~\ref{superslice}}
$(G_{3,{\rm dir}},\cP_1,\cP_1)$ is still a $[K_1, 1,m_1, \epszero,\eps]$-scheme.%
\COMMENT{AL: changed to as $[K_1, m_1, \epszero,\eps]$-scheme is no longer defined.}
Moreover,
$$\frac{r^\diamond}{m_1}\stackrel{(\ref{eq:rs})}{\le} \frac{2r_1}{m_1}=2\gamma_1\ll 1.$$
Together with~(b) this ensures that we can apply
Lemma~\ref{lma:EF-bipartite} to $(G_{3,{\rm dir}},\cP_1)$ with $\cP_1$, $\cJ_{PCA}$, $K_1$, $1$, $7$, $r^\diamond$ playing the roles of
$\cP$, $\cJ$, $K$, $L$, $f$, $q$ in order to obtain $r^\diamond$ edge-disjoint balanced exceptional factors $BF'_1,\dots,BF'_{r^\diamond}$ with parameters $(1,7)$
for $G_{3,{\rm dir}}$ (with respect to $C$, $\mathcal P_1$) such that together these balanced exceptional factors
cover all edges in $\bigcup \mathcal{J}_{\rm PCA}$. Let $\mathcal{BF}_{\rm PCA}:=BF'_1+\dots + BF'_{r^\diamond}$.

Apply Corollary~\ref{rdeccor} to obtain a spanning
subgraph $PCA(r)$ of $G_2$ as described there.%
   \COMMENT{Daniela: had $G_2-V_0$ instead of $G_2$}
In particular, $PCA(r)$ is $10r^\diamond$-regular
and edge-disjoint from $CA(r)+ \mathcal{BF}_{\rm CA}+ \mathcal{BF}_{\rm PCA}$.

Let $G^{\rm rob}:=CA(r)+ PCA(r)+ \mathcal{BF}_{\rm CA}+ \mathcal{BF}_{\rm PCA}$.
Note that by~(\ref{EFdeg}) all the vertices in $V_0:=A_0\cup B_0$ have the same degree  
$r_0^{\rm rob}:=2(Lfr_3+7r^\diamond)$ in $G^{\rm rob}$.
So 
\begin{equation}\label{eq:rrob}
7r_1\stackrel{(\ref{eq:rs})}{\le} r_0^{\rm rob} \stackrel{(\ref{eq:rs})}{\le} 30r_1.
\end{equation}
Moreover,~(\ref{EFdeg}) also implies that all the vertices in $A\cup B$ have the same degree $r^{\rm rob}$ in $G^{\rm rob}$,
where $r^{\rm rob}=2(r_1+r_2+r_3+6r^\diamond)$. So 
$$
r_0^{\rm rob}-r^{\rm rob}=2 \left(Lfr_3+ r^\diamond- (r_1+r_2+r_3)\right)=2(Lfr_3+r-(Lf-1)r_3-r_3 )=2r.
$$

\noindent\textbf{Step 5: Choosing a $(K_2,m_2,\eps_0)$-partition $\cP_2$.}
We now prepare the ground for the approximate decomposition step (i.e.~to apply Lemma~\ref{almostthm}). 
For this, we need to work with a finer partition of $A \cup B$ than the previous one
(this will ensure that the leftover from the approximate decomposition step is sufficiently sparse compared to $G^{\rm rob}$).

Let $G_4:=G_1-G^{\rm rob}$ (where $G_1$ was defined in Step~1) and note that 
\begin{equation} \label{D4D1}
D_4= D_1-r_0^{\rm rob}=D_1-r^{\rm rob}-2r.
\end{equation}
So 
\begin{equation} \label{degrees4}
d_{G_4}(x)=D_4+2r \mbox{ for all } x \in A \cup B \qquad \mbox{and} \qquad d_{G_4}(x)=D_4 \mbox{ for all } x \in V_0.
\end{equation}
(Note%
   \COMMENT{Daniela deleted "Thus every vertex in $V_0$
has degree $D_4$ in $G_4$ while every vertex in $A\cup B$ has degree $D_4+2r$." which we had immediately after (\ref{degrees4}), since
this is what (\ref{degrees4}) says}
that $D_4$ is even since $D_1$ and $r_0 ^{\rm rob}$ are even.)
So $G_4$ is $D_4$-balanced with respect to $(A,A_0,B,B_0)$
by Proposition~\ref{edge_number}.
Together with the fact that $(G_1,A,A_0,B,B_0)$ is an $(\eps_*,\eps_0,K_2,D_1)$-framework, this implies that
$(G_4,G_4,A,A_0,B,B_0)$ satisfies conditions (WF1)--(WF5) in the definition of an $(\eps_*,\eps_0,K_2,D_4)$-weak framework.
However, some vertices in $A_0\cup B_0$ might violate condition~(WF6). 
(But every vertex in $A\cup B$ will still satisfy~(WF6) with room to spare.) 
So we need to modify the partition of $V_0=A_0\cup B_0$ to obtain a new weak framework.

Consider a partition $A^*_0,B^*_0$ of $A_0\cup B_0$ which maximizes the number of edges in $G_4$ between
$A^*_0\cup A$ and $B^*_0\cup B$. 
Then $d_{G_4}(v, A^*_0\cup A) \le d_{G_4}(v)/2$ for all $v\in A^*_0$ since otherwise $A^*_0\setminus \{v\},B^*_0\cup \{v\}$ would
be a better partition of $A_0\cup B_0$. Similarly $d_{G_4}(v, B^*_0\cup B) \le d_{G_4}(v)/2$ for all $v\in B^*_0$.
 Thus (WF6) holds in $G_4$ (with respect to the partition
$A\cup A^*_0$ and $B\cup B^*_0$).
Moreover, Proposition~\ref{keepbalance} implies that $G_4$ is still $D_4$-balanced with respect to $(A,A^*_0,B,B^*_0)$.
Furthermore, with (FR3) and (FR4) applied to $G_1$, we obtain
$e_{G_4}(A\cup A^*_0)\le e_{G_1}(A\cup A_0)+|A^*_0||A\cup A^* _0|\le \eps_0 n^2$ and similarly $e_{G_4}(B\cup B^*_0)\le \eps_0 n^2$.
Finally, every vertex in $A\cup B$ has internal degree at most $\eps_0 n+|A_0 \cup B_0|\le 2\eps_0 n$ in $G_4$ (with respect to the partition
$A\cup A^*_0$ and $B\cup B^*_0$).%
\COMMENT{osthus replaced V0 with A0 B0} 
Altogether this implies that
$(G_4,G_4,A,A_0^*,B,B_0^*)$ is an $(\eps_0,2\eps_0,K_2,D_4)$-weak framework and thus also an $(\eps_0,\eps',K_2,D_4)$-weak framework.     

Without loss of generality we may assume that $|A^*_0|\geq |B^*_0|$. Apply Lemma~\ref{coverA0B02}
to the $(\eps_0,\eps',K_2,D_4)$-weak framework $(G_4,G_4,A,A_0^*,B,B_0^*)$ to find a set
$\mathcal C_2$ of $|\mathcal C_2| \leq \eps _0 n$ edge-disjoint Hamilton cycles in $G_4$ so that the graph $G_5$
obtained from $G_4$ by deleting all the edges of these Hamilton cycles forms part of an
$(\eps_0,\eps',K_2,D_5)$-framework $(G_5,A,A_0^*,B,B_0^*)$, where
\begin{align}\label{eqD5}
D_5=D_4-2|\mathcal C_2| \geq D_4 -2\eps _0 n.
\end{align}
Since $D_4$ is even, $D_5$ is even. Further, 
\begin{equation} \label{degrees5}
d_{G_5}(x)\stackrel{(\ref{degrees4})}{=} D_5+2r \mbox{ for all } x \in A \cup B \qquad \mbox{and} \qquad d_{G_5}(x)\stackrel{(\ref{degrees4})}{=}D_5 \mbox{ for all } x \in A^*_0\cup B^*_0.
\end{equation}

Choose an additional constant $\eps '_4$ such that $\eps _3 \ll \eps '_4 \ll 1/K_2$ and so
that%
\COMMENT{Vital: we have had to wait until now to define $\eps '_4$, since it depends on $D_5$ and
$D_5$ only just defined. In particular, $D_5$ depends on $|\mathcal C_2|$.}
$$t_{K_2}:=\frac{(1-20\eps'_4)D_5}{2K_2^4} \in \mathbb N.$$

Now apply Lemma~\ref{part} to $(G_5,A,A_0^*,B,B_0^*)$ with $D_5$, $K_2$, $\eps_0$
playing the roles of $D$, $K$, $\eps$ in order to obtain partitions $A_1,\dots,A_{K_2}$ and
$B_1,\dots,B_{K_2}$ of $A$ and $B$ into sets of size 
\begin{equation} \label{m2def}
m_2:=|A|/K_2
\end{equation} such that together with $A^*_0$ and $B^*_0$ the sets
$A_i$ and $B_i$ form a $(K_2,m_2,\eps_0,\eps_1,\eps_2)$-partition $\cP_2$ for $G_5$.
(Note that the previous partition of $A$ and $B$  plays no role in the subsequent argument, so
denoting the clusters in~$\cP_2$ by $A_i$ and $B_i$ again will cause no notational conflicts.)%
   \COMMENT{Daniela deleted an entire paragraph verifying that $G'_5:=G_5[A,B]$ is part of a scheme since this is not used anymore}
%Let $G'_5$ be the spanning subgraph of $G_5$ obtained by keeping all $AB$-edges and deleting all other edges.
%Then $(G'_5, \mathcal{P}_2)$ is a $(K_2, m_2, \epszero,\eps)$-scheme.
%Indeed, clearly (Sch1) follows from (P1) and~(Sch2) also holds. In order to verify (Sch3), note that by (P2) and (FR5),
%for all $x \in B$, we have
%\begin{align*}
%d_{G'_5}(x,A_j) & = d_{G_5}(x,A_j)=\frac{d_{G_5}(x,A)\pm \eps_1n}{K_2}
%\stackrel{(\ref{degrees5})}{\ge}\frac{D_5+2r-\eps 'n -|A^*_0|- \eps_1n}{K_2} \\ 
%& \stackrel{(\ref{eqD5})}{\ge} \frac{D_4-2\eps_1 n}{K_2} \stackrel{(\ref{D4bound})}{\ge}  \frac{(1 -6 \gamma_1 )|A|}{K_2} 
% \stackrel{(\ref{m2def})}{=} (1-6\gamma_1)m_2. 
%\end{align*}
%Similarly, $d_{G'_5}(y,B_i)\ge (1-6\gamma_1)m_2$ for all $y\in A$. Since $\gamma_1\ll \eps$, this implies~(Sch3).

\smallskip

\noindent\textbf{Step 6: Balanced exceptional systems for the approximate decomposition.}
In order to apply Lemma~\ref{almostthm}, we first need to construct suitable balanced exceptional systems.
Apply Corollary~\ref{BEScor} to the $(\eps_0,\eps',K_2,D_5)$-framework $(G_5,A,A_0^*,B,B_0^*)$ with
 $G_5$, $\eps_0$, $\eps'_0$, $\eps'_4$, $K_2$, $D_5$, $\cP_2$ playing the roles
of $F$, $\eps$, $\eps_0$, $\eps_4$, $K$, $D$, $\cP$. 
(Note that since we are letting $G_5$ play the role of $F$, condition (WF5) in the corollary immediately follows from (FR5).)
Let $\cJ'$ be the union of the sets $\cJ_{i_1i_2i_3i_4}$ guaranteed by Corollary~\ref{BEScor}.
So $\cJ'$ consists of $K_2^4t_{K_2}$ edge-disjoint balanced exceptional systems with parameter $\eps'_0$ in $G_5$
(with respect to $\cP_2$).%
   \COMMENT{Daniela added brackets}
Let $\cC_3$ denote the set of Hamilton cycles guaranteed by Corollary~\ref{BEScor}. So $|\cC_3|=10\eps'_4D_5$.%
   \COMMENT{Daniela added new sentence}

Let $G_6$ be the subgraph obtained from $G_5$ by deleting all those edges lying in the Hamilton cycles from~$\cC_3$.
Set $D_6:=D_5-2|\cC_3|$.
So
\begin{equation} \label{degrees6}
d_{G_6}(x)\stackrel{(\ref{degrees5})}{=}D_6+2r \mbox{ for all } x \in A \cup B \qquad \mbox{and} \qquad d_{G_6}(x)\stackrel{(\ref{degrees5})}{=}D_6 \mbox{ for all } x \in V_0.
\end{equation}
(Note that $V_0=A_0 \cup B_0=A_0^* \cup B_0^*.$)%
\COMMENT{osthus added bracket}
Let $G'_6$ denote the subgraph of $G_6$ obtained by deleting all those
edges lying in the balanced exceptional systems from $\cJ'$. Thus $G'_6=G^\diamond$, where $G^\diamond$ is as defined in Corollary~\ref{BEScor}(iv).
In particular, $V_0$ is an isolated set in $G'_6$ and $G'_6$ is bipartite with vertex classes $A\cup A^*_0$ and $B\cup B^*_0$
(and thus also bipartite with vertex classes $A'=A\cup A_0$ and $B'=B\cup B_0$).%
    \COMMENT{Daniela deleted the following paragraph (as well as the next one saying that $G'_6[A,B]$ is still part of a scheme): 
"Moreover, recall that (as remarked after the corresponding definition in Section~\ref{besconstruct}), a balanced exceptional
system does not contain any $AB$-edges and thus the same holds for $\bigcup \cJ'$ too.
Thus $G'_6$ can be obtained from $G_6$ by keeping all $AB$-edges and deleting all other edges."}

%Furthermore, $G'_6$ can be obtained from $G'_5$ by deleting all those edges lying in the Hamilton cycles in $\cC_3$.
%(Again, this holds since $\bigcup \cJ'$ does not contain any $AB$-edges.)
%But $|\cC_3|= 10\eps'_4D_5$ by Corollary~\ref{BEScor} and we have already checked that
%$(G'_5, \mathcal{P}_2)$ is a $(K_2, m_2, \epszero,\eps)$-scheme.
%(Note that the choice of $\eps'_4$ implies that $10\eps'_4D_5 \in \mathbb{N}$.)
%Thus $(G'_6, \mathcal{P}_2)$ is still a $(K_2, m_2, \epszero,\eps^{1/2})$-scheme.

Consider any vertex $v\in V_0$. Then $v$ has degree $D_5$ in $G_5$, degree two in each Hamilton cycle from $\cC_3$,
degree two in each balanced exceptional system from $\cJ'$ and degree zero in~$G'_6$.
Thus
$$D_6+2|\cC_3|=D_5\stackrel{(\ref{degrees5})}{=}d_{G_5}(v)=2|\cC_3|+2|\cJ'|+d_{G'_6}(v)=2|\cC_3|+2|\cJ'|$$
and so
\begin{equation}\label{eq:D5}
D_6=2|\cJ'|.
\end{equation}

\noindent\textbf{Step 7: Approximate Hamilton cycle decomposition.}
Our next aim is to apply Lemma~\ref{almostthm} with $G_6$, $\cP_2$, $K_2$, $m_2$, $\cJ'$, $\eps'$ playing the
roles of $G$, $\cP$, $K$, $m$, $\cJ$, $\eps_0$.
 Clearly, condition~(c) of Lemma~\ref{almostthm} is satisfied.
In order to see that condition (a) is satisfied, let $\mu:=(r^{\rm rob}_0-2r)/4K_2m_2$ and note that%
   \COMMENT{Daniela changed inequality below, before had $1/K_2\ll \mu\ll \eps$, but we don't need this for Lemma~\ref{almostthm}}
$$
0\le \frac{\gamma_1 m_1}{4K_2m_2}\leq
\frac{7r_1-2r}{4K_2m_2}\stackrel{(\ref{eq:rrob})}{\le} \mu \stackrel{(\ref{eq:rrob})}{\le} \frac{30r_1}{4K_2m_2}\le
\frac{30\gamma_1 }{K_1}\ll 1.
$$
Recall that every vertex $v\in B$ satisfies%
\COMMENT{osthus adapted calculation to be an `equality' and not just a lower bound}
$$d_{G_5}(v) \stackrel{(\ref{degrees5})}{=} D_5+2r \stackrel{(\ref{D4D1}),(\ref{eqD5})}{=} D_1-r^{\rm rob}_0+2r \pm 2 \eps_0 n
\stackrel{(\ref{D1eq})}{=}  |A| -r^{\rm rob}_0+2r \pm 4 \eps_0 n.
$$
Moreover,
$$d_{G_5}(v,A)= d_{G_5}(v)-d_{G_5}(v,B\cup B^*_0)-|A^*_0|\ge d_{G_5}(v)-2\eps ' n,$$
where the last inequality holds since $(G_5,A,A^*_0,B,B^*_0)$ is an $(\eps_0,\eps ',K_2,D_5)$-framework (c.f.~conditions (FR4) and (FR5)).
Together with the fact that $\cP_2$ is a $(K_2,m_2,\eps_0,\eps_1,\eps_2)$-partition for $G_5$ (c.f. condition (P2)),
this implies that
\begin{align*}
d_{G_5}(v,A_i) & =\frac{d_{G_5}(v,A)\pm \eps_1 n}{K_2}
=\frac{|A|-r^{\rm rob}_0+2r\pm 2\eps_1 n}{K_2}=
\left(1-\frac{r^{\rm rob}_0-2r}{K_2m_2}\pm 5\eps_1\right)m_2\\
& =(1-4\mu\pm 5\eps_1)m_2=(1-4\mu\pm 1/K_2)m_2.
\end{align*}
Recall that $G_6$ is obtained from $G_5$ by deleting all those edges lying in the Hamilton cycles in $\cC_3$ and that
$$
|\cC_3|= 10\eps'_4 D_5 \leq 10\eps'_4 D_4
 \stackrel{(\ref{D4bound})}{\le} 11 \eps'_4 |A| \stackrel{(\ref{m2def})}{\leq} m_2/K_2.
$$
Altogether this implies that
$d_{G_6}(v,A_i)=(1-4\mu\pm 4/K_2)m_2$. Similarly one can show that $d_{G_6}(w,B_j)=(1-4\mu\pm 4/K_2)m_2$ for all $w\in A$.
So condition~(a) of Lemma~\ref{almostthm} holds.

To check condition~(b), note that%
   \COMMENT{Daniela added $\frac{D_4}{2}$ to the display below}
$$|\cJ'|\stackrel{(\ref{eq:D5})}{=}  \frac{D_6}{2}\le  \frac{D_4}{2}
\stackrel{(\ref{D4D1})}{\le}  \frac{D_1-r^{\rm rob}_0}{2}\le \frac{n}{4}-\mu \cdot 2K_2m_2-r\le
\left(\frac{1}{4}-\mu-\frac{\gamma}{3 K_1}\right)n.$$
Thus condition~(b) of Lemma~\ref{almostthm} holds with $\gamma/3 K_1$ playing the role of $\rho$.

Since the edges in $\cJ'$ lie in $G_5$ and  $(G_5,A,A^*_0,B,B^*_0)$ is an $(\eps_0,\eps',K_2,D_5)$-framework,
(FR5) implies that each $v \in A \cup B$ is incident with an edge in $J$ for at most $\eps ' n+|V_0|\leq 2\eps 'n$ of the%
\COMMENT{osthus added `of the} 
$J \in \cJ '$. 
(Recall that in a balanced exceptional system there are no edges between $A$ and $B$.)
So condition~(d) of Lemma~\ref{almostthm} holds with $\eps '$ playing the role of $\eps _0$.

So we can indeed apply Lemma~\ref{almostthm} to obtain a collection $\cC_4$ of $|\cJ'|$ edge-disjoint Hamilton cycles in $G_6$
which cover all edges of $\bigcup \cJ'$.

\smallskip

\noindent\textbf{Step 8: Decomposing the leftover and the robustly decomposable graph.}
Finally, we can apply the `robust decomposition property' of $G^{\rm rob}$ guaranteed by Corollary~\ref{rdeccor}
to obtain a Hamilton decomposition of the leftover from the previous step together with $G^{\rm rob}$.
 
To achieve this, let $H'$ denote the subgraph of $G_6$ obtained by deleting all those edges lying in the Hamilton cycles from~$\mathcal{C}_4$.
Thus~(\ref{degrees6}) and~(\ref{eq:D5}) imply that
every vertex in $V_0$ is isolated in $H'$ while every vertex $v\in A\cup B$ has degree $d_{G_6}(v)-2|\cJ'|=D_6+2r-2|\cJ'|=2r$ in~$H'$
(the last equality follows from~(\ref{eq:D5})).
Moreover, $H'[A]$ and $H'[B]$ contain no edges. (This holds
since $H'$ is a spanning subgraph of $G_6-\bigcup \cJ'=G'_6$ and since we have
already seen that $G'_6$ is bipartite with vertex classes $A'$ and $B'$.) Now let $H:=H'[A ,B]$. Then 
Corollary~\ref{rdeccor}(ii)(b) implies that $H + G^{\rm rob}$
has a Hamilton decomposition. Let $\cC_5$ denote the set of Hamilton cycles thus obtained.
Note that $H+G^{\rm rob}$ is a spanning subgraph of $G$ which contains all edges of $G$ which were not covered by $\cC_1\cup \cC_2\cup \cC_3\cup \cC_4$.
So $\cC_1\cup \cC_2\cup \cC_3\cup \cC_4 \cup \cC_5$ is a Hamilton decomposition of $G$.
\endproof

\medskip
{\footnotesize \obeylines \parindent=0pt
B\'ela Csaba
Bolyai Institute,
University of Szeged,
H-6720 Szeged, Aradi v\'ertan\'uk tere 1.
Hungary

\begin{flushleft}
{\it{E-mail address}: \tt{bcsaba@math.u-szeged.hu}}
\end{flushleft}

Daniela K\"{u}hn, Allan Lo, Deryk Osthus 
School of Mathematics
University of Birmingham
Edgbaston
Birmingham 
B15 2TT
UK

\begin{flushleft}
{\it{E-mail addresses}: \tt{\{d.kuhn, s.a.lo, d.osthus\}@bham.ac.uk}}
\end{flushleft}

Andrew Treglown
School of Mathematical Sciences
Queen Mary, University of London
Mile End Road
London 
E1 4NS
UK

\begin{flushleft}
{\it{E-mail addresses}: \tt{a.treglown}@qmul.ac.uk}
\end{flushleft}
}

\end{document}